\documentclass[a4paper,11pt]{amsart}
 
 \usepackage[OT2,T1]{fontenc}
 \usepackage{eucal,fullpage,times,amsmath,amsthm,amssymb,mathrsfs,stmaryrd,color,enumerate,accents}
 \usepackage[all]{xy}
 \usepackage{url}
 \usepackage[leqno]{amsmath}
 \usepackage{amsthm}
 \usepackage{amsfonts}
 \usepackage{amssymb}
 \usepackage{eucal}
 \usepackage[all]{xy}
 \CompileMatrices
 
 \usepackage{mathrsfs}

 \usepackage{amsmath,amssymb,mathrsfs,xy,amsthm,bm,url}
 \input{xy}
 \xyoption{all}
 
 \usepackage{xspace}
  \usepackage[ansinew]{inputenc}
  \usepackage[T1]{fontenc}
  \usepackage{hyperref}
 
 \usepackage{hyperref}

 \theoremstyle{definition}
   \newtheorem{definition}[subsection]{Definition}
   \newtheorem{notation}[subsection]{Notation}
   \newtheorem{definition-proposition}[subsection]{Definition-Proposition}
   \newtheorem{example}[subsection]{Example}
   \newtheorem{remark}[subsection]{Remark}

    \newtheorem{assumption}[subsection]{Assumption}

   \newtheorem{convention}[subsection]{Convention}

 \theoremstyle{theorem}
   \newtheorem{theorem}[subsection]{Theorem}
   \newtheorem{proposition}[subsection]{Proposition}
   \newtheorem*{proposition*}{Proposition}
   \newtheorem{lemma}[subsection]{Lemma}
   \newtheorem{corollary}[subsection]{Corollary}
   \newtheorem{conjecture}[subsection]{Conjecture}

 \newcommand{\maxx}{{\mathrm{max}}}

 \newcommand{\Abb}{\mathbb{A}}
 \newcommand{\Cbb}{\mathbb{C}}
 
 \newcommand{\Gbb}{\mathbb{G}}
 \newcommand{\Nbb}{\mathbb{N}}
 \newcommand{\Qbb}{\mathbb{Q}}
 \newcommand{\Rbb}{\mathbb{R}}
 \newcommand{\Sbb}{\mathbb{S}}
 \newcommand{\Zbb}{\mathbb{Z}}

 \newcommand{\Cbf}{\mathbf{C}}
 \newcommand{\Hbf}{\mathbf{H}}
 \newcommand{\Gbf}{\mathbf{G}}
 \newcommand{\ibf}{\mathbf{i}}
 \newcommand{\Lbf}{\mathbf{L}}
 
 \newcommand{\Nbf}{\mathbf{N}}
 \newcommand{\Pbf}{\mathbf{P}}
 \newcommand{\Qbf}{\mathbf{Q}}
 \newcommand{\Tbf}{\mathbf{T}}
 \newcommand{\Ubf}{\mathbf{U}}
 \newcommand{\Vbf}{\mathbf{V}}
 \newcommand{\Wbf}{\mathbf{W}}
 \newcommand{\Xbf}{\mathbf{X}}
 \newcommand{\Zbf}{\mathbf{Z}}
 \newcommand{\GLbf}{\mathbf{GL}}

 \newcommand{\gfrak}{\mathfrak{g}}

 \newcommand{\pfrak}{\mathfrak{p}}
 \newcommand{\Xfrak}{\mathfrak{X}}
 \newcommand{\Yfrak}{\mathfrak{Y}}

 \newcommand{\arm}{\mathrm{a}}
 \newcommand{\mrm}{\mathrm{m}}

 \newcommand{\Ascr}{{\mathscr{A}}}
 \newcommand{\Mscr}{{\mathscr{M}}}
 \newcommand{\Yscr}{{\mathscr{Y}}}
 \newcommand{\Hscr}{\mathscr{H}}
 
 \newcommand{\Pscr}{\mathscr{P}}
 \newcommand{\Sscr}{\mathscr{S}}
 \newcommand{\Lscr}{\mathscr{L}}

 \newcommand{\Rcal}{\mathcal{R}}

 \newcommand{\ad}{\mathrm{ad}}
 \newcommand{\Ad}{\mathrm{Ad}}
 \newcommand{\ab}{\mathrm{ab}}
 \newcommand{\der}{\mathrm{der}}
 
 \newcommand{\ra}{\rightarrow}
 \newcommand{\lra}{\longrightarrow}

 \newcommand{\mono}{\hookrightarrow}
 \newcommand{\epim}{{\twoheadrightarrow}}	
 \newcommand{\isom}{\cong}
 \newcommand{\limproj}{\varprojlim}
 
 \newcommand{\Hom}{\mathrm{Hom}}
 
 \newcommand{\GSp}{\mathrm{GSp}}
 \newcommand{\Sp}{\mathrm{Sp}}
 \newcommand{\Ker}{\mathrm{Ker}}
 \newcommand{\Lie}{\mathrm{Lie}}
 \newcommand{\Nm}{\mathrm{Nm}}
 \newcommand{\Res}{\mathrm{Res}}
 
 \newcommand{\Int}{\mathrm{Int}}
 
 \newcommand{\GL}{{\mathbf{GL}}}
 
 \newcommand{\bsh}{\backslash}
 \newcommand{\inv}{{-1}}
 \newcommand{\ot}{\overset}

 \newcommand{\wrt}{{with\ respect\ to}\xspace}
 \newcommand{\ifof}{{if\ and\ only\ if}\xspace}
 \newcommand{\cosg}{{compact\ open\ subgroup}\xspace}
 \newcommand{\cosgs}{{compact\ open\ subgroups}\xspace}
 
 \newcommand{\Qac}{\bar{\mathbb{Q}}}
 \newcommand{\Gal}{\mathrm{Gal}}
 \newcommand{\Gr}{\mathrm{Gr}}

 \newcommand{\ubar}{{\bar{u}}}

 \newcommand{\supp}{\mathrm{supp}}
 \newcommand{\pibar}{\overline{\pi}_\circ}
 
 \newcommand{\mult}{{\mathbb{G}_\mathrm{m}}}
 \newcommand{\adele}{{\hat{\mathbb{Q}}}}

 \newcommand{\Qbbp}{{\mathbb{Q}_p}}
 \newcommand{\Zbbp}{{\mathbb{Z}_p}}
 \newcommand{\Zbbptimes}{{\mathbb{Z}^\times_p}}

 \newcommand{\TW}{{{(\mathbf{T},w)}}}

 \newcommand{\ord}{\mathrm{ord}}
 
 \newcommand{\rec}{{\mathrm{rec}}}
 \newcommand{\Aut}{{\mathrm{Aut}}}

 \newcommand{\Gaa}{{\mathbb{G}_\mathrm{a}}}
 
 \newcommand{\Uscr}{\mathscr{U}}
 \newcommand{\Vscr}{\mathscr{V}}

 \newcommand{\MT}{\mathrm{MT}}
 \newcommand{\Vbb}{{\mathbb{V}}}

 \title{bounded equidistribution of special subvarieties \\ in mixed Shimura varieties}
 
 \author{Ke Chen}
 \address{Key Laboratory of Wu Wen-Tsun Mathematics, Chinese Academy of Sciences, and School of Mathematical Sciences, University of Science and Technology of China\\ No. 96 Jinzhai Road, Hefei, Anhui Province, 230026, P. R. China}
 
 \email{kechen@ustc.edu.cn}
 
 
 
\begin{document}

 \begin{abstract}
 In this paper, we study the equidistribution of certain families of special subvarieties in general mixed Shimura varieties. We introduce the notion of bounded sequences of special subvarieties, and we prove that the Andr\'e-Oort conjecture holds for such sequences. The proof follows the equidistribution approach used by Clozel, Ullmo, and Yafaev in the pure case. We then propose the notion of test invariant of a special subvariety, which is adapted from the lower bound formula of degrees of special subvarieties in the pure case studied by Ullmo and Yafaev, and we show that sequences of special subvarieties with bounded test invariants are bounded, hence the Andr\'e-Oort conjecture holds in this case.
 \end{abstract}
 \maketitle
 \tableofcontents

 \section{Introduction to the main results}

 In this paper, we study the equidistribution of certain families of special subvarieties in a general mixed Shimura variety, and  the Andr\'e-Oort conjecture for these varieties, as a generalization of some results in \cite{clozel ullmo} and in \cite{ullmo yafaev}.
 
 The theory of mixed Shimura varieties, including their canonical models and toroidal compactifications, is developed in \cite{pink thesis}. They serve as moduli spaces of mixed Hodge structures, and often arise as boundary components in the toroidal compactifications of pure Shimura varieties. Among mixed Shimura varieties there are Kuga varieties, cf. \cite{chen kuga}, which are certain ``universal'' abelian schemes over Shimura varieties, and in general, a mixed Shimura variety can be realized as a torus bundle over a Kuga variety (namely a torsor whose structure group is a torus). Similar to the pure case, we have the notion of special subvarieties in mixed Shimura varieties.
 
 Y. Andr\'e and F. Oort conjectured that the Zariski closure of a sequence of special subvarieties in a given pure Shimura variety remains a finite union of special subvarieties, cf. \cite{andre note} and \cite{oort conjecture}. R. Pink has proposed a generalization of this conjecture for mixed Shimura varieties by combining it with the Manin-Mumford conjecture and the Mordell-Lang conjecture for abelian varieties, a principal case of which is the following
 
 \begin{conjecture}[Andr\'e-Oort conjecture for mixed Shimura varieties, \cite{pink conjecture}]
 
 Let $M$ be a mixed Shimura variety, and let $(M_n)_n$ be a sequence of special subvarieties. Then the Zariski closure of $\bigcup_nM_n$ is a finite union of special subvarieties.
 
 \end{conjecture}
 
 Remarkable progress has been made for the Andr\'e-Oort conjecture in the pure case, including  the ergodic-Galois approach, cf.  \cite{clozel ullmo}, \cite{klingler yafaev}, \cite{ullmo yafaev}, \cite{yafaev bordeaux}, the $p$-adic approach cf. \cite{moonen compositio}, \cite{yafaev compositio}, and the model-theoretic approach, cf. \cite{pila annals},  \cite{scanlon bourbaki}. For the case of mixed Shimura varieties, Pila's work in \cite{pila annals} has already included  products of universal families of elliptic curves over modular curves as well as some torus bundles on them, and Scanlon has also considered in \cite{scanlon inventiones} some cases in the universal families of abelian varieties over the Siegel modular variety. Recently, Z. Gao has proved the conjecture for mixed Shimura varieties fibred over Siegel modular varieties. 
 
 The  strategy of Klingler-Ullmo-Yafaev is  summarized in \cite{yafaev bordeaux}. It assumes the GRH (Generalized Riemann Hypothesis) for CM fields, and does not involve model-theoretic tools. The main ingredients of the strategy in the pure case can be expressed as the following ``ergodic-Galois alternative'':
 
 \begin{itemize}
 
 \item equidistribution of special subvarieties with bounded Galois orbits (using ergodic theory), cf. \cite{clozel ullmo} and \cite{ullmo yafaev};
 
 \item for a sequence of special subvarieties $(M_n)$ of unbounded Galois orbits, one can construct a new sequence of special subvarieties  $(M_n')$ such that $\bigcup_nM_n$ has the same Zariski closure as $\bigcup_nM_n'$, and that $\dim M_n<\dim M_n'$ for $n$ large enough.
 
 \end{itemize} Note that both ingredients involve estimations that rely on the GRH  for CM fields.  
 
 In this paper, we study the equidistribution part of the ergodic-Galois alternative for mixed Shimura varieties. Our main result is the following:
 
 \begin{theorem} Let $M$ be a mixed Shimura variety of the form $M_K(\Pbf,Y)$ with $(\Pbf,Y)=\Wbf\rtimes(\Pbf,Y)$ and $K=K_\Wbf\rtimes K_\Gbf$ a \cosg of $\Pbf(\adele)$, and let $\pi:M\ra S=M_{K_\Gbf}(\Gbf,X)$ be the fibration over the pure Shimura variety $S$. Let $M_n$ be a sequence of special subvarieties in $M$. Write $E$ for the field of definition of $M$. Assume that the test invariants of $(M_n)$ are bounded, i.e. $$\tau_M(M_n)\leq C,\ \forall n$$ for some constant $C\in\Rbb_{>0}$. Then the Zariski closure of $\bigcup_n M_n$ is a finite union of special subvarieties.
 \end{theorem}

 Here the notion of test invariants is an analogue of the degree of Galois orbits against the automorphic line bundle in the pure case. The main theorem is deduced from a theorem of bounded equidistribution in certain spaces associated to mixed Shimura varieties, cf.\ref{equidistribution of TW-special subspaces}, \ref{bounded equidistribution}, \ref{bounded Andr\'e-Oort} and a characterization of special subvarieties with bounded Galois orbits using test invariants, cf. \ref{pure special subvarieties of bounded Galois orbits}, \ref{special subvarieties of bounded test invariants}.   
 
 We briefly explain the main idea of the paper. A mixed Shimura datum in the sense of \cite{pink thesis} is of the form $(\Pbf,\Ubf,Y)$ with $\Pbf$ a connected linear $\Qbb$-group, with a Levi decomposition $\Pbf=\Wbf\rtimes\Gbf$, $\Ubf$ a normal unipotent $\Qbb$-subgroup of $\Pbf$ central in $\Wbf$, and $Y$ a complex manifold homogeneous under $\Ubf(\Cbb)\Pbf(\Rbb)$ subject to some algebraic constraints. The notion of special subvarieties in mixed Shimura varieties is defined in a similar way as in the pure case. We often express it as an extension $(\Pbf,Y)=\Wbf\rtimes(\Gbf,X)$ with $(\Gbf,X)$ some pure Shimura datum, and $\Wbf$ the unipotent radical of $\Pbf$, in which $\Ubf$ is central. When $\Ubf$ is trivial, we get Kuga data and Kuga varieties, cf. \cite{chen kuga}.
 
 Parallel to the pure case studied in \cite{clozel ullmo} and \cite{ullmo yafaev}, we first consider the Andr\'e-Oort conjecture for sequences of $\TW$-special subvarieties in a mixed Shimura variety $M$ defined by $(\Pbf,Y)=\Wbf\rtimes(\Gbf,X)$. Here $\Tbf$ is a $\Qbb$-torus in $\Gbf$ and $w$ an element of $\Wbf(\Qbb)$. Using the language of \cite{chen kuga} 2.10 etc, in a Kuga variety $M=\Gamma_\Vbf\rtimes\Gamma_\Gbf\bsh Y^+$ defined by $(\Pbf,Y)=\Vbf\rtimes(\Gbf,X)$,  $\TW$-special subvarieties are defined by subdata of the form $\Vbf'\rtimes(w\Gbf'w^\inv,wX')$, and they are obtained  from diagrams of the form $$\xymatrix{ M'\ar[r]^\subset & M_{S'}\ar[r]^\subset \ar[d]^\pi & M \ar[d]^\pi\\ & S'\ar[r]^\subset & S}$$
 where \begin{itemize}
   \item $S=\Gamma_\Gbf\bsh X^+$ is a pure Shimura variety and $\pi:M\ra S$ is an abelian $S$-scheme defined by the natural projection $(\Pbf,Y)\ra(\Gbf,X)$;
 
   \item $S'$ is a (pure) special subvariety of $S$ defined by $(\Gbf',X')$ with $\Tbf$ equal to the connected center of $\Gbf'$,
 
   \item $M_{S'}$ is the abelian $S'$-scheme pulled-back from $M\ra S$, and $\Vbf'$ is a subrepresentation in $\Vbf$ of $\Gbf'$, corresponding to an abelian subscheme $A'$ of $M_{S'}\ra S'$;
 
   \item $M'$ is a translation of $A'$ by a torsion section of $M_{S'}\ra S'$ given by $w$.
 \end{itemize} In particular, this notion is more restrictive than $\Tbf$-special subvarieties studied in \cite{chen kuga} as we specify $w$.

 The case in general mixed Shimura varieties is similar. We show that certain spaces of probability measures on $M$ associated to $\TW$-special subvarieties are compact for the weak topology, from which we deduce the equidistribution of the supports of such measures, as well as the Andr\'e-Oort conjecture for such sequences of $\TW$-special subvarieties.
 
 We formulate the notion of $B$-bounded sequences of special subvarieties, which means special subvarieties that are $\TW$-special with $\TW$ coming from some prescribed finite set $B$ of pairs $\TW$ as above. The main result of \cite{ullmo yafaev} shows that  in the pure case a sequence with bounded Galois orbits is $B$-bounded for some $B$, where sequences with bounded Galois orbits are sequences of special subvarieties whose Galois orbits are of bounded degree against the automorphic line bundle.  In the mixed case, we propose the notion of test invariants for special subvarieties, and we prove a similar characterization of bounded sequences using test invariants.

 \bigskip

 The paper is organized as follows:
 
 In Section 1, we recall the basic notions of mixed Shimura data, their subdata, mixed Shimura varieties, their special subvarieties, as well as their connected components. We emphasize the fibration of a mixed Shimura variety over a pure Shimura variety, whose fibers are torus bundles over abelian varieties. We also include some results about irreducible subdata and a few reductions for the Andr\'e-Oort conjecture.

 Section 2 and 3 are concerned with ergodic-theoretic results in the equidistribution of special subvarieties. In Section 2, we  introduce some measure-theoretic objects, such as  lattice (sub)spaces,   S-(sub)spaces, and canonical probability measures associated to special subvarieties in mixed Shimura varieties. The lattice (sub)spaces are similar to the cases treated in \cite{clozel ullmo} and \cite{chen kuga}. For a Kuga variety, the associated S-space is the variety itself; for a general mixed Shimura variety, the S-space is a subspace of the variety, which is a torsor over the corresponding Kuga variety by some compact tori. In particular, they support canonically defined probability measures and they are Zariski dense in the ambient mixed Shimura varieties.  We also introduce the notion of a $B$-bounded sequence of special subvarieties, where $B$ is a finite set of pairs of the form $(\Tbf,w)$ as is explained above.
  
 In Section 3, we prove the equidistribution of bounded sequences of special lattice subspaces and special S-spaces. The proof is reduced to the case when the bound $B$ consists of a single element $\TW$, and the arguments are completely parallel to the pure $\Tbf$-special case in \cite{clozel ullmo} and \cite{ullmo yafaev}. The equidistribution of $B$-bounded S-subspaces implies the Andr\'e-Oort conjecture for a $B$-bounded sequence  of special subvarieties in a mixed Shimura variety.
 
 In the remaining sections we investigate the relation between bounded sequences and lower bounds of degrees of Galois orbits. In Section 4, we give a lower bound of the degrees of Galois orbits of a pure special subvariety $M'$ in a given mixed Shimura variety $M$ against the pull-back of the automorphic line bundle. The estimation is essentially reduced to the case studied in \cite{ullmo yafaev}. If the pure special subvariety under consideration is $\TW$-special, then the lower bound relies on the GRH for the splitting field $F_\Tbf$ of $\Tbf$, and it involves the discriminant of $F_\Tbf$ and the position of $w$ relative to the level structure of the ambient mixed Shimura variety. We show that the study of the Andr\'e-Oort conjecture can be reduced to ambient mixed Shimura data that are embedded in a ``good product'', in which the splitting fields $F_\Tbf$ of irreducible $\TW$-special subdata are CM fields. In this latter setting we estimate the contribution of unipotent translation in the lower bound. 
 
 In Section 5, we consider a general special subvariety which is not pure.  We did not prove an explicit lower bound formula in this case, instead we introduce the notion of test invariant as a substitute, and we show that a sequence with bounded test invariants is $B$-bounded for some finite $B$. We also show that in this case the Galois orbits are essentially minorated by the test invariants. The results in this section are formulated for ambient mixed Shimura embedded in ``good products'' as in Section 4.\bigskip

Recently we have been informed by GAO Ziyang on the work \cite{gao mixed} where he has proved the Andr\'e-Oort conjecture for mixed Shimura varieties whose pure parts are subvarieties of Siegel modular varieties assuming the GRH for CM fields. His approach is model-theoretic, generalizing the works of \cite{pila annals} etc. He has obtained independently some results related to the Galois orbits of special points. Our treatment works for special subvarieties of higher dimension, but relies on the GRH. We hope that the results presented are still useful as a step towards the ergodic-Galois alternative for mixed Shimura varieties. 
 
 \subsection*{Acknowledgement} The author thanks Prof.Emmanuel Ullmo heartily for suggesting to him the equidistribution approach towards the Andr\'e-Oort conjecture for mixed Shimura varieties, without whose guidance this work would not have been possible. He thanks Mr. GAO Ziyang for discussion on mixed Shimura varieties and his work \cite{gao mixed}. He also thanks Mr. Cyril D\'emarche and Mr. LIANG Yongqi for discussion on strong approximation of semi-simple groups. Finally, he thanks the anonymous referee sincerely for a very careful reading of the manuscript and many useful suggestions.
 
 The author is partially supported by the following grants: National Key Basic Research Program of China, No. 2013CB834202, Chinese Universities Scientific Fund Project WK0010000029, and National Natural Science Foundation of China,  Grant No. 11301495.

 \section*{Notations and conventions}\label{notations and conventions}
 
 Over a base ring $k$, a linear $k$-group $\Hbf$ is a smooth affine algebraic $k$-group scheme, and $\Tbf_\Hbf$ is the connected center of $\Hbf$, namely the neutral component of the center of $\Hbf$. For $\Vbf$ a free $k$-module of finite type, we have the general linear $k$-group $\GL_\Vbf$, and we also view $\Vbf$ as a vectorial $k$-group, i.e. isomorphic to $\Gbb_\arm^r$ with $r$ the rank of $\Vbf$.
 
 We write $\Sbb$ for the Deligne torus $\Res_{\Cbb/\Rbb}\mult_\Cbb$. The ring of finite adeles is denoted by $\adele$. $\ibf$ is a fixed square root of -1 in $\Cbb$.
 
 For a real or complex analytic space (not necessarily smooth), its analytic topology is the one locally deduced from the archimedean metric on $\Rbb^n$ or $\Cbb^m$.
 
 A linear $\Qbb$-group is compact if its set of $\Rbb$-points form a compact Lie group. For $\Pbf$ a linear $\Qbb$-group with maximal reductive quotient $\Pbf\epim\Gbf$, we write $\Pbf(\Rbb)^+$ resp. $\Pbf(\Rbb)_+$ for the preimage of $\Gbf(\Rbb)^+$ resp. of $\Gbf(\Rbb)_+$, in the sense of \cite{deligne pspm}. $\Pbf(\Rbb)^+$ is just the neutral component of the Lie group $\Pbf(\Rbb)$ because the fiber $\Wbf(\Rbb)$ of the projection $\Pbf(\Rbb)\ra\Gbf(\Rbb)$, namely the unipotent radical of $\Pbf(\Rbb)$, is a connected Lie group.
 
 For $\Hbf$ a linear $\Qbb$-group and $L$ a number field, we write $\Hbf^L$ for the $\Qbb$-group $\Res_{L/\Qbb}\Hbf_L$.

 For $\Hbf$ a linear $\Qbb$-group, we write $\Xfrak(\Hbf)$ for the set of $\Rbb$-group homomorphisms $\Sbb\ra\Hbf_\Rbb$, and $\Yfrak(\Hbf)$ for the set of $\Cbb$-group homomorphisms $\Sbb_\Cbb\ra\Hbf_\Cbb$. We have the natural action of $\Hbf(\Rbb)$ on $\Xfrak(\Hbf)$ by conjugation, and similarly the action of $\Hbf(\Cbb)$ on $\Yfrak(\Hbf)$. In particular, we have an inclusion $\Xfrak(\Hbf)\mono\Yfrak(\Hbf)$, equivariant \wrt the inclusion $\Hbf(\Rbb)\mono\Hbf(\Cbb)$.

 \section{Preliminaries on  mixed Shimura varieties}\label{Preliminaries on mixed Shimura varieties}
 
 We start with the definition of mixed Shimura data, which is ``essentially'' the same as  \cite{pink thesis}2.1, cf. \cite{chen kuga}2.1 :
 
 \begin{definition}[mixed Shimura data]\label{mixed Shimura data} (1) A \emph{mixed Shimura datum} is a triple $(\Pbf,\Ubf,Y)$ consisting of \begin{itemize}
 \item a connected linear $\Qbb$-group $\Pbf$, with a Levi decomposition $\Pbf=\Wbf\rtimes\Gbf$;
 \item a unipotent normal $\Qbb$-subgroup $\Ubf$ (necessarily contained in $\Wbf$);
 \item a $\Pbf(\Rbb)\Ubf(\Cbb)$-orbit $Y\subset\Yfrak(\Pbf)$;
 \end{itemize}
 
 such that by putting $\pi_\Ubf:\Pbf\ra\Pbf/\Ubf$ and $\pi_\Wbf:\Pbf\ra\Pbf/\Wbf=\Gbf$ for the quotient maps, the following properties hold for any $y\in Y$:\begin{enumerate}
 \item[(i)] the composition $\pi_\Ubf\circ y:\Sbb_\Cbb\ra\Pbf_\Cbb\ra(\Pbf/\Ubf)_\Cbb$ is defined over $\Rbb$;
 \item[(ii)] the composition $\pi_\Wbf\circ y\circ w:\Gbb_{\mrm \Rbb}\mono\Sbb\ra\Gbf_\Rbb$ is a central cocharacter of $\Gbf_\Rbb$, where $w:\Gbb_{\mrm \Rbb}\ra\Sbb$ is induced by $\Rbb^\times\mono\Cbb^\times$;
 \item[(iii)] the composition $\Ad_\Pbf\circ y:\Sbb_\Cbb\ra\Pbf_\Cbb\ra\GL_{\pfrak,\Cbb}$ induces on $\pfrak=\Lie\Pbf$ a rational mixed Hodge structure of type $\{(-1,-1),(-1,0),(0,-1),(-1,1),(0,0),(1,-1)\}$, with rational weight filtration $W_{-2}=\Lie\Ubf$, $W_{-1}=\Lie\Wbf$, and $W_0=\Lie\Pbf$; 
 \item[(iv)] the conjugation by $y(\ibf)$ induces on $\Gbf^\ad_\Rbb$ a Cartan involution, and $\Gbf^\ad$ admits no compact $\Qbb$-factors;
 \item[(v)] it is also required that the center of $\Gbf$ acts on $\Wbf$ through some $\Qbb$-torus isogeneous to the product of a compact $\Qbb$-torus with a split $\Qbb$-torus.
 \end{enumerate}
 
 $Y$ is actually a complex manifold on which $\Pbf(\Rbb)\Ubf(\Cbb)$ acts transitively preserving the complex structure.
 
 (2) A \emph{pure Shimura datum} (in the sense of Deligne \cite{deligne pspm}) is the same as a mixed Shimura datum $(\Gbf,X)$ where the unipotent radical $\Wbf$ is trivial. A \emph{Kuga datum} (cf.\cite{chen kuga}) is just a mixed Shimura datum $(\Pbf,Y)$ with $\Ubf=1$. 
 
 (3) For $S$ a subset of $Y$, the Mumford-Tate group of $S$, written as $\MT(S)$, is the smallest $\Qbb$-subgroup $\Pbf'$ of $\Pbf$ such that $y(\Sbb_\Cbb)\subset\Pbf'_\Cbb$ for all $y\in\Sbb$. A mixed Shimura datum $(\Pbf,\Ubf,Y)$ is \emph{irreducible} if $\Pbf$ equals $\MT(Y)$. 
 
 \end{definition}
 
 \begin{remark}[Deligne vs. Pink]\label{Deligne vs. Pink} In the original definition \cite{pink thesis}2.1, the space $Y$ is not a subset of $\Yfrak(\Pbf)$; Pink uses a  complex manifold $Y$ homogeneous under the Lie group $\Pbf(\Rbb)\Ubf(\Cbb)$, together with a $\Pbf(\Rbb)\Ubf(\Cbb)$-equivariant map $h:Y\ra \Yfrak(\Pbf)$ of finite fibers, such that the Hodge-theoretic conditions (i)-(iv) in \ref{mixed Shimura data} hold for points in $h(Y)$. One can show, cf. \cite{pink thesis}2.12, that the connected components of the space $Y$ in the sense of Pink are the same as the connected components of the space $Y$ in the sense of \ref{mixed Shimura data}. The main results of this paper focus on connected mixed Shimura varieties, and we prefer to use the simpler definition given above.
 
 \end{remark}
 
 \begin{definition}[morphisms of mixed Shimura data]\label{morphisms of mixed Shimura data}
 A \emph{morphism} between mixed Shimura data is of the form $(f,f_*):(\Pbf,\Ubf,Y)\ra(\Pbf',\Ubf',Y')$ where $f:\Pbf\ra\Pbf'$ is a $\Qbb$-group homomorphism sending $\Ubf$ into $\Ubf'$, and the push-forward $f_*:\Yfrak(\Pbf)\ra\Yfrak(\Pbf'),\ h\mapsto f\circ h$ sends $Y$ into $Y'$, such that $f_*:Y\ra Y'$ is  equivariant \wrt $f:\Pbf(\Rbb)\Ubf(\Cbb)\ra\Pbf'(\Rbb)\Ubf'(\Cbb)$; one can show  that $f_*:Y\ra Y'$ is a smooth map between complex manifolds, cf.\cite{pink thesis}2.3 and 2.4. We further single out the following cases:
 
 (1) $(\Pbf,\Ubf,Y)$ is said to be a \emph{mixed Shimura subdatum} of $(\Pbf',\Ubf',Y')$ if $f$ and $f_*$ are both injective.
 
 (2) For $\Nbf\subset\Pbf$ a normal $\Qbb$-subgroup, the \emph{quotient} of $(\Pbf,\Ubf,Y)$ by $\Nbf$ is the mixed Shimura datum $(\Pbf',\Ubf',Y')$ where $\Pbf'$ is the quotient $\Qbb$-group $\Pbf/\Nbf$, $\Ubf'$ is the image of $\Ubf$ in $\Pbf'$, and $Y'$ is the $\Pbf'(\Rbb)\Ubf'(\Cbb)$-orbit of the composition $\pi_\Nbf\circ y:\Sbb_\Cbb\ra\Pbf_\Cbb\ra\Pbf'_\Cbb$ for any $y\in Y$, where $\pi_\Nbf:\Pbf\ra \Pbf'=\Pbf/\Nbf$ is the natural projection, cf.\cite[2.9]{pink thesis}. We thus write $(\Pbf',\Ubf',Y')=(\Pbf/\Nbf,\Ubf/(\Ubf\cap\Nbf),Y/\Nbf)$ with $Y/\Nbf:=Y'$.
 
 It should be mentioned that in the quotient construction the map $Y\ra Y'$ is not  surjective in general. For example, if $(\Gbf,X)$ is a pure Shimura datum and $\Zbf$ is the center of $\Gbf$, then the quotient $(\Gbf^\ad,X^\ad)$ of $(\Gbf,X)$ by $\Zbf$ is a pure Shimura datum. The connected components of $X^\ad$ and $X$ are all isomorphic to the Hermitian symmetric domain defined by the connected Lie group $\Gbf^\ad(\Rbb)^+$ as the center of $\Gbf(\Rbb)^+$ acts on the domain trivially. However more connected components could appear in $X^\ad$ than in $X$, simply because $\Gbf^\ad(\Rbb)$ could have more connected components. This can be also seen from the exactness of $1\ra\Zbf(\Rbb)\ra\Gbf(\Rbb)\ra\Gbf^\ad(\Rbb)$ where the last arrow is not surjective in general, which is deduced from the exact sequence of linear $\Qbb$-groups  $1\ra\Zbf\ra\Gbf\ra\Gbf^\ad\ra 1$ with $\Zbf$ the center of $\Gbf$.
 
 When $\Nbf$ is unipotent, the map $Y\ra Y/\Nbf$ is surjective, whose fibers are isomorphic to $\Nbf(\Rbb)\Ubf_\Nbf(\Cbb)$ with $\Ubf_\Nbf=\Ubf\cap\Nbf$, and in this case we often use the more precise notation $Y/\Nbf(\Rbb)\Ubf_\Nbf(\Cbb)$  in place of the vague expression $Y/\Nbf$, cf.\cite{pink thesis} 2.18. 
 
 In particular, taking $\Nbf=\Ubf$ and $\Wbf$ successively, we see that a mixed Shimura datum fits into a sequence $$(\Pbf,\Ubf,Y)\ra(\Pbf/\Ubf,Y/\Ubf(\Cbb))\ra(\Pbf/\Wbf,Y/\Wbf(\Rbb)\Ubf(\Cbb))$$ where $(\Pbf/\Ubf,Y/\Ubf(\Cbb))$ is a Kuga datum and $(\Pbf/\Wbf,Y/\Wbf(\Rbb)\Ubf(\Cbb))$ is a pure Shimura datum.
 
 (3) As a natural combination of (1) and (2), when a morphism between mixed Shimura data $(f,f_*):(\Pbf,\Ubf,Y)\ra(\Pbf',\Ubf',Y')$ is given, its \emph{image} is the triple $(f(\Pbf),f(\Ubf),f_*(Y))$. One verifies directly from the definition that the image is a subdatum of $(\Pbf',\Ubf',Y')$ and equals the quotient of $(\Pbf,\Ubf,Y)$ by $\Nbf:=\Ker(f:\Pbf\ra\Pbf')$.

 (4) A \emph{pure section} of $(\Pbf,\Ubf,Y)$ associated to the Levi decomposition $\Pbf=\Wbf\rtimes\Gbf$ is a pure Shimura datum $(\Gbf,X)$ which is a subdatum of $(\Pbf,\Ubf,Y)$ such that the $\Qbb$-group homomorphism $\Gbf\mono\Pbf$ is given by the Levi decomposition and  the composition $(\Gbf,X)\mono(\Pbf,\Ubf,Y)\ra(\Pbf/\Wbf,Y/\Wbf(\Rbb)\Ubf(\Cbb))$ is an isomorphism.
 
 (5) For $(\Pbf_i,\Ubf_u,Y_i)$ two mixed Shimura data ($i=1,2$), we have the \emph{product} $(\Pbf_1\times\Pbf_2,\Ubf_1\times\Ubf_2,Y_1\times Y_2)$ which is a mixed Shimura datum in an evident way, cf.\cite{pink thesis}2.5.
 \end{definition}
 \begin{proposition}[unipotent radical and Levi decomposition]\label{unipotent radical and Levi decomposition} 
 Let $(\Pbf,\Ubf,Y)$ be a mixed Shimura datum, with $\Pbf=\Wbf\rtimes\Gbf$ a Levi decomposition. Write $\Vbf=\Wbf/\Ubf$. Then:
 
 (1) $\Ubf$ and $\Vbf$ are commutative, and $\Wbf$ is a central extension of $\Vbf$ by $\Ubf$ i.e. $1\ra\Ubf\ra\Wbf\ra\Vbf\ra 1$. Writing the group laws on $\Ubf$ and on $\Vbf$ additively and fixing an isomorphism of $\Qbb$-schemes $\Wbf\isom\Ubf\times\Vbf$, the group law on $\Wbf$ writes as $$(u,v)\cdot(u',v')=(u+u'+\frac{1}{2}\psi(v,v'),v+v')$$ where the commutator map $\Wbf\times\Wbf\ra\Wbf$ has image in $\Ubf$ and factors through a unique alternating bilinear map $\psi:\Vbf\times\Vbf\ra\Ubf$.

 (2) For any $y\in Y$, the action of $\Sbb_\Cbb$ on $\Lie\Pbf_\Cbb$ induces on $\Ubf$ resp. on $\Vbf$ (both viewed as finite-dimensional $\Qbb$-vector spaces) a Hodge structure of type $(-1,-1)$ resp. of type $\{(-1,0),(0,-1)\}$. 
 
 (3) For any $y\in Y$, write $x:\Sbb_\Cbb\ra\Gbf_\Cbb$ for the composition $\pi_\Wbf\circ y$. Then $x$ is defined over $\Rbb$ and $x\in Y$. Outting $X=\Gbf(\Rbb)x$ the orbit of $x$ in $Y$ under $\Gbf(\Rbb)$ we obtain a pure Shimura subdatum $(\Gbf,X)$ of $(\Pbf,\Ubf,Y)$, and the composition of the inclusion with the reduction modulo $\Wbf$ is an isomorphism: $(\Gbf,X)\mono(\Pbf,\Ubf,Y)\epim(\Pbf/\Wbf,Y/\Wbf)$. Moreover the Hodge types of $\rho_\Ubf\circ x$ resp. of $\rho_\Vbf\circ x$ $(-1,-1)$ resp. $\{(-1,0),(0,-1)\}$, where $\rho_\Ubf$ resp. $\rho_\Vbf$ are the action of $\Gbf$ on $\Ubf$ resp. on $\Wbf/\Ubf=\Vbf$ by conjugation, and $\psi:\Vbf\times\Vbf\ra\Ubf$ is $\Gbf$-equivariant. The representation $\rho_\Ubf$ factors through a split $\Qbb$-torus.
 
 In particular, $(\Pbf,\Ubf,Y)=\Wbf\rtimes(\Gbf,X)$ is a split unipotent extension in the sense of \cite{pink thesis} 2.21.
 
 
 
 (4) $\Pbf^\der$ equals $\Wbf\rtimes\Gbf^\der$ and it admits no non-trivial compact quotient $\Qbb$-groups.
 \end{proposition}
 
 \begin{proof}
 (1) and (2) are found in \cite{pink thesis} 2.15, 2.16.
 

 
 (3) Since $\pi_\Ubf\circ y$  is already defined over $\Rbb$,   the homomorphism $\pi_\Ubf\circ y:\Sbb\ra(\Vbf\rtimes\Gbf)_\Rbb$, whose image is an $\Rbb$-torus, factors through some Levi $\Rbb$-subgroup of the form $v\Gbf_\Rbb v^\inv$ for some $v\in\Vbf(\Rbb)$. Thus $\pi_\Ubf\circ y:\Sbb_\Cbb\ra\Pbf_\Cbb/\Ubf_\Cbb$ factors through $v\Gbf_\Cbb v^\inv$. The pre-image of $v\Gbf_\Cbb v^\inv$ in $\Pbf_\Cbb$ is $\Ubf_\Cbb\rtimes(w\Gbf_\Cbb w^\inv)$ for some $w\in\Wbf(\Cbb)$ lifting $v$. In $\Ubf_\Cbb\rtimes w\Gbf_\Cbb w^\inv$ the maximal reductive $\Cbb$-subgroups are Levi $\Cbb$-subgroups of the form $w'\Gbf_\Cbb w'^\inv$ with $w'\in\Ubf(\Cbb)w\subset\Ubf(\Cbb)\Wbf(\Rbb)$. In particular, conjugate $y$ by $w'$ we get $x\in Y$ such that $x(\Sbb_\Cbb)\subset\Gbf_\Cbb\subset\Pbf_\Cbb$. Since the composition $\Gbf\mono\Pbf\epim\Gbf$ is the identity, we factorize $x:\Sbb_\Cbb\ra\Pbf_\Cbb$ into the composition $\Sbb_\Cbb\ra\Pbf_\Cbb\epim\Gbf_\Cbb\mono\Pbf_\Cbb$. The composition $\pi_\Wbf\circ x:\Sbb_\Cbb\ra\Pbf_\Cbb\epim\Gbf_\Cbb$ is defined over $\Rbb$, and the projection $\Pbf_\Cbb\ra\Gbf_\Cbb$ is defined over $\Rbb$ as $\Pbf\ra\Gbf$ is already defined over $\Qbb$. Hence $x$ is defined over $\Rbb$. 
 
 
 
 Since $x\in Y$, and $\Lie\Gbf=\Lie\Pbf/\Lie\Wbf$, we see that the Hodge structure given by $x:\Sbb\ra\Gbf_\Rbb$ on $\Lie\Gbf$ is of type $\{(-1,1),(0,0),(1,-1)\}$, and the conjugation by $x(\ibf)$ induces a Cartan involution on $\Gbf_\Rbb^\ad$.  The $\Gbf(\Rbb)$-orbit $X$ of $x$ inside $Y\subset\Yfrak(\Pbf)$ clearly lies in $\Yfrak(\Gbf)$ (and actually lies in the real part $\Xfrak(\Gbf))$, hence the pair $(\Gbf,X)$ is a pure Shimura datum, and the inclusion $(\Gbf,X)\mono(\Pbf,\Ubf,Y)$ makes it a pure Shimura subdatum. 
 
 The claims on Hodge types and the pairing $\psi$ are immediate from (1) and (2). The claim on the action of $\Gbf$ on $\Ubf$ is clear because $\Pbf$ acts on $\Ubf$ through a split $\Qbb$-torus by \cite{pink thesis} 2.14 and $\Gbf$ acts through $\Gbf\mono\Pbf$.
 
 
 
 (4) From \cite{pink thesis} 2.10 we know that $\Pbf^\der$ contains $\Wbf$. It clearly contains $\Gbf^\der$, hence $\Pbf^\der\supset\Wbf\rtimes\Gbf^\der$. The quotient $\Pbf/(\Wbf\rtimes\Gbf^\der)$ is already commutative, which gives the inclusion $\Pbf^\der\subset\Wbf\rtimes\Gbf^\der$ in the other direction. From \ref{mixed Shimura data} (iv) we know that $\Pbf^\der$ admits no compact quotient $\Qbb$-groups other than the trivial one.
 \end{proof}

  \begin{notation}[group law]\label{group law}
  Aside from the group law $(u,v)\cdot(u',v')=(u+u'+\frac{1}{2}\psi(v,v'),v+v')$ and the evident equality $(u,v)^n=(nu,nv)$, the following identities will be useful for elements $(u,v,g)$ in $\Pbf\isom\Ubf\times\Vbf\times\Gbf$, in which the neutral element is $(0,0,1)$:\begin{itemize}
     \item multiplication $(u,v,g)(u',v',g')=(u+g(u')+\psi(v,g(v')),v+g(v'),gg')$;
     \item inverse $(u,v,g)^\inv=(-g^\inv(u) , -g^\inv(v), g^\inv)$, namely $(w,g)^\inv=(g^\inv(w^\inv),g^\inv)$ for $w=(u,v)$
     \item and the commutator between $\Wbf$ and $\Gbf$ is $$(u,v,1)(0,0,g)(-u,-v,1)(0,0,g^\inv)=(u-g(u),v-g(v),1)$$
   \end{itemize}
   where we write $g(u)=gug^\inv=\rho_\Ubf(g)(u)$ and similarly for $g(v)$, $g(w)$.
  \end{notation}

 We thus prefer treating a general mixed Shimura datum as a split unipotent  extension of a pure Shimura datum by two unipotent $\Qbb$-groups subject to certain conditions, and we often reformulate this as the following:
 
 \begin{definition-proposition}[fibred mixed Shimura data]\label{fibred mixed Shimura data new}
 (1) Let $(\Gbf,X)$ be a pure Shimura datum, and let $\rho_\Ubf:\Gbf\ra\GLbf_\Ubf$ and $\rho_\Vbf:\Gbf\ra\GLbf_\Vbf$ be two finite-dimensional algebraic representation, together with an alternating $\Gbf$-equivariant bilinear map $\psi:\Vbf\times\Vbf\ra\Ubf$, giving rise to a central extension of unipotent $\Qbb$-groups $1\ra\Ubf\ra\Wbf\ra\Vbf\ra 1$. Assume that \begin{itemize}
 \item for any $x\in X$, the composition $\rho_\Ubf\circ x$ is a rational Hodge structure of type $(-1,-1)$, and $\rho_\Vbf\circ x$ is a rational Hodge structure of type $\{(-1,0),(0,-1)\}$;
 
 \item the connected center of $\Gbf$ acts on $\Ubf$ and on $\Vbf$ through $\Qbb$-tori subject to the condition (iv) in \ref{mixed Shimura data}.
 \end{itemize}  Then by putting $\Pbf=\Wbf\rtimes\Gbf$ and $Y$ the $\Pbf(\Rbb)\Ubf(\Cbb)$-orbit of any $x:\Sbb\ra\Gbf_\Rbb\ra\Pbf_\Rbb$, the triple $(\Pbf,\Ubf,Y)$ thus obtained is a mixed Shimura datum. The canonical projection $(\Pbf,\Ubf,Y)\ra(\Gbf,X)$ is the quotient by $\Wbf$, and $(\Gbf,X)$ is naturally a pure section by the evident inclusions $\Gbf\mono\Pbf$ and $X\mono Y$. In particular, $Y$ can be viewed as a holomorphic vector bundle over $X$, whose fibers are isomorphic to the Lie algebra of $\Wbf(\Rbb)\Ubf(\Cbb)$.
  
 We call $(\Pbf,\Ubf,Y)$ the mixed Shimura datum \emph{fibred} over $(\Gbf,X)$ by the representations $\rho_\Ubf$ and $\rho_\Vbf$ and the alternating map $\psi$. We write $(\Pbf,\Ubf,Y)=\Wbf\rtimes(\Gbf,X)=(\Ubf,\Vbf)\rtimes(\Gbf,X)$ and $Y=\Wbf(\Rbb)\Ubf(\Cbb)\rtimes X$ to emphasize that the role of the Levi decomposition.
 
 (2) A \emph{morphism} between fibred mixed Shimura data is a commutative diagram of the form $$\xymatrix{(\Pbf,\Ubf,Y)\ar[d]^{\pi_\Wbf} \ar[r]^{(f,f_*)}&(\Pbf',\Ubf',Y')\ar[d]^{\pi_{\Wbf'}}\\ (\Gbf,X)\ar[r]^{(f,f_*)} &(\Gbf',X')}$$ where the vertical arrows are reductions modulo the unipotent radicals, inducing the bottom horizontal morphism of pure Shimura data from the upper one. Note that the commutative diagram give rise to homomorphisms   $\alpha:\Vbf\ra\Vbf'$, $\beta:\Ubf\ra\Ubf'$ and $\Wbf\ra\Wbf'$ with $\beta(\psi(v,v'))=\psi'(\alpha(v),\alpha(v'))$. 
 
 Identify $(\Gbf,X)$ resp. $(\Gbf',X')$ as a pure subdatum of $(\Pbf,\Ubf,Y)$ resp. of $(\Pbf',\Ubf',Y')$ via split unipotent extension as in (1). If the morphism $(f,f_*):(\Pbf,\Ubf,Y)\ra(\Pbf',\Ubf',Y')$ sends the pure subdatum $(\Gbf,X)$ into $(\Gbf',X')$, then the morphisms $\alpha$ and $\beta$ are equivariant \wrt $f:\Gbf\ra\Gbf'$, and $\Pbf\ra\Pbf'$ can be recovered as $(u,v,g)\mapsto(\beta(u),\alpha(v),f(g))$ when we use the isomorphisms of $\Qbb$-schemes $\Pbf=\Ubf\times\Vbf\times\Gbf$ and $\Pbf'=\Ubf'\times\Vbf'\times\Gbf'$.

 Conversely, assume that fibred mixed Shimura data $(\Pbf,\Ubf,Y)=(\Ubf,\Vbf)\rtimes(\Gbf,X)$ and $(\Pbf',\Ubf',Y')=(\Ubf',\Vbf')\rtimes(\Gbf',X')$ are given. If $(f,f_*):(\Gbf,X)\ra(\Gbf',X')$ is  a morphism of pure Shimura data, together with $f$-equivariant homomorphisms $\alpha:\Vbf\ra\Vbf'$ $\beta:\Ubf\ra\Ubf'$ and $\beta(\psi(v,v'))=\psi'(\alpha(v),\alpha(v'))$. Then $(f,f_*):(\Gbf,X)\ra(\Gbf',X')$ extends   to a morphism of mixed Shimura data $(f,f_*):(\Pbf,\Ubf,Y)\ra(\Pbf',\Ubf',Y')$, with the $\Qbb$-group homomorphism being $(u,v,g)\mapsto(\beta(u),\alpha(v),f(g))$ under the isomorphism of $\Qbb$-schemes $\Pbf\isom\Ubf\times\Vbf\times\Gbf$ and $\Pbf'=\Ubf'\times\Vbf'\times\Gbf'$.
  
 \end{definition-proposition}
 
 \begin{proof}
 (1)  is clear from \ref{unipotent radical and Levi decomposition} and \cite{pink thesis} 2.16, 2.17, 2.21. See \cite{pink thesis} 2.18 and 2.19 for the proof for $Y$ being a holomorphic vector bundle over $X$. 
 
 (2) When $(f,f_*):(\Pbf,\Ubf,Y)\ra(\Pbf',\Ubf',Y')$ is a morphism of mixed Shimura data, the Lie algebra map $\Lie f:\Lie\Pbf\ra\Lie\Pbf'$ respects the rational weight filtration and the central extension structures on the unipotent radicals. Hence $f(\Ubf)\subset\Ubf'$, $f(\Wbf)\subset\Wbf'$, with $\beta(\psi(v,v'))=\psi'(\alpha(v),\alpha(v'))$ for $v,v'\in\Vbf$, where $\alpha$ and $\beta$ are induced by $f$. Reduce modulo the unipotent radicals of $(f,f_*)$ gives $(\Gbf,X)\ra(\Gbf',X')$ together with a commutative diagram of the mentioned form.
 
 If moreover $(f,f_*):(\Pbf,\Ubf,Y)\ra(\Pbf',\Ubf',Y')$ sends $(\Gbf,X)$ into $(\Gbf',X')$, then the homomorphism $f:\Pbf\ra\Pbf'$ sends the Levi $\Qbb$-subgroup $\Gbf$ into $\Gbf'$. The homomorphisms between normal unipotent $\Qbb$-groups $\Ubf\ra\Ubf'$ and $\Wbf\ra\Wbf'$ are equivariant \wrt $\Pbf\ra\Pbf'$, hence we get $\alpha:\Vbf\ra\Vbf'$ and $\beta:\Ubf\ra\Ubf'$ equivariant \wrt $f:\Gbf\ra\Gbf'$. The recovery of $f:\Pbf\ra\Pbf'$ by $\alpha$, $\beta$ and $f:\Gbf\ra\Gbf'$ is immediate.
 
 Conversely, if we are given fibred mixed Shimura data $(\Pbf,\Ubf,Y)=(\Ubf,\Vbf)\rtimes(\Gbf,X)$ and $(\Pbf',\Ubf',Y')=(\Ubf',\Vbf')\rtimes(\Gbf',X')$, a $\Qbb$-group homomorphism $f:\Gbf\ra\Gbf'$ with $f$-equivariant maps $\alpha$ and $\beta$ naturally gives rise to a unique $\Qbb$-group homomorphism $f:\Pbf\ra\Pbf'$ subject to the formula $(u,v,g)\mapsto(\beta(u),\alpha(v),f(g))$, and $f_*:\Yfrak(\Pbf)\ra\Yfrak(\Pbf')$ sends $Y$ the $\Pbf(\Rbb)\Ubf(\Cbb)$-orbit of $X$ into $Y'$ the $\Pbf'(\Rbb)\Ubf'(\Cbb)$-orbit of $X'$. It is easy to verify that for any $y\in Y$, the mixed Hodge structure on $\Lie\Pbf$ satisfies the constraints in \ref{mixed Shimura data} and that $(f,f_*):(\Pbf,\Ubf,Y)\ra(\Pbf',\Ubf',Y')$ is a morphism of mixed Shimura data, using the Hodge type conditions on $\Ubf$, $\Vbf$ and $\Ubf'$, $\Vbf'$ given in (1).\end{proof}

  In particular we have the following corollaries on pure sections and subdata:
  \begin{corollary}[pure section]\label{pure section} If $(\Pbf,\Ubf,Y)=\Wbf\rtimes(\Gbf,X)$ is a fibred mixed Shimura datum, then the pure sections of $(\Pbf,\Ubf,Y)\epim(\Gbf,X)$ are exactly subdata of the form $(w\Gbf w^\inv,wX)$ with $w$ running through $\Wbf(\Qbb)$, and they are the same as maximal pure subdata of $(\Pbf,\Ubf,Y)$.  
  \end{corollary}
  
  \begin{proof}
 It is clear that each $(w\Gbf w^\inv,wX)$ is a pure section for any given $w\in\Wbf(\Qbb)$. Conversely, if $(\Gbf',X')$ is a pure section in the sense of \ref{morphisms of mixed Shimura data}(3), then $\Gbf'=w\Gbf w^\inv$ is conjugate to $\Gbf$ by some $w\in\Wbf(\Qbb)$, hence $(w^\inv\Gbf'w,w^\inv X')=(\Gbf,w^\inv X')$ is a pure subdatum of $(\Pbf,\Ubf,Y)$. Since $\Pbf=\Wbf\rtimes\Gbf$, the composition of the evident maps between $\Yfrak(\Gbf)\ra\Yfrak(\Pbf)\ra\Yfrak(\Gbf)$ induced by $\Gbf\mono\Pbf\epim\Gbf$ is the identity. Apply the composition to the subset $w^\inv X'\subset\Xfrak(\Gbf)\subset\Yfrak(\Gbf)$, we see that its image in $\Yfrak(\Gbf)$ must be $X$ because $(\Gbf,w^\inv X')$ is a pure section, hence $w^\inv X'=X$.
 
 The maximality is clear because maximal reductive $\Qbb$-subgroups of $\Pbf$ are exactly the Levi $\Qbb$-subgroups, hence conjugate to $\Gbf$ by $\Wbf(\Qbb)$.
  \end{proof}
 
 \begin{corollary}[structure of subdata]\label{structure of subdata} Let $(\Pbf,\Ubf,Y)$ be a mixed Shimura datum fibred as $(\Ubf,\Vbf)\rtimes(\Gbf,X)$, and let $(\Pbf',\Ubf',Y')$ be a mixed Shimura subdatum. Then there exists a pure Shimura subdatum $(\Gbf',X')$ of $(\Gbf,X)$, an element $w\in\Wbf(\Qbb)$, and a unipotent $\Qbb$-subgroup $\Wbf'$ of $\Wbf$, such that $\Pbf'=\Wbf'\rtimes w\Gbf' w^\inv$ and $(\Pbf',\Ubf',Y')\isom(\Ubf',\Vbf')\rtimes(w\Gbf'w^\inv,wX')$ as a fibred mixed Shimura datum. Here $\Vbf'=\Wbf'/\Ubf'$ resp. $\Ubf'=\Ubf\cap\Wbf'$ is a $\Qbb$-subspace of $\Vbf$ resp. of $\Ubf$ stabilized under $w\Gbf'w^\inv$, and $\Vbf\times\Vbf\ra\Ubf$ restricts to $\Vbf'\times\Vbf'\ra\Ubf'$ and is equivariant under $w\Gbf'w^\inv$. The pure subdatum $(w\Gbf'w^\inv,wX')$ in $(\Pbf,\Ubf,Y)$ is a pure section of $(\Pbf',\Ubf',Y')$.
 \end{corollary}
 
 \begin{proof}
 The idea is the same as \cite{chen kuga} 2.10. Let $(\Pbf_0,Y_0)$ be a maximal pure Shimura subdatum of $(\Pbf',\Ubf',Y')$, then its image $(\Gbf',X')$ in $(\Gbf,X)$ is a pure Shimura subdatum. Note that $(\Pbf',\Ubf',Y')$ is a subdatum of $\Wbf\rtimes(\Gbf',X')$ containing a maximal pure subdatum $(\Pbf_0,Y_0)=(w\Gbf'w^\inv,w')$ for some $w\in\Wbf(\Qbb)$. We are thus reduced to the case when $(\Gbf',X')=(\Gbf,X)$. 
 
 In this case $(\Pbf_0,Y_0)=(w\Gbf w^\inv,wX)$, and the unipotent radical $\Wbf'$ of $\Pbf'$ is naturally a $\Qbb$-subgroup of $\Wbf$ stabilized by $w\Gbf w^\inv$-conjugation. Using $(w\Gbf w^\inv,wX)$ as a pure section corresponding to the Levi decomposition $\Pbf'=\Wbf'\rtimes(w\Gbf w^\inv)$, we see that the intersection $\Ubf':=\Ubf\cap\Wbf'$ is the weight -2 part due to the rational weight filtration given by any $y\in wX$, and $\Vbf'=\Wbf'/\Ubf'$ equals the image of $\Wbf'$ in $\Vbf$, which is clearly stable under $w\Gbf w^\inv$. The bilinear map $\psi:\Vbf\times\Vbf\ra\Ubf$ clearly restricts to $\psi':\Vbf'\times\Vbf'\ra\Ubf'$ and is $w\Gbf w^\inv$-equivariant, as immediate consequences of the Lie bracket structure on $\Lie\Wbf'$ and the Hodge types. 
 \end{proof}
  
 The following lemma is the mixed analogue of \cite{ullmo yafaev} Lemma 3.7.  
  
 \begin{lemma}[common Mumford-Tate group]\label{common Mumford-Tate group}. Let $(\Pbf,Y)=\Wbf\rtimes(\Gbf,X)$ be a mixed Shimura datum.   If $\Pbf'\subset\Pbf$ is a $\Qbb$-subgroup coming from some subdatum $(\Pbf',Y')$, then there are only finitely many subdata of the form $(\Pbf',Y'')$ in $(\Pbf,Y)$. 
 \end{lemma}
 
 \begin{proof}
 When $(\Pbf,\Ubf,Y)$ is pure, this is proved in \cite{ullmo yafaev} Lemma 3.7. For a general $\Qbb$-subgroup $\Pbf'$, if there exists  a subdatum of the form $(\Pbf',\Ubf',Y')$, then $\Ubf'=\Ubf\cap\Pbf'$ and the unipotent radical $\Wbf'=\Wbf\cap\Pbf'$ in $\Pbf'$ are independent of $Y'$ by the constraints of Hodge types. Choose a Levi decomposition $\Pbf'=\Wbf'\rtimes w\Gbf' w^\inv$, we have $(\Pbf',\Ubf',Y')=\Wbf'\rtimes(w\Gbf' w^\inv,wX')$ for some pure Shimura subdatum $(w\Gbf'w^\inv,wX')\subset(w\Gbf w^\inv,wX)$, namely the $(\Pbf',\Ubf',Y')$ is constructed out of $(w\Gbf' w^\inv,wX')$ by some unipotent $\Qbb$-subgroup $\Wbf'$. There are at most finitely many pure Shimura subdatum of the form $(w\Gbf'w^\inv,wX')$ in $(w\Gbf w^\inv,wX)$, hence the finiteness of mixed Shimura subdata associated to $\Pbf'$ follows. \end{proof}

 We also include the following result that allow us to  generate subdata by ``taking orbits'':
 
 \begin{lemma}[generating subdata]\label{generating subdata}
 
 Let $(\Pbf,\Ubf,Y)=\Wbf\rtimes(\Gbf,X)$ be a mixed Shimura datum. Let $\Pbf'$ be a $\Qbb$-subgroup of $\Pbf$ admitting no compact semi-simple quotient $\Qbb$-group, and $\Pbf'_\Cbb\supset y(\Sbb_\Cbb)$ for some $y\in Y$. Then 
 
 (1) $\Pbf'=\Wbf'\rtimes w\Gbf'w^\inv$ for some reductive $\Qbb$-subgroup $\Gbf'$ of $\Gbf$ and some $w\in\Wbf(\Qbb)$ and unipotent $\Qbb$-subgroup $\Wbf'\subset\Wbf$. 
 
 (2) if moreover the connected center of $w\Gbf'w^\inv$ acts on $\Wbf'$ through a $\Qbb$-torus subject to \ref{mixed Shimura data}(v), then the triple $(\Pbf',\Ubf',Y')$ with $\Ubf'=\Ubf\cap\Wbf'$ and $Y'=\Pbf'(\Rbb)\Ubf'(\Cbb)y$ is a mixed Shimura subdatum of $(\Pbf,\Ubf,Y)$.
 
 (3) In particular, if $\Pbf'=\MT(Y^+)$ is the generic Mumford-Tate group of a connected component of $Y$, then  (1) and (2) holds for $\Pbf'$, with $\Pbf'^\der=\Pbf^\der$.
 \end{lemma}
 
 \begin{proof}
 
 (1) The image of $\Pbf'$ along $\pi=\pi_\Wbf:\Pbf\ra\Gbf$ is a $\Qbb$-subgroup $\Gbf'$ of $\Gbf$ such that $\Gbf'_\Cbb\supset x(\Sbb_\Cbb)$ for $x=\pi_*y$. Since $x$ is already defined over $\Rbb$ by \ref{mixed Shimura data}(i), we have $x(\Sbb)\subset\Gbf'_\Rbb\subset\Gbf_\Rbb$. Since the centralizer of $x(\Sbb)$ in $\Gbf_\Rbb$ is compact, by \cite{eskin mozes shah} Lemma 5.1 we see that $\Gbf'$ is reductive. The kernel $\Wbf':=\Ker(\Pbf'\ra\Gbf')$ is contained in $\Wbf$, hence unipotent. Thus $\Pbf'$ admits a Levi decomposition of the form $\Wbf'\rtimes\Hbf'$, where $\Hbf'$ is a maximal reductive $\Qbb$-subgroup of $\Pbf'$. $\Hbf$ extends to a maximal reductive $\Qbb$-subgroup in $\Pbf$ of the form $w\Gbf w^\inv$, hence $w^\inv\Hbf w$ is a reductive $\Qbb$-subgroup of $\Gbf$, and it coincides with the image of $\Pbf'$ modulo $\Wbf'$, which gives $\Hbf=w\Gbf'w^\inv$.
 
 (2) Note that $w\Gbf'w^\inv$ admits no compact semi-simple quotient $\Qbb$-group as this is already true for $\Pbf'$. Since the homomorphism $y:\Sbb_\Cbb\ra\Pbf_\Cbb$ factors through $\Pbf'_\Cbb$, we see that the Lie algebra $\pfrak'=\Lie\Pbf'$ is a rational mixed Hodge substructure of $\pfrak=\Lie\Pbf$, where the weight filtration is induced from the one on $\pfrak$ by restriction, and the Hodge types do not exceed the set $$\{(-1,-1),(-1,0),(0,-1),(-1,1),(0,0),(1,-1)\}.$$ Thus $\Ubf'=\Ubf\cap\Pbf'$ is the weight -2 part and $\Wbf'$ is the part of weight at most -1. The involution induced by $y(\ibf)$ in $\Gbf_\Rbb$ stabilizes $\Gbf'_\Rbb$, hence it induces further a Cartan involution on $\Gbf'^\ad_\Rbb$ because $\Gbf'^\ad$ admits no compact $\Qbb$-factors. The remaining conditions in \ref{mixed Shimura data} are automatic, hence $(\Pbf',\Ubf',Y')$ is a mixed Shimura datum, and it is clearly a subdatum of $(\Pbf,\Ubf,Y)$. 
 
 (3) When $\Pbf'=\MT(Y^+)$, the image of $Y^+$ in $X$ is a connected component $X^+$ of $X$, and the image $\Gbf'$ of $\Pbf'$ is a reductive $\Qbb$-subgroup of $\Gbf$ such that $x(\Sbb)\subset\Gbf'_\Rbb$ for all $x\in X^+$. If $\MT(X^+)\subsetneq \Gbf'$, then the pre-image $\Pbf''$ of $\MT(X^+)$ in $\Pbf'$ is a proper $\Qbb$-subgroup and $\Pbf''_\Cbb\supset y(\Sbb_\Cbb)$ for all $y\in Y^+$, which is absurd, and we get $\Gbf'=\MT(X^+)$. When $x$ runs through $X^+$, using \ref{common Mumford-Tate group} we only get finitely many pure subdata of the form $(\Gbf',X'_{i}=\Gbf'(\Rbb)x)$, $i=1,\cdots,m$. Each $X'_i$ is an complex submanifold of $X$, whose connected components are Hermitian symmetric subdomains of connected components of $X$, and the finite union $\bigcup_iX'_i$ contains $X^+$. It turns out that at least one of them, written as $X'$, is of dimension equal to $\dim X^+$, and we must have $X'^+=X^+$ for some connected component of $X'$. Since $X'^+$ resp. $X^+$ is homogeneous under $\Gbf'^\der(\Rbb)^+$ resp. under $\Gbf^\der(\Rbb)^+$, the inclusion of connected semi-simple Lie groups  $\Gbf'^\der(\Rbb)^+\subset\Gbf^\der(\Rbb)^+$ has to be an equality, and we get $\Gbf'^\der=\Gbf^\der$, and only one subdatum of the form $(\Gbf',X')$ is produced this way: $X'=\Gbf'(\Rbb)X^+$. 
 
 In particular, the center of $\Gbf'$ is a $\Qbb$-subtorus of $\Gbf$, and its action on $\Wbf$ satisfies the condition (v) in \ref{mixed Shimura data}.
 
 The kernel of $\Pbf'\ra\Gbf'$ is unipotent, hence $\Pbf'=\Wbf'\rtimes\Gbf'$ for some unipotent $\Qbb$-subgroup $\Wbf'\subset\Wbf$, which is the extension of $\Vbf'=\Wbf'/\Ubf'$ by $\Ubf':=\Ubf\cap\Pbf'$. When $y$ runs through $Y^+$, again by \ref{common Mumford-Tate group} we only get finitely many subdata of the form $(\Pbf',\Ubf',Y'_i)$ with $Y'_i=\Pbf'(\Rbb)\Ubf'(\Cbb)y_i$ for some $y_i\in Y^+$, and each connected component of $Y'_i$ is a complex submanifold of $Y$ homogeneous under $\Pbf'(\Rbb)^+\Ubf'(\Cbb)$. We thus have a finite union $\bigcup_iY_i$ containing $Y^+$. By \cite{pink thesis} 2.19, each $Y_i^+$ is a complex vector bundle over an Hermitian symmetric domain isomorphic to $X^+$, and the fibers are isomorphic to $\Wbf'(\Rbb)\Ubf'(\Cbb)$. If $\Wbf'\subsetneq\Wbf$, then the finite union $\bigcup_iY_i$ cannot contain $Y^+$ by dimension arguments because $Y^+$ is isomorphic to a complex vector bundle over $X^+$ with fibers isomorphic to $\Wbf(\Rbb)\Ubf(\Cbb)$. Hence we must have $\Wbf'=\Wbf$, and thus $\Pbf'^\der=\Wbf\rtimes\Gbf^\der=\Pbf^\der$. 
 
 \end{proof}
 
 \begin{corollary}[pure irreducibility]\label{pure irreducibility} Let $(\Pbf,Y)=\Wbf\rtimes(\Gbf,X)$ be a mixed Shimura datum fibred over a pure one $(\Gbf,X)$. Then $(\Pbf,Y)$ is irreducible \ifof $(\Gbf,X)$ is irreducible.
 
 \end{corollary}
 
 \begin{proof}
 If $(\Pbf,Y)$ is irreducible, then $(\Gbf,X)$ is irreducible by \cite{chen kuga} 2.4(2).
 
 Conversely, assume that $(\Gbf,X)$ is irreducible, and $\Pbf'$ is a $\Qbb$-subgroup of $\Pbf$ such that $y(\Sbb_\Cbb)\subset\Pbf'_\Cbb$. Write $\Gbf'$ for the image of $\Pbf'$ in $\Gbf$, then $x(\Sbb)\subset\Gbf'_\Rbb$ for all $x\in X$, which gives $\Gbf'=\Gbf$, and $\Pbf'=\Wbf'\rtimes w\Gbf w^\inv$ for some $w\in\Wbf(\Qbb)$ by \ref{generating subdata}(1). Since $Y\supset Y^+$ for any connected component $Y^+$ of $Y$, we have $\Pbf'\supset\MT(Y^+)$, and thus $\Pbf'\supset\Wbf$ by the arguments in \ref{generating subdata}(3), which gives $\Pbf'=\Wbf\rtimes w\Gbf w^\inv=\Wbf\rtimes\Gbf=\Pbf$.
 \end{proof}

 \begin{convention}\label{abbreviation of notations} In a mixed Shimura datum $(\Pbf,\Ubf,Y)$, the $\Qbb$-group $\Ubf$, the unipotent radical $\Wbf$, and thus the quotient $\Vbf=\Wbf/\Ubf$ as well, are uniquely determined by any $y\in Y$, because the weight filtration of the mixed Hodge structure given by $\Ad_\Pbf\circ y:\Sbb_\Cbb\ra\Pbf_\Cbb\ra\GLbf_{\pfrak \Cbb}$ on $\pfrak=\Lie\Pbf$ determines $\Lie\Ubf$ and $\Lie\Wbf$, which in turn determine the connected unipotent $\Qbb$-groups $\Ubf$ and $\Wbf$. We will call $\Ubf$ the unipotent part of weight -2 of $\Pbf$ or of the datum. In the sequel we simply write $(\Pbf,Y)$ for the datum, and $\Ubf$ always denotes the unipotent part of weight -2 if no further use of the notation is specified. As for morphism between mixed Shimura data, we often write $f$ instead of a pair $(f,f_*)$.
 \end{convention}

 \begin{definition}[mixed Shimura varieties, cf. \cite{pink thesis} 3.1]\label{mixed shimura varieties}
 
 Let $(\Pbf,Y)=(\Ubf,\Vbf)\rtimes(\Gbf,X)$ be a mixed Shimura datum, and let $K\subset\Pbf(\adele)$ be a \cosg. The (complex) \emph{mixed Shimura variety} associated to $(\Pbf,Y)$ at level $K$ is the quotient space $$M_K(\Pbf,Y)(\Cbb)=\Pbf(\Qbb)\bsh[ Y\times\Pbf(\adele)/K]\isom\Pbf(\Qbb)_+\bsh[ Y^+\times\Pbf(\adele)/K]$$ where the last equality makes sense for any connected component $Y^+$ of $Y$ because $\Pbf(\Qbb)_+$ equals the stabilizer of $Y^+$ in $\Pbf(\Qbb)$.
 
 Using the finiteness of class numbers  in \cite{platonov rapinchuk} 8.1, we see that the double quotient $\Pbf(\Qbb)_+\bsh\Pbf(\adele)/K$ is finite. Writing $\Rcal$ for a set of  representatives, we then have $$M_K(\Pbf,Y)(\Cbb)=\coprod_{a\in\Rcal}\Gamma_K(a)\bsh Y^+$$ with $\Gamma_K(a)=\Pbf(\Qbb)_+\cap aKa^\inv$ a congruence subgroup of $\Pbf(\Rbb)_+$.
 
 Pink has shown in \cite{pink thesis} that such double quotients are normal quasi-projective varieties over $\Cbb$, generalizing  a theorem of Baily and Borel cf.\cite{baily borel}. He has further shown that mixed Shimura varieties $M_K(\Pbf,Y)$ admit canonical models over their reflex fields $E(\Pbf,Y)$, which are certain number fields embedded in $\Cbb$. In this paper, we treat mixed Shimura varieties as algebraic varieties over $\Qac$, and we denote them as $M_K(\Pbf,Y)$, equipped with the Galois action using the canonical model. In Section 2 and 3 we only use complex mixed Shimura varieties, and in Section 4 and 5 we will use some elementary properties of canonical models. 
 
 \emph{Kuga varieties} resp. \emph{pure Shimura varieties} are mixed Shimura varieties associated to Kuga data resp. pure Shimura data.
 
 \end{definition} 
 
 If $(\Pbf,Y)\ra(\Pbf',Y')$ is a morphism of mixed Shimura data, then we have the inclusion of reflex fields $E(\Pbf',Y')\subset E(\Pbf,Y)$, cf. \cite{pink thesis} 11.2. For the moment it suffices to know that the following morphisms are functorially defined with respect to the canonical models:
 
 \begin{definition}[morphisms of mixed Shimura varieties and Hecke translates, cf. \cite{pink thesis} 3.4]\label{morphisms of mixed Shimura varieties and Hecke translates}
 
  (1) Let $f:(\Pbf,Y)\ra(\Pbf',Y')$ be a morphism of mixed Shimura data, with \cosgs $K\subset\Pbf(\adele)$ and $K'\subset\Pbf'(\adele)$ such that $f(K)\subset K'$, then there exists a unique morphism $M_K(\Pbf,Y)\ra M_{K'}(\Pbf',Y')$ of mixed Shimura varieties  whose evaluation over $\Cbb$-points is simply $[x,aK]\mapsto[f_*(x), f(a)K']$. It is actually defined over $E(\Pbf,Y)$. 

 For $(\Pbf,X)=(\Ubf,\Vbf)\rtimes(\Gbf,X)$, the natural projection $\pi:(\Pbf,Y)\ra(\Gbf,X)$ gives the natural projection onto the pure Shimura variety $$\pi:M_K(\Pbf,Y)\ra M_{\pi(K)}(\Gbf,X).$$ We can refine this projection into $$M_K(\Pbf,Y)\ot{\pi_\Ubf}\lra M_{\pi_\Ubf(K)}(\Pbf/\Ubf,Y/\Ubf(\Cbb))\ot{\pi_\Vbf} \lra M_{\pi(K)}(\Gbf,X)$$ as $\pi=\pi_\Wbf=\pi_\Vbf\circ\pi_\Ubf$, and the sequence means that a general mixed Shimura variety is fibred over some Kuga variety.

 (2) Let $(\Pbf,Y)$ be a mixed Shimura datum, and $g\in\Pbf(\adele)$. For $K\subset\Pbf(\adele)$ a \cosg, there exists a unique isomorphism of mixed Shimura varieties $\tau_g:M_{gKg^\inv}(\Pbf,Y)\ra M_K(\Pbf,Y)$ whose evaluation on $\Cbb$-points is $$\tau_g:[x,agKg^\inv]\mapsto[x,agK].$$  It is actually defined over $E(\Pbf,Y)$, and we call it the \emph{Hecke translation} associated to $g$. 
 
 Later in Section 4 we will use Hecke translation by $w\in\Wbf(\Qbb)$ to obtain isomorphisms between different pure sections of $M_K(\Pbf,Y)$ under suitable constraints on the level structures.
 \end{definition}
 
 We will show later in \ref{insensitivity of levels} that the Andr\'e-Oort conjecture is insensitive to the change of $K$ by smaller \cosgs, hence in this paper we will mainly work with levels $K$ that are neat, see \cite{pink thesis} Introduction (page 5). Mixed Shimura varieties at neat levels are smooth.
 
 We also introduce an auxiliary condition of the \cosg $K$:
 
 \begin{definition}[levels of product type]\label{levels of product type}
 
 Let $(\Pbf,Y)=(\Ubf,\Vbf)\rtimes(\Gbf,X)$ be a fibred mixed Shimura datum.
 
 (1) A \cosg $K$ of $\Pbf(\adele)$ is said to be of \emph{product type}, if it is of the form $K=K_\Wbf\rtimes K_\Gbf$ for \cosgs $K_\Wbf\subset\Wbf(\adele)$, $K_\Gbf\subset\Gbf(\adele)$, with $K_\Wbf$ the central extension of a \cosg $K_\Vbf\subset\Vbf(\adele)$ by a \cosg $K_\Ubf\subset\Ubf(\adele)$ through the restriction of $\psi$; $K_\Ubf$ and $K_\Vbf$ are required to be stabilized by $K_\Gbf$.

 (2) A  \cosg $K$ in $\Pbf(\adele)$ is said to be of \emph{fine product type} if \begin{itemize}
 
 \item (2-a)  it is of product type and $K=\prod_p K_p$ for \cosgs $K_p\subset\Pbf(\Qbbp)$ for any rational prime $p$, such that for some $\wp$ prime, $K_\wp$ is neat (hence $K$ is neat, and $K_{\Gbf,p}$ is neat); 
 
 \item (2-b) we also require that $K_\Gbf=K_{\Gbf^\der}K_\Cbf$ where $\Cbf$ is the connected center of $\Gbf$, with \cosgs $K_{\Gbf^\der}\subset\Gbf^\der(\adele)$ and $K_\Cbf\subset\Cbf(\adele)$ both of fine product type in the sense of (a).
 
 \item (2-c) there exist $\Zbb$-lattices $\Gamma_\Ubf\subset\Ubf(\Qbb)$ and $\Gamma_\Vbf\subset\Vbf(\Qbb)$ such that \begin{itemize}
 \item $\psi(\Gamma_\Vbf\times\Gamma_\Vbf)\subset\Gamma_\Ubf$, hence they generate a congruence subgroup $\Gamma_\Wbf$ in $\Wbf(\Qbb)$;
 
 \item $K_\Ubf$ resp. $K_\Vbf$ is the profinite completion of $\Gamma_\Ubf$ resp. of $\Gamma_\Vbf$, hence the same for $K_\Wbf$ \wrt $\Gamma_\Wbf$.
 
 \end{itemize}
 \end{itemize}
 In this case we also write $K_p=K_{\Wbf,p}\rtimes K_{\Gbf,p}$
 and $K_?=\prod_p K_{?,p}$ for $?\in\{\Ubf,\Vbf,\Wbf,\Gbf,\Pbf\}$.
 \end{definition} 
 
 \begin{remark}[two-step fibration, cf. \cite{pink thesis} Chapter 10]\label{two step fibration}
 If $K=K_\Wbf\rtimes K_\Gbf$, then $\pi(K)=K_\Gbf$ and we have an evident morphism $\iota(0): M_{K_\Gbf}(\Gbf,X)\mono M_K(\Pbf,Y)$, which we called the zero section of the (fibred) mixed Shimura variety defined by $(\Pbf,Y)=(\Ubf,\Vbf)\rtimes(\Gbf,X)$. The natural projection can be refined into $$ M_K(\Pbf,Y)\ot{\pi_\Ubf}\lra M_{K_\Vbf\rtimes K_\Gbf}(\Vbf\rtimes(\Gbf,X))\ot{\pi_\Vbf}\lra M_{K_\Gbf}(\Gbf,X)$$ where $\pi_\Vbf$ is an abelian scheme with zero section $\pi_\Ubf\circ \iota(0)$, and $\pi_\Ubf$ is a torsor under $\Gamma_\Ubf\bsh\Ubf(\Cbb)$. Since $\Ubf$ is commutative, $\Gamma_\Ubf$ is a $\Zbb$-lattice in the $\Qbb$-vector space $\Ubf(\Qbb)$, and $\Gamma_\Ubf\bsh\Ubf(\Cbb)\isom(\Cbb/\Zbb)^{\dim\Ubf}$ is an algebraic torus, whose character group is naturally identified with $\Gamma_\Ubf$.  
 
 \end{remark}

 \begin{example}[Data of Siegel type]\label{Siegel data}
 
 Let $\Vbf$ be a finite-dimensional $\Qbb$-vector space, equipped with a symplectic form $\psi:\Vbf\times\Vbf\ra\Ubf$ where $\Ubf=\Gaa$ is the one-dimensional rational Hodge structure of type $(-1,-1)$. 
 (1) We have the following data:\begin{enumerate}
 \item[(1-1)] pure Shimura datum of Siegel type: From the  $\Qbb$-group $\GSp_\Vbf=\GSp(\Vbf,\psi)$ of symplectic similitude, we obtain the pure Shimura datum $(\GSp_\Vbf,\Hscr_\Vbf)$, with $\Hscr_\Vbf=\Hscr_\Vbf^+\coprod\Hscr_\Vbf^-$ the Siegel double space associated to $(\Vbf,\psi)$. The pure Shimura varieties it defines are Siegel modular varieties (with suitable level structures). When $(\Vbf,\psi)$ is the standard symplectic structure on $\Qbb^{2g}$, it is often written as $(\GSp_{2g},\Hscr_g)$.
 
 \item[(1-2)] Kuga datum of Siegel type: For any $x\in\Hscr_\Vbf$, the standard representation $\rho_\Vbf:\GSp_\Vbf\ra\GL_\Vbf$ defines a rational Hodge structure $(\Vbf,\rho_\Vbf\circ x)$ of type $\{(-1,0),(0,-1)\}$, hence we get the Kuga datum $\Vbf\rtimes(\GSp_\Vbf,\Hscr_\Vbf)$, which we denote as $(\Qbf_\Vbf,\Vscr_\Vbf)$.
 
 \item[(1-3)] mixed Shimura datum of Siegel type: The symplectic form defines a central extension $\Wbf$ of $\Vbf$ by $\Ubf$, and it is easy to verify that $(\Pbf_\Vbf,\Uscr_\Vbf):=(\Ubf,\Vbf)\rtimes(\GSp_\Vbf,\Hscr_\Vbf)$ is a mixed Shimura datum fibred over the Siegel datum. Note that when $g=0$, the mixed Shimura datum is of the form $\Gbb_\arm\rtimes(\mult,\Hscr_0)$ with   and $\Hscr_0$ is a single point.
 \end{enumerate}
 
 Sometimes we also call Kuga data of Siegel type as mixed Shimura data of Siegel type.
 
 (2) The three classes of mixed Shimura data given above are all irreducible:\begin{itemize}
 
 \item[(2-1)] For $(\GSp_\Vbf,\Hscr_\Vbf)$, $\Hscr_\Vbf^+$ is the simple Hermitian symmetric domain corresponding to the connected simple Lie group $\Sp_\Vbf(\Rbb)$. If there is a pure subdatum $(\Hbf,X_\Hbf)$ with $X_\Hbf$ containing $\Hscr_\Vbf^+$, then $X_\Hbf$ is either $X_\Hbf^+:=\Hscr_\Vbf^+$ or $\Hscr_\Vbf$, and $X_\Hbf^+$ is a homogeneous space under $\Hbf^\der(\Rbb)^+$. Since $\Hbf^\der\subset\Sp_\Vbf$, the classification of simple Hermitian symmetric domains forces the equality $\Hbf^\der=\Sp_\Vbf$ because they both give rise to $\Hscr_\Vbf^+$. Note that the image of $\mult_\Rbb\subset\Sbb$ (corresponding to $\Rbb^\times\Cbb^\times$) under any $x\in\Hscr_\Vbf^+$ is the center $\mult_\Rbb$ of $\GL_{\Vbf,\Rbb}$, namely acting on $\Vbf_\Rbb$ by the central scaling. Hence $\Hbf\supset\mult$, i.e. $\Hbf$ contains the center $\mult$ of $\GSp_\Vbf$, which implies $\Hbf=\GSp_\Vbf$.
 
 \item[(2-2)] For $(\Qbf,\Vscr)=\Vbf\rtimes(\GSp_\Vbf,\Hscr_\Vbf)$, $\Vscr$ is the $\Vbf(\Rbb)$-orbit of $\Hscr_\Vbf$ in $\Xfrak(\Qbf)$. If $\Qbf'\subset\Qbf$ is a $\Qbb$-subgroup such that $y(\Sbb)\subset\Qbf'_\Rbb$ for all $y\in\Vscr$, then restricting to $x\in\Hscr_\Vbf\subset\Vscr$ we get $\Qbf'\supset\GSp_\Vbf$. The image of $\Qbf'$ in $\GSp_{\Vbf}$ is clearly equal to $\GSp_\Vbf$, and  the unipotent radical $\Vbf'$ of $\Qbf'$ is necessarily a $\Qbb$-subgroup of $\Vbf$. Thus a Levi decomposition of $\Qbf'$ over $\Qbb$ is of the form $\Vbf'\rtimes  \GSp_\Vbf$. Since the $\Vbf$ is irreducible as a representation of $\GSp_\Vbf$, we must have $\Vbf'=\Vbf$, which gives $\Qbf'=\Qbf$.
 
 Note that similar arguments show that $\Ubf\rtimes(\GSp_\Vbf,\Hscr)$ is an irreducible mixed Shimura datum, using the action of $\GSp_\Vbf$ on $\Ubf=\Gbb_\arm$ by the square of the central character.
 
 \item[(2-3)] For $(\Pbf,\Uscr)=(\Ubf,\Vbf)\rtimes(\GSp_\Vbf,\Hscr_\Vbf)$, $\Wbf$ is the extension of $\Vbf$ by $\Ubf=\Gbb_\arm$ using the symplectic form $\psi$. If $\Pbf'\subset\Pbf$ is a $\Qbb$-subgroup such that $y(\Sbb_\Cbb)\subset\Pbf'_\Rbb$, then when $y$ runs through points in $\Ubf(\Cbb)\rtimes\Hscr_\Vbf$, the irreducible subdatum $\Ubf\rtimes(\GSp_\Vbf,\Hscr_\Vbf)$ forces the inclusion $\Ubf\rtimes\GSp_\Vbf\subset\Pbf'$, and in particular $\GSp_\Vbf\subset\Pbf'$. Hence $\Pbf'$ admits a Levi decomposition over $\Qbb$ of the form $\Wbf'\rtimes\GSp_\Vbf$, where $\Wbf'$ is a unipotent $\Qbb$-subgroup of $\Wbf$ containing $\Ubf$. Thus $\Wbf'$ is an extension of $\Vbf'=\Wbf'/\Ubf$ by $\Ubf$ using the restriction of $\psi$. $\Wbf'$ and $\Ubf'$ being both stable under $\GSp_\Vbf$, we see that $\Vbf'$ is a $\GSp_\Vbf$-stable $\Qbb$-vector subspace of $\Vbf$, hence the equalities $\Vbf'=\Vbf$, $\Wbf'=\Wbf$, and $\Pbf'=\Pbf$, because $\Vbf$ is an irreducible representation of $\GSp_\Vbf$.
  
 \end{itemize}
 
 (3) Note that in a product of the form $(\Gbf,X)=(\GSp_{\Vbf_1},\Hscr_{\Vbf_1})\times\cdots\times(\GSp_{\Vbf_n},\Hscr_{\Vbf,n})$, 
 we can construct irreducible subdata of the form $(\Gbf',X')$ where $\Gbf'$ is the $\Qbb$-subgroup generated by $\Gbf^\der=\prod_j\Sp_{\Vbf_j}$ plus a split $\Qbb$-torus $\mult$ that acts on each $\Vbf_j$ by the central scaling, and $X'$ is the $\Gbf'(\Rbb)$-orbit of some $x=(x_1,\cdots,x_n)\in X=\prod_j\Hscr_{\Vbf_j}$. In fact $x$ sends $\Sbb^1$ into $x_1(\Sbb^1)\times\cdots\times x_n(\Sbb^1)\in\prod_j\Sp_{\Vbf_j,\Rbb}$, and sends $\mult_\Rbb$ to the center of $\GL_{\oplus_j\Vbf_j}$. Hence by \ref{generating subdata}, $(\Gbf',\Gbf'(\Rbb)x)$ is a subdatum of $(\Gbf,X)$, which is clearly irreducible. However $\Gbf(\Rbb)$ has only two connected components, as its center $\mult(\Rbb)$ has only two connected components and $\Gbf'^\der(\Rbb)=\prod_j\Sp_{\Vbf_j}(\Rbb)$ is connected, and thus $\Gbf'(\Rbb)x$ has only two connected components, while $X=\prod_j\Hscr_{\Vbf_j}$ has $2^n$ connected components. 
 
 For a point $x\in X$, we have its signature vector $s_x\in(\pm)^n$, describing whether it is positive or negative definite on $\Vbf_j$, $j=1,\cdots,n$. Two points $x=(x_j)$ and $x'=(x'_j)$ in $X$ fall in the same connected component \ifof they have the same signature  on each $\Vbf_j$. Hence $x'\in\Gbf'(\Rbb)x$ \ifof $s_{x'}=\pm s_x$. When $x$ runs through $X$, we get finitely many irreducible subdata of the form $(\Gbf',\Gbf'(\Rbb)x)$, which follows from \ref{common Mumford-Tate group}. It is also clear that the generic Mumford-Tate group of each connected component of $X$ is $\Gbf'$.  Although we mainly work with Shimura data in the sense of Deligne, we mention that the pair $(\Gbf',X=\prod_j\Hscr_{\Vbf_j})$ is a pure Shimura datum in the sense of Pink \cite{pink thesis} 2.1, as each $x\in X$ gives a homomorphism $\Sbb\ra\Gbf'_\Rbb$, and $X\ra\Xfrak(\Gbf')$ is a $\Gbf'(\Rbb)$-equivariant map with finite fibers.

 Similarly, in the Kuga case, $(\Pbf,Y)=\prod_j(\Qbf_j\Vscr_j)=(\oplus_j\Vbf_j)\rtimes(\Gbf,X)$ admits irreducible subdata of the form $(\Pbf',Y')=(\oplus_j\Vbf_j)\rtimes(\Gbf',X')$. The general  mixed Shimura case of Siegel type is parallel. 

These irreducible subdata are actually strictly irreducible in the sense of \ref{special subvarieties}(3). The fact that the connected center is simply a split $\Qbb$-torus $\mult$ will be useful in the estimations \ref{torsion order in the product case} and \ref{torsion order in the embedded case}.
 
 (4) Assume that for some $\Zbb$-lattice $\Gamma_\Vbf$ of $\Vbf$, the restriction $\psi:\Gamma_\Vbf\times\Gamma_\Vbf\ra\Qbb(-1)$ has value in $\Zbb(-1)=\Zbb(2\pi\ibf)^\inv$ and is of discriminant $\pm1$. The profinite completions of lattices $\Gamma_\Vbf\subset\Vbf$ and $\Zbb(-1)\subset\Qbb(-1)$ are \cosgs $K_\Vbf$ and $K_\Ubf$ respectively. Take a \cosg $K_\Gbf\subset\GSp_\Vbf(\adele)$ small enough and stabilizing both  $K_\Vbf$ and $K_\Ubf$, we get the mixed Shimura variety $M_K(\Pbf_\Vbf,\Uscr_\Vbf)$ for  $K=K_\Wbf\rtimes K_\Gbf$, $K_\Wbf$ being the \cosg generated by $K_\Ubf$ and $K_\Vbf$. We also have the universal abelian scheme over the Siegel moduli space of level $K_\Gbf$, namely $$\Ascr=M_{K_\Vbf\rtimes K_\Gbf}(\Qbf_\Vbf,\Vscr_\Vbf)\ra \Sscr=M_{K_\Gbf}(\GSp_\Vbf,\Hscr_\Vbf)$$ and $M_K(\Pbf_\Vbf,\Uscr_\Vbf)$ is a $\mult$-torsor over $\Ascr$.
 
 The \cosgs thus obtained are  levels of fine product type when $K_\Gbf$ is of fine product type.
 
 \end{example}

 \begin{definition}[special subvarieties]\label{special subvarieties}
 
 Let $(\Pbf,Y)$ be a mixed Shimura datum, with $M=M_K(\Pbf,Y)$ a mixed Shimura variety associated with it.
 
 (1) The map $\wp_\Pbf:Y\times\Pbf(\adele)/K\ra M(\Cbb),\ (y,aK)\mapsto [y,aK]$ is called the (complex) \emph{uniformization map} of $M$. It is clear that the source is not connected in general, but its connected components are simply connected complex manifolds isomorphic to each other.
 
 A \emph{special subvariety} of $M_K(\Pbf,Y)$ is a priori a subset of $M(\Cbb)$ of the form $\wp_\Pbf( Y'^+\times aK)$ with $a\in\Pbf(\adele)$ and $Y'^+$  a connected component of some mixed Shimura subdatum $(\Pbf',Y')\subset(\Pbf,Y)$, where $K'\subset\Pbf'(\adele)$ is the \cosg $aKa^\inv\cap\Pbf'(\adele)$.
 
 A special subvariety is actually a closed algebraic subvariety of $M_K(\Pbf,Y)$ over $\Qac$: it is a connected component of the image of the morphism $M_{K'}(\Pbf',Y')\ra M_{aKa^\inv}(\Pbf,Y)$ under the Hecke translate $M_{aKa^\inv}(\Pbf,Y)\isom M_K(\Pbf,Y)$. 
 
 (2) In Section 2 and 3, we will often work with connected mixed Shimura varieties defined as follows:
 
 \begin{itemize}
 
 \item a \emph{connected mixed Shimura datum} is of the form $(\Pbf,Y;Y^+)$ where $(\Pbf,Y)$ is a mixed Shimura datum and $Y^+$ a connected component of $Y$; a \emph{morphism} between connected mixed Shimura data is $f:(\Pbf_1,Y_1;Y_1^+)\ra(\Pbf_2,Y_2;Y_2^+)$ with $f$ a morphism of mixed Shimura data $(\Pbf_1,Y_1)\ra(\Pbf_2,Y_2)$ sending $Y_1^+$ into $Y_2^+$; in particular, a \emph{connected mixed Shimura subdatum} is of the form $(\Pbf',Y';Y'^+)\subset(\Pbf,Y;Y^+)$ with $(\Pbf',Y')$ a subdatum of $(\Pbf,Y)$ and $Y'^+$ a connected component of $Y'$ contained in $Y^+$; 
 
 \item a \emph{connected mixed Shimura variety} is a quotient space of the form $M^+=\Gamma\bsh Y^+$ where $\Gamma\subset\Pbf(\Qbb)_+$ is a congruence subgroup; such quotients are normal quasi-projective algebraic varieties defined over a finite extension of the reflex field of $(\Pbf,Y)$, and we treat them as varieties over $\Qac$;
 
 \item for a connected mixed Shimura variety $M^+$ as above we have the (complex) uniformization map $\wp_\Gamma: Y^+\ra M^+$ $y\mapsto\Gamma y$, and a special subvariety of $M^+$ is a subset of the form $\wp_\Gamma(Y'^+)$ given by some connected mixed Shimura subdatum $(\Pbf',Y';Y'^+)$; special subvarieties are closed irreducible algebraic subvarieties defined over $\Qac$, with  canonical models defined over some number fields.
 
 \end{itemize}

 For example, in the Kuga case $(\Pbf,Y)=\Vbf\rtimes(\Gbf,X)$,  we have explained in   Introduction that special subvarieties are certain torsion subschemes of abelian schemes over some pure special subvariety $S'\subset S$.  
 
 (3) We also introduce a variant of irreducible data in the connected setting. A connected mixed Shimura data $(\Pbf,Y;Y^+)$ is said to be strictly irreducible if $\Pbf=\MT(Y^+)$. Note that in this case $Y=\Pbf(\Rbb)\Ubf(\Cbb)Y^+$ is determined by $\Pbf$ and $Y^+$, and $(\Pbf,Y)$ is necessarily irreducible. The pair $(\Pbf,Y)$ thus obtained is also said to be strictly irreducible.
 
 For example, the pure and mixed Shimura data of Siegel type associated to a symplectic $\Qbb$-space $(\Vbf,\psi)$ gives rise to strictly irreducible connected mixed Shimura data $(\Pbf,Y;Y^+)$, and the data $(\Gbf',X')$ constructed in \ref{Siegel data}(3) gives rise to strictly irreducible ones $(\Gbf',X';X'^+)$.
 
 \end{definition}
 
 \begin{lemma}[strict irreducibility]\label{strict irreducibility}(1) Let $f:(\Pbf,Y;Y^+)\ra(\Pbf',Y';Y'^+)$ be a morphism of connected mixed Shimura data, such that $f:\Pbf\ra\Pbf'$ is surjective. Assume that $(\Pbf_1,Y_1;Y_1^+)$ is a strictly irreducible connected subdatum of $(\Pbf,Y;Y^+)$. Then its image $(\Pbf_1',Y_1';Y_1'^+)$ in $(\Pbf',Y';Y'^+)$ remains strictly irreducible.
 
 (2) Let $(\Pbf,Y;Y^+)=\Wbf\rtimes(\Gbf,X;X^+)$ be a connected mixed Shimura datum fibred over a connected pure Shimura datum $(\Gbf,X;X^+)$. Then $(\Pbf,Y;Y^+)$ is strictly irreducible \ifof so it is with $(\Gbf,X;X^+)$.
 \end{lemma}
 
 \begin{proof}
 (1) Since $f:\Pbf\ra\Pbf'$ is surjective, the map $f_*:Y^+\ra Y'^+$ is also surjective. Assume that $(\Pbf_1',Y_1';Y_1'^+)$ is not strictly irreducible. Then there exists $\Pbf_2'\subsetneq\Pbf_1'$ such that $\Pbf_2'=\MT(Y_1'^+)$, and putting $Y_2'^+=Y_1'^+$ plus $Y_2'=\Pbf_2'(\Rbb)\Ubf_2'(\Cbb)Y_2'^+$, we get a strictly irreducible subdatum $(\Pbf_2',Y_2';Y_2'^+)$. We have $\Pbf_2:=f^\inv(\Pbf_2)\subsetneq\Pbf_1$ by the epimorphism $f:\Pbf_1\epim\Pbf_1'$, and clearly the inclusion $\Pbf_{2,\Cbb}\supset y(\Sbb_\Cbb)$ holds  for all $y\in Y_1^+$ because $f(\Pbf_{2,\Cbb})\supset f(y)(\Sbb_\Cbb)$, contradicting the strict irreducibility of $(\Pbf_1,Y_1;Y_1^+)$.
 
 (2) The reduction modulo $\Wbf$ gives $\pi:(\Pbf,Y;Y^+)\ra(\Gbf,X;X^+)$ with $\pi:\Pbf\ra\Gbf$ surjective. If $(\Pbf,Y;Y^+)$ is strictly irreducible, then so it is with its image $(\Gbf,X;X^+)$ by (1).

 Conversely, assume $(\Gbf,X;X^+)$ is strictly irreducible, then $(\Gbf,X)$ is irreducible, and we get $(\Pbf,Y)$ irreducible by \ref{pure irreducibility}. Let $\Pbf'\subset\Pbf$ be the generic Mumford-Tate group of $Y^+$, then $\Pbf'^\der=\Pbf^\der=\Wbf\rtimes\Gbf^\der$ by \ref{unipotent radical and Levi decomposition}(4) and \ref{generating subdata}(3). The reduction of $\Pbf'$ modulo $\Wbf$ is a $\Qbb$-subgroup $\Gbf'$ of $\Gbf$ such that $\Gbf'_\Rbb\supset x(\Sbb)$ for all $x\in X^+$, hence $\Gbf'=\Gbf$ and $\Pbf'=\Pbf$.\end{proof}

 The Andr\'e-Oort conjecture can be reduced to the case of special subvarieties within a connected mixed Shimura variety, and it suffices to prove it for some level structure sufficiently small, due to the following elementary lemma:
 
 \begin{lemma}[insensitivity of levels]\label{insensitivity of levels} Let $(\Pbf,Y;Y^+)$ be a connected mixed Shimura data, giving rise to a morphism of connected mixed Shimura varieties  $\pi:M=\Gamma\bsh Y^+\ra M'=\Gamma'\bsh Y^+$ via an inclusion of congruence subgroups  $\Gamma\subset\Gamma'$ in $\Pbf(\Rbb)_+$. 
 
 (1) If $S\subset M$ is a special subvariety, then its image $\pi(S)$ in $M'$ is special; conversely, if
 $S'\subset M'$ is a special subvariety, then the pre-image $\pi^\inv(S')$ in $M$ is a finite union of special subvarieties.
 
 (2) The Andr\'e-Oort conjecture holds for $M$ \ifof it holds for $M'$.
 \end{lemma}
 
 \begin{proof}
 (1) If $S=\wp_\Gamma(Y_1^+)$ is a special subvariety defined by some conncted subdatum $(\Pbf_1,Y_1;Y_1^+)$, then $\pi(S)=\wp_{\Gamma'}(Y_1^+)$ is special.
 
 Conversely, if $\Gamma'=\coprod\Gamma a_j$ is a decomposition into finitely many cosets, and $S'=\wp_{\Gamma'}(Y_1^+)$ is special given by a connected mixed Shimura subdatum $(\Pbf_1,Y_1;Y_1^+)$, then $\pi^\inv(S')$ is the finite union of special subvarieties $S_j=\wp_{\Gamma}(a_jY_1^+)$ defined by $(a_j\Pbf_1a_j^\inv,a_jY_1;a_jY_1^+)$.
 
 (2) If the Andr\'e-Oort conjecture holds for $M$, and $(S_n')$ is a sequence of special subvarieties in $M'$, then the pre-images $\pi^\inv(S_n')$ form a sequence of special subvarieties, whose Zariski closure is a finite union $\bigcup_jZ_j$ of special subvarieties in $M$. The finite map $\pi:M\ra M'$ is closed, hence it sends $\bigcup_jZ_j$ to the Zariski closure of $\bigcup_nS_n'$ and it is a finite union of special subvarieties $\bigcup_j\pi(Z_j)$.
 
 Conversely, if a geometrically irreducible subvariety $Z\subset M$ is the Zariski closure of a sequence of special subvarieties $\bigcup_nS_n$, then $Z$ is a geometrically irreducible component of $\pi^\inv(\pi(Z))$, hence special because $\pi(Z)$ is the Zariski closure of a sequence of special subvarieties $(\pi(S_n))$ in $M'$, and $\pi^\inv(\pi(Z))$ is a finite union of special subvarieties by (1).\end{proof}

 \begin{remark}[arithmetic quotients]\label{congruence subgroups} Although we have used congruence subgroups in $\Pbf(\Qbb)^+$ to define connected mixed Shimura varieties of the form $\Gamma\bsh Y^+$ and their special subvarieties, it makes no harm to use arithmetic subgroups of $\Pbf^\der(\Qbb)^+$, as long as we only treat them as complex algebraic varieties. In this setting we can also define special subvarieties and formulate the Andr\'e-Oort conjecture (see (3) below). We assume for simplicity that the arithmetic subgroups involved are torsion-free, since this suffices for the study of Andr\'e-Oort type conjectures following the idea of \ref{insensitivity of levels}. 
 
 By \cite{baily borel} and \cite{pink thesis}, quotients of the form $\Gamma\bsh Y^+$ with $\Gamma$ torsion-free arithmetic subgroups of $\Pbf^\der(\Qbb)^+$ are normal quasi-projective algebraic varieties over $\Cbb$, which are also smooth.   In fact:
 
 (1) In the pure case $(\Pbf,Y;Y^+)=(\Gbf,X;X^+)$, the Lie group $\Gbf(\Rbb)_+$ acts on $X^+$ through $\Gbf^\ad(\Rbb)^+$. If $\Gamma\subset\Gbf^\der(\Qbb)^+$ is an arithmetic subgroup, then using the quotient map $\Gbf^\der\ra\Gbf^\ad$, its image $\Gamma^\ad\subset\Gbf^\ad(\Qbb)^+$ is arithmetic by \cite{borel arithmetic} 8.9 and 8.11, and the quotient $\Gamma\bsh X^+=\Gamma^\ad\bsh X^+$ is an algebraic variety. 
 
 If $\Gamma_\Gbf$ is a torsion-free arithmetic subgroup in $\Gbf(\Qbb)_+$, then $\Gamma_\Gbf^\dag:=\Gamma_\Gbf\cap\Gbf^\der(\Qbb)^+$ is a torsion-free arithmetic subgroup of $\Gbf^\der(\Qbb)^+$, and $\Gamma_\Gbf$ acts on $X^+$ through its image $\Gamma_\Gbf^\ad\subset\Gbf^\ad(\Qbb)^+$, which is an arithmetic subgroup of $\Gbf^\ad(\Qbb)^+$ again by \cite{borel arithmetic} using $\Gbf\epim\Gbf^\ad$. Again because we only focus on Andr\'e-Oort type conjectures, following the idea of \ref{insensitivity of levels}  we only consider the case when $\Gamma_\Gbf^\ad$ is torsion-free. The group $\Gbf^\der(\Rbb)^+$ also acts on $X^+$ through $\Gbf^\ad(\Rbb)^+$, and the evident map $\Gamma_\Gbf^\dag\bsh X^+\ra \Gamma_\Gbf\bsh X^+$ is the same as $\Gamma'\bsh X^+\ra \Gamma_\Gbf^\ad\bsh X^+$, where $\Gamma'$ is the image of $\Gamma_\Gbf^\dag$ in $\Gbf^\ad(\Qbb)^+$. Clearly $\Gamma'$ is also an arithmetic subgroup, contained in $\Gamma_\Gbf^\ad$ as a subgroup of finite index, and it is torsion-free. Hence $\Gamma_\Gbf^\dag\bsh X^+\ra \Gamma_\Gbf\bsh X^+$ is a finite morphism between algebraic varieties: since these quotients are given by torsion-free arithmetic subgroups, by \cite{borel metric properties} 3.10 we deduce that $\Gamma'\bsh X^+\ra \Gamma_\Gbf^\ad\bsh X^+$ is algebraic, and in fact it is a finite \'etale covering using Riemann existence theorem, cf. \cite{gille szamuely} 5.7.4.
 

 (2) In the mixed case, the connected mixed Shimura varieties of interest in our study are often given in a product form, i.e. given as $\Gamma\bsh Y^+$ using connected data $(\Pbf,Y;Y^+)=\Wbf\rtimes(\Gbf,X;X^+)$ with   $\Gamma=\Gamma_\Wbf\rtimes \Gamma_\Gbf$ for arithmetic subgroups $\Gamma_\Wbf\subset\Wbf(\Qbb)$ and $\Gamma_\Gbf\subset\Gbf(\Qbb)_+$. Write $\Gamma_\Gbf^\dag=\Gamma_\Gbf\cap\Gbf^\der(\Qbb)^+$ and $\Gamma^\dag=\Gamma_\Wbf\rtimes\Gamma_\Gbf^\dag$, we have $\Gamma^\dag=\Gamma\cap\Pbf^\der(\Qbb)^+$ as an arithmetic subgroup of $\Pbf^\der(\Qbb)^+$ using $\Pbf^\der=\Wbf\rtimes\Gbf^\der$. Consider the following diagrams: $$\xymatrix{\Gamma^\dag \ar[r]\ar[d] & \Gamma \ar[d] & & \Gamma^\dag\bsh Y^+\ar[r]\ar[d] & \Gamma\bsh Y^+\ar[d]    \\ \Gamma_\Gbf^\dag\ar[r] & \Gamma_\Gbf & & \Gamma_\Gbf^\dag\bsh X^+ \ar[r] & \Gamma_\Gbf\bsh X^+}$$   where the left one is Cartesian whose arrows are inclusions, and the right one is Cartesian in the category of complex analytic spaces using the group actions from the left one on $Y^+\ra X^+$. Hence $\Gamma^\dag \bsh Y^+\ra \Gamma\bsh Y^+$ is a finite morphism between algebraic varieties as it is pulled back from the finite morphism on the bottom. One may also talk about $\Gamma\bsh Y^+$ for general (torsion-free) arithmetic subgroups of $\Pbf^\der(\Qbb)^+$, because such subgroups contain normal subgroups of finite index of the product form above, and we may argue by finite group quotients and Riemann existence theorem.
 
 Note that we cannot define mixed Shimura data of the form $(\Pbf,Y)=\Wbf\rtimes(\Gbf,X)$ with $\Gbf$ of adjoint type, because in this case, the Hodge structure given by $x\in X$ on any algebraic representation $\Vbf$ of $\Gbf$ is necessarily of weight zero as the center of $\Gbf$ is trivial. Hence in the diagram above it would not make sense to write groups like $\Gamma_\Wbf\rtimes\Gamma_\Gbf^\ad$.
 
 (3) The Andr\'e-Oort conjecture for special subvarieties still make sense in this setting  which only involves complex algebraic varieties. Namely we start with an arithmetic subgroup $\Gamma\subset\Pbf^\der(\Qbb)^+$, form the uniformization map $\wp_\Gamma:Y^+\ra\Gamma\bsh Y^+$, and define special subvarieties to be of the form $\wp_\Gamma(Y'^+)$ using connected mixed Shimura subdata $(\Pbf',Y';Y'^+)\subset(\Pbf,Y;Y^+)$. If $\Gamma\subset\Pbf(\Qbb)_+$ is a torsion-free congruence subgroup, and $\Gamma'\subset\Pbf^\der(\Qbb)^+$ is an arithmetic subgroup contained in $\Gamma$, then using the arguments through the universal coverings as in \ref{insensitivity of levels} and the finite morphism $\Gamma'\bsh Y^+\ra \Gamma\bsh Y^+$, we see that the Andr\'e-Oort conjecture for $\Gamma\bsh Y^+$ is equivalent to the one for $\Gamma'\bsh Y^+$.
 
 
  When we need finer information about the canonical models, we will always work with classical mixed Shimura data with  level structures  given by compact open subgroups of $\Pbf(\adele)$.

 \end{remark}

 The remark above draws our attention to the derived groups. Analogue to modifying congruence subgroups, we can also modify the derived groups:

 \begin{lemma}[insensitivity of isogeny]\label{insensitivity of isogeny}
 Let $f:M\ra M'$ be a morphism between connected mixed Shimura varieties given by some morphism of connected mixed Shimura data $(f,f_*):(\Pbf,Y;Y^+)\ra(\Pbf',Y';Y'^+)$ together with congruence subgroups $\Gamma\subset\Pbf(\Qbb)_+$ and $\Gamma'\subset\Pbf'(\Qbb)_+$ satisfying $f(\Gamma)\subset\Gamma'$. Assume that the $\Qbb$-group homomorphism $f:\Pbf^\der\ra\Pbf'^\der$ is an isogeny, i.e. surjective of finite kernel. Then
 
 (1) $f_*:Y^+\ra Y'^+$ is an isomorphism and $f:M\ra M'$ is finite;
 
 (2)  the Andr\'e-Oort conjecture holds for $M$ \ifof it holds for $M'$.
 \end{lemma}
 
 Note that (1) implies that a connected mixed Shimura variety could be realized by different connected mixed Shimura data, and (2) affirms that one might choose any mixed Shimura datum to study the Andr\'e-Oort conjecture. 
 For example, in the pure case, the evident morphism between pure Shimura data $(\Gbf,X)\ra(\Gbf^\ad,X^\ad)=(\Gbf,X)/\Zbf$ with $\Zbf$ the center of $\Gbf$ satisfies the lemma: $\Gbf^\der\ra\Gbf^\ad$ is an isogeny. Hence the Andr\'e-Oort conjecture for pure Shimura varieties is reduced to the case where the ambient Shimura variety is defined by a $\Qbb$-group of adjoint type, which has been used in \cite{klingler yafaev} \cite{ullmo yafaev} etc.
 \begin{proof}
 (1) By \ref{fibred mixed Shimura data new}, we may assume that $(\Pbf,Y)=(\Ubf,\Vbf)\rtimes(\Gbf,X)$, $(\Pbf',Y')=(\Ubf',\Vbf')\rtimes(\Gbf',X')$, and $(f,f_*)$ is given by $(f,f_*):(\Gbf,X)\ra(\Gbf',X')$ a morphism of pure Shimura data together with equivariant maps between $\Qbb$-vector spaces $f:\Ubf\ra\Ubf'$ and $f:\Vbf\ra\Vbf'$ compatible with the alternating bilinear maps $\psi$ and $\psi'$. When we regard $\Gbf$ as a $\Qbb$-subgroup of $\Pbf$, $f(\Gbf)\subset\Pbf'$ is a reductive $\Qbb$-subgroup, which extends to some maximal $\Qbb$-subgroup $w'\Gbf'w'^\inv$ for some $w'\in\Wbf'(\Qbb)$, and we assume for simplicity that $w'=1$, hence the morphism $(\Gbf,X)\ra(\Gbf',X')$ extends to $(\Pbf,Y)\ra(\Pbf',Y')$. 
 
 Since $f:\Pbf^\der\ra\Pbf'^\der$ is an isogeny, $f$ induces an isomorphism of Lie algebras $\Lie\Pbf^\der\ra\Lie\Pbf'^\der$, and thus $\Lie\Wbf\isom\Lie\Wbf'$ under $f$. Taking $y\in Y$ mapped to $y'=f(y)\in Y'$, we see that the mixed Hodge structures   on $\Lie\Wbf$ and on $\Lie\Wbf'$ are isomorphic. This forces the $\Qbb$-group homomorphism $\Ubf\ra\Ubf'$ to be an isomorphism because it underlies an isomorphism of rational Hodge structures of type $(-1,-1)$, and so it is with the quotient $\Vbf\ra\Vbf'$, hence $f:\Wbf\ra\Wbf'$ is also an isomorphism. As for the pure part, the isogeny $\Pbf^\der\ra\Pbf'^\der$   gives an isogeny   $\Gbf^\der\ra\Gbf'^\der$ which gives further an isomorphism $\Gbf^\ad\ra\Gbf'^\ad$,  hence   the
 map $f_*:X\ra X'$ is an isomorphism on each connected components in $X$. Therefore $f_*:Y^+=\Wbf(\Rbb)\Ubf(\Cbb)\rtimes X^+\ra Y'^+=\Wbf'(\Rbb)\Ubf'(\Cbb)\rtimes X'^+$ is an isomorphism. 
 
 Assume that $f:\Gamma\bsh Y^+\ra\Gamma'\bsh Y^+$ is associated to congruence subgroups $\Gamma$ and $\Gamma'$ respectively with $f(\Gamma)\subset\Gamma'$, then the corresponding map for the pure quotient $f:\Gamma_\Gbf\bsh X^+\ra\Gamma_{\Gbf'}\bsh X'^+$ is finite. In fact, the image $\Gamma_\Gbf$ of $\Gamma$ in $\Gbf(\Qbb)_+$ resp. $\Gamma_{\Gbf'}$ of $\Gamma'$ in $\Gbf'(\Qbb)_+$ acts on $X^+$ through $\Gbf^\ad(\Qbb)^+$ resp. on $X'^+$ through $\Gbf'^\ad(\Qbb)^+$, and from the isomorphism $\Gbf^\ad\isom\Gbf'^\ad$ we know that the image of $\Gamma_\Gbf$ in $\Gbf^\ad(\Qbb)^+$ is n arithmetic subgroup of finite index in the image of $\Gamma_{\Gbf'}$ in $\Gbf'^\ad(\Qbb)^+$. Since $\Wbf\isom\Wbf'$, we also get $\Gamma\cap\Wbf(\Qbb)$ of finite index in $\Gamma'\cap\Wbf'(\Qbb)$. Hence  $f:\Gamma\bsh Y^+\ra\Gamma'\bsh Y'^+$ is finite.
 
 (2) If $S\subset \Gamma\bsh Y^+$ is a special subvariety defined by $(\Pbf_1,Y_1;Y_1^+)$, then its image in $\Gamma'\bsh Y'^+$ is a special subvariety defined by the image subdatum $(f(\Pbf_1),f_*(Y_1);f_*(Y_1^+))$. Conversely, if $S'$ is a special subvariety of $\Gamma'\bsh Y'^+$ defined by $(\Pbf_1',Y_1';Y_1'^+)$, then we claim that its preimage in $\Gamma\bsh Y^+$ is a finite union of special subvarieties which are Hecke translates of a connected subdatum of the form $(\Pbf_1,Y_1;Y_1^+)$ where \begin{itemize}
 \item $\Pbf_1$ is the neutral component of $f^\inv(\Pbf_1')$;
 \item $Y_1^+=f_*^\inv(Y_1'^+)$ under the isomorphism $f_*:Y^+\ra Y'^+$;
 \item $Y_1$ is the $\Pbf_1(\Rbb)\Ubf_1(\Cbb)$-orbit of $Y_1^+$ with $\Ubf_1=\Ubf\cap\Pbf_1$. 
 \end{itemize}

 In fact, any $y'\in Y_1'^+$ has a unique preimage $y\in Y^+$ under the isomorphism $f_*:Y^+\ra Y'^+$, and the inclusion $y'(\Sbb_\Cbb)\subset\Pbf'_{1,\Cbb}$ gives $y(\Sbb_\Cbb)\subset \Pbf_{1,\Cbb}$ because $y(\Sbb_\Cbb)$ is necessarily connected.  We verify that the $\Qbb$-group $\Pbf_1$ satisfies the condition in \ref{generating subdata}: $\Pbf_1$ is the neutral component of $f^\inv(\Pbf_1')$, hence its image $\Gbf_1$ modulo $\Wbf$ in $\Gbf$ is the neutral component of $f^\inv(\Gbf_1')$ with $\Gbf_1'$ the reduction of $\Pbf_1'$ modulo $\Wbf'$ in $\Gbf'$. Since the image of $(\Pbf_1',Y_1')$ in $(\Gbf',X')$ is a pure Shimura subdatum of the form $(\Gbf_1',X_1')$, we see that $\Gbf_1'^\ad$ admits no compact $\Qbb$-factors, hence so it is with $\Gbf_1$.
 
 By \ref{generating subdata}, $\Pbf_1$ does give rise to a mixed Shimura subdatum $(\Pbf_1,\Ubf_1,Y_1)$ of $(\Pbf,\Ubf,Y)$ using the element $y$, with $\Ubf_1=\Ubf\cap\Pbf_1$ and $Y_1$ being the $\Pbf_1(\Rbb)\Ubf_1(\Cbb)$-orbit of $y$. The image of $(\Pbf_1,\Ubf_1,Y_1)$ in $(\Pbf,\Ubf',Y')$ is clearly contained in $(\Pbf_1',\Ubf_1',Y_1')$. We have $\Pbf_1^\der\subset\Pbf^\der$, and the evident homomorphism $f:\Pbf_1^\der\ra\Pbf_1'^\der$ is an isogeny because it is restricted from $f:\Pbf^\der\ra\Pbf'^\der$. By the arguments in (1) we have isomorphisms $f:\Ubf_1\ra\Ubf_1'$, $f:\Wbf_1\ra\Wbf_1'$, and $f_*:Y_1^+\ra Y_1'^+$. 
 
 Using \ref{congruence subgroups}(3), the Andr\'e-Oort conjecture for $\Gamma\bsh Y^+$ resp. for $\Gamma\bsh Y'^+$ is equivalent to the conjecture for $\Gamma^\dagger\bsh Y^+$ resp. for $\Gamma'^\dagger\bsh Y'^+$. The isogeny $f:\Pbf^\der\ra\Pbf'^\der$ induces an isomorphism $Y^+\ra Y'^+$ and it sends $\Gamma^\dagger$ into an arithmetic subgroup of $\Gamma'^\dagger$ with finite kernel, hence the arguments in (1) show that the Andr\'e-Oort conjecture is the same for $\Gamma^\dagger\bsh Y^+$ and $\Gamma'^\dagger\bsh Y'^+$, hence the same for $\Gamma\bsh Y^+$ and $\Gamma'\bsh Y'^+$.\end{proof}

  In the rest of this section, we explain how a general mixed Shimura datum could be embedded, up to isogeny, into the product of a pure Shimura datum with Kuga data and mixed Shimura data defined in \ref{Siegel data}. This phenomenon is similar to the reduction lemma in \cite{pink thesis} 2.26, but the proof is easier as we only require embedding up to isogeny. 
  
  \begin{lemma}[reduction lemma]\label{reduction lemma}
  
  Let $(\Pbf,Y)$ be an irreducible mixed Shimura datum. Then there exists a morphism of mixed Shimura data $$f:(\Pbf,Y)\ra(\Gbf_0,X_0)\times\prod_{i=1,\cdots,r}(\Pbf_i,\Uscr_i)\times (\Qbf_0,\Vscr_0)$$ where $(\Gbf_0,X_0)$ is a pure Shimura datum, $(\Pbf_i,\Uscr_i)=(\Pbf_{\Vbf_i},\Uscr_{\Vbf_i})$, $(\Qbf_0,\Vscr_0)=(\Qbf_{\Vbf_0},\Vscr_{\Vbf_0})$ defined as in \ref{Siegel data} for some symplectic spaces $\Vbf_i$ ($i=0,1,\cdots,r$), and that the $\Qbb$-group homomorphism $\Pbf\ra\Gbf\times\prod_i\Pbf_i\times \Qbf_0$ is of finite kernel.

  If moreover $(\Pbf,Y)$ is a Kuga datum, then $f$ can be taken of the form $$(\Pbf,Y)=\Vbf\rtimes(\Gbf,X)\ra(\Gbf_0,X_0)\times(\Qbf_0,\Vscr_0)=(\Gbf_0,X_0)\times (\Vbf\rtimes(\GSp_\Vbf, \Hscr_\Vbf)).$$
  \end{lemma}
  
  \begin{proof}
  Write $(\Pbf,Y)=\Wbf\rtimes(\Gbf,X)=(\Ubf,\Vbf)\rtimes(\Gbf,X)$ with $\Wbf$ given by an alternating bilinear map $\psi:\Vbf\times\Vbf\ra\Ubf$ equivariant under $\Gbf$. 
  
  (1) We first consider the Kuga case, i.e. $\psi=0$ and $\Ubf=0$. We claim in this case $\Gbf$ preserves a symplectic form $\Psi$ on $\Vbf$ up to similitude, and it gives rise to a morphism of Kuga data $(\Pbf,Y)\ra\Vbf\rtimes(\GSp_\Vbf,\Hscr_\Vbf)$.
  
  Using \cite{pink thesis} 1.4, we get a variation of rational Hodge structures $\Vbb$ on $X$ associated to the representation $\rho_\Vbf$, which is pure of weight 1 and type $\{(-1,0),(0,-1)\}$, and the Hodge structure on $\Vbb_x$ is given by $\rho_\Vbf\circ x$; using further \cite{pink thesis} 1.12, this variation is polarized by some $\Psi:\Vbb\otimes_{\Qbb_X}\Vbb\ra\Qbb(1)_X$, coming from some $\Gbf$-equivariant non-degenerate bilinear map $\Psi:\Vbf\otimes_\Qbb\Vbf\ra\Gbb_\arm(\isom\Qbb(1))$, which is a symplectic form due to the weight and the Hodge type of $\rho_\Vbf\circ x$ ($x\in X$). Hence we obtain a map $f_*:X\ra \Hscr_\Vbf$ sending $x\in X$ to the polarized rational Hodge structure $(\rho_\Vbf\circ x,\Psi)$, and $f_*$ is equivariant \wrt a $\Qbb$-group homomorphism $f:\Gbf\ra\GSp_\Vbf$. Using \ref{fibred mixed Shimura data new}(2), we get a morphism of Kuga data $f:\Vbf\rtimes(\Gbf,X)\ra\Vbf\rtimes(\GSp_\Vbf,\Hscr_\Vbf)=:(\Qbf_0,\Vscr_0)$, and the restriction of the $\Qbb$-group homomorphism $f$ to the unipotent radical $\Vbf$ is identity.
  
  Let $\Gbf'$ be the image of $\Gbf$ in $\GSp_\Vbf$. Since $\gfrak=\Lie\Gbf$ is reductive, the epimorphism $\Lie f:\gfrak=\Lie\Gbf\ra\gfrak'=\Lie\Gbf'$ gives rise to a decomposition $\gfrak=\gfrak'\oplus\gfrak''$ of ideals of $\gfrak$. Now that $\gfrak$ contains a Lie subalgebra $\gfrak'$, $\Gbf$ contains a connected normal reductive $\Qbb$-subgroup $\Hbf'$ whose Lie algebra is $\gfrak'$. Let $(\Gbf_0,X_0)$ be the quotient of $(\Gbf,X)$ by $\Hbf'$. We have $\Gbf_0=\Gbf/\Hbf'$, and thus the homomorphism $\Gbf\ra\Gbf_0\times\Gbf'\mono\Gbf_0\times\GSp_\Vbf$ is of finite kernel because the Lie algebra homomorphism $\gfrak\ra(\gfrak/\gfrak')\oplus\gfrak'$ is an isomorphism. From the morphism $(\Gbf,X)\ra(\Gbf_0,X_0)\times(\GSp_\Vbf,\Hscr_\Vbf)$ we get $$(\Pbf,Y)=\Vbf\rtimes(\Gbf,X)\ra(\Gbf_0,X_0)\times(\Vbf\rtimes(\GSp_\Vbf,\Hscr_\Vbf))=(\Gbf_0,X_0)\times(\Qbf_0,\Vscr_0)$$ and $\Pbf\ra\Gbf_0\times\Qbf_0$ is of finite kernel.
  
  (2) When $\Ubf\neq 0$, by \cite{pink thesis} 2.14, $\Gbf$ acts on $\Ubf$ through a split $\Qbb$-torus, and $\Ubf$ splits into a direct sum of one-dimensional subrepresentations $\Ubf=\oplus_\alpha\Ubf_\alpha$.  We thus get $\psi=\oplus_\alpha\psi_\alpha$ with $\psi_\alpha$ the composition $\Vbf\times\Vbf\ra\Ubf\ra\Gbb_\arm$ using the $\alpha$-th projection. Write $\Wbf_\alpha$ for the extension of $\Vbf$ by $\Gbb_\arm$ using $\psi_\alpha$, we have an inclusion $\Wbf\mono\prod_{\alpha}\Wbf_\alpha$ which is the identify when restricted to $\Ubf$, and it becomes the diagonal embedding $\Vbf\mono\prod_\alpha\Vbf$ when reduced modulo $\Ubf$. It extends to an inclusion $\Pbf=\Wbf\rtimes\Gbf\mono\prod_\alpha\Pbf_\alpha$ with $\Pbf_\alpha=\Wbf_\alpha\rtimes\Gbf$, and we get an embedding $$(\Pbf,Y)=(\Ubf,\Vbf)\rtimes(\Gbf,X)\mono\prod_{\alpha}(\Pbf_\alpha,Y_\alpha)=\prod_\alpha\Wbf_\alpha\rtimes(\Gbf,X)=\prod_{\alpha}(\Gbb_\arm,\Vbf)\rtimes(\Gbf,X)$$ and it remains to prove the lemma for each $(\Pbf_\alpha,Y_\alpha)$.
  
We thus assume that $\Ubf$ is one-dimensional. The radical $\Vbf_0$ of $\psi:\Vbf\times\Vbf\ra\Ubf$ is $\{v\in\Vbf:\psi(v,v')=0\forall v'\in\Vbf\}$ is a subrepresentation in $\Vbf$ under $\Gbf$, hence we have a splitting $\Vbf=\Vbf_0\oplus\Vbf_1$ because $\Gbf$ is reductive, and the restriction $\psi:\Vbf_1\times\Vbf_1\ra\Ubf$ is non-degenerate, i.e. a symplectic form. We thus get an embedding $$(\Pbf,Y)=(\Ubf,\Vbf)\rtimes(\Gbf,X)\mono((\Ubf,\Vbf_1)\rtimes(\Gbf,X))\times(\Vbf_0\rtimes(\Gbf,X)).$$ The Kuga case $\Vbf_0\rtimes(\Gbf,X)$ is already treated in (1). As for $(\Ubf,\Vbf_1)\rtimes(\Gbf,X)$, $\Ubf$ is one-dimensional and $\Gbf$ acts on it by scalars, hence $\Gbf$ preserves $\psi$ up to similitude, which gives $(\Gbf,X)\ra(\GSp_{\Vbf_1},\Hscr_{\Vbf_1})$ and a morphism $(\Ubf,\Vbf_1)\rtimes(\Gbf,X)\ra(\Ubf,\Vbf_1)\rtimes(\GSp_{\Vbf_1},\Hscr_{\Vbf_1})$, whose restriction to the unipotent radical $\Wbf_1$ (extension of $\Vbf_1$ by $\Ubf$ via $\psi$) is the identity. Repeat the construction in (1) (i.e. lifting the image of $\Gbf\ra\GSp_{\Vbf_1}$ into a connected normal $\Qbb$-subgroup of $\Gbf$), we get a morphism $$(\Ubf,\Vbf_1)\rtimes(\Gbf,X)\ra(\Gbf_1,X_1)\times((\Ubf,\Vbf_1)\rtimes(\GSp_{\Vbf_1},\Hscr_{\Vbf_1})$$ in which the $\Qbb$-group homomorphism is of finite kernel. Hence the claim.\end{proof}

  \begin{remark}[reduction to subdata of a ``good product'']\label{reduction to subdata of a good product}
  
   When we study the Andr\'e-Oort conjecture for $M_K(\Pbf,Y)$, it suffices to restrict to each connected component of the form $M^+=\Gamma\bsh Y^+$ for $Y^+$ some fixed connected component of $Y$ and $\Gamma$ some suitably defined congruence subgroup of $\Pbf(\Qbb)_+$. In order to use the strategy of \cite{klingler yafaev} and \cite{ullmo yafaev} for $M^+$, it suffices to take the subdatum $(\Pbf',Y')$ where $\Pbf'=\MT(Y^+)\subset\Pbf$ and $Y'=\Pbf'(\Rbb)\Ubf'(\Cbb)Y^+$ using \ref{generating subdata}(3), which is irreducible, hence admits a morphism into $(\Gbf_0,X_0)\times(\Lbf,Y_\Lbf)$ satisfying \ref{reduction lemma}. Moreover the reduction modulo the center $\Zbf_0$ of $\Gbf_0$ gives
   $(\Gbf_0,X_0)\times(\Lbf,Y_\Lbf)\ra(\Gbf_0^\ad,X_0^\ad)\times(\Lbf,Y_\Lbf)$ satisfying \ref{insensitivity of isogeny}, hence we may reduce the Andr\'e-Oort conjecture to mixed Shimura varieties defined by a subdatum of a ``good product'' of the form $(\Gbf_0,X_0)\times(\Lbf,Y_\Lbf)$ with $\Gbf_0$ semi-simple of adjoint type and $(\Lbf,Y_\Lbf)$ a product of finitely many mixed Shimura data of Siegel type.
   
   We will encounter these good products in Section 4 and Section 5, which provide convenient settings for the estimation of degrees of Galois orbits of special subvarieties. \end{remark}

 \section{Measure-theoretic constructions on mixed Shimura varieties}\label{Measure-theoretic constructions on mixed Shimura varieties}
 
 In this section, we introduce some measure-theoretic constructions  associated to connected mixed Shimura varieties. Most of them are analogue to the Kuga case discussed in \cite{chen kuga} Section 2, 2.14-2.18, except that in the general case, we work with the notion of S-spaces. We also introduce the notion of $\TW$-special subdata, which is the analogue of $\Tbf$-special subdata of \cite{ullmo yafaev} 3.1 in the mixed case.
 
 \begin{definition}[lattice spaces and canonical measures]\label{lattice spaces and canonical measures}
 
 (1) A linear $\Qbb$-group $\Pbf$ is said to be \emph{of type $\Hscr$} if it is of the form $\Pbf=\Wbf\rtimes\Hbf$ with $\Wbf$ a unipotent $\Qbb$-group and $\Hbf$ a connected semi-simple $\Qbb$-group without normal $\Qbb$-subgroups $\Hbf'\subset\Hbf$ of dimension $>0$ such that $\Hbf'(\Rbb)$ is compact.
 
 For a mixed Shimura datum $(\Pbf,Y)$, the $\Qbb$-group of commutators  $\Pbf^\der$ is of type $\Hscr$, due to \ref{unipotent radical and Levi decomposition}(4). 
 
 (2) For $\Pbf$ a linear group of type $\Hscr$ and $\Gamma\subset\Pbf(\Rbb)^+$ a congruence subgroup, the quotient $\Omega=\Gamma\bsh \Pbf(\Rbb)^+$ is called the (connected) \emph{lattice space} associated to $(\Pbf,\Gamma)$. Since $\Gamma$ is discrete in $\Pbf(\Rbb)^+$, the space $\Omega$ is a smooth manifold. We also have the \emph{uniformization map} $\wp_\Gamma:\Pbf(\Rbb)^+\ra\Omega,\ a\mapsto\Gamma a$.
 
 (3) Let $\Omega=\Gamma\bsh\Pbf(\Rbb)^+$ be a lattice space as in (2). The left Haar measure $\nu_\Pbf$ on $\Pbf(\Rbb)^+$ passes to a measure $\nu_\Omega$ on $\Omega$: choose a fundamental domain $F\subset\Pbf(\Rbb)^+$ \wrt $\Gamma$, we put $\nu_\Omega(A)=\nu_\Pbf(F\cap\wp_\Gamma^\inv A)$ for $A\subset\Omega$ measurable.
 
 Following \cite{chen kuga} 2.15 (1), $\nu_\Omega$ is of finite volume and is normalized such that $\nu_\Omega(\Omega)=1$. We call it the \emph{canonical measure} on $\Omega$.

 \end{definition}

 \begin{definition}[lattice space and S-space]\label{lattice space and s-space}
 
 Let $(\Pbf,Y;Y^+)=(\Ubf,\Vbf)\rtimes(\Gbf,X;X^+)$ be a connected mixed Shimura datum with pure section $(\Gbf,X;X^+)$. Let $\Gamma$ be a congruence subgroup of $\Pbf(\Rbb)_+$, which gives us the connected mixed Shimura variety $M=\Gamma\bsh Y^+$.
 
 (1) The \emph{lattice space associated to $M$} is $\Omega=\Gamma^\dagger\bsh \Pbf^\der(\Rbb)^+$ where $\Gamma^\dagger=\Gamma\cap\Pbf^\der(\Rbb)^+$. It is equipped with the \emph{canonical measure} $\nu_\Omega$, and we have the \emph{uniformization map} $\wp_\Gamma:\Pbf^\der(\Rbb)^+\ra \Omega$, $g\mapsto \Gamma^\dagger g$.
 
 A \emph{lattice subspace} of $\Omega$ is of the form $\wp_\Gamma(\Hbf(\Rbb)^+)$ for some $\Qbb$-subgroup $\Hbf\subset\Pbf^\der$ of type $\Hscr$. Since $\Gamma^\dagger$ acts on $\Pbf^\der(\Rbb)^+$ by translation, we have $\wp_\Gamma(\Hbf(\Rbb)^+)\isom(\Gamma\cap\Hbf(\Rbb)^+)\bsh \Hbf(\Rbb)^+$ is a lattice space itself, and it is a closed submanifold of $\Omega$. It carries a canonical measure $\nu$ by \ref{lattice spaces and canonical measures}, and we regard $\nu$ as a probability measure on $\Omega$ with support equal to the submanifold $\wp_\Gamma(\Hbf(\Rbb)^+)$.
 \bigskip
 
 (2) We write $\Yscr^+$ for the $\Pbf(\Rbb)_+$-orbit of $X^+$ in $Y$, called the \emph{real part} of $Y^+$. It is a connected real submanifold of $Y^+$. The (connected) \emph{S-space} associated to $M$ is $\Mscr=\Gamma\bsh \Yscr^+$. Since $\Gamma$ contains a neat subgroup of finite index, the quotient $\Mscr$ is a real analytic space with at most singularities of finite group quotient.
 
 We also have the following \emph{orbit map} associated to any point $y\in\Yscr^+$:
 
 $$\kappa_y:\Omega=\Gamma^\dagger\bsh\Pbf^\der(\Rbb)^+\ra \Mscr=\Gamma\bsh \Yscr^+,\ \Gamma^\dagger g\mapsto \Gamma gy.$$ It is surjective because $\Gamma^\dagger\subset\Gamma$ and $\Yscr^+$ is a single $\Pbf^\der(\Rbb)^+$-orbit, as $X^+$ is homogeneous under $\Gbf^\der(\Rbb)^+$. It is a submersion of smooth real analytic spaces when $\Gamma$ is neat.
 
 \end{definition}
 
 \begin{remark}\label{remarks on s-spaces}
 
  The $\Pbf(\Rbb)$-orbit $\Yscr$ of $X$ in $Y$ is independent of the choice of pure section $(\Gbf,X)$ as different pure sections are conjugate to each other under $\Pbf(\Qbb)\subset\Pbf(\Rbb)$. $\Yscr^+$ is simply a connected component of $\Yscr$, as it is the preimage of $X^+$ under the natural projection $\Yscr(\mono Y)\epim X$ whose fibers are connected (isomorphic to $\Wbf(\Rbb))$.
  
  The notion of real part of $Y$ is inspired by the \emph{imaginary part} of $Y$ defined in \cite{pink thesis} 4.14.

 \end{remark}
 
 \begin{lemma}
 
 Let $(\Pbf,Y;Y^+)=(\Ubf,\Vbf)\rtimes(\Gbf,X;X^+)$ be a connected mixed Shimura datum with $\Ubf\neq 0$. Then for any congruence subgroup $\Gamma\subset\Pbf(\Qbb)_+$, the S-space $\Gamma\bsh \Yscr^+$ is dense in $\Gamma\bsh Y^+$ for the Zariski topology.
 
 \end{lemma}
 
 \begin{proof}
 
 In the Kuga case, we have $\Ubf=0$, hence the real part $\Yscr^+$ equals $Y^+$, and the S-space $\Mscr$ is just the Kuga variety. In the non-Kuga case, the projection $\pi_\Ubf$ gives us the commutative diagram $$\xymatrix{\Yscr^+\ar[d]^{\pi_\Ubf}\ar[r]^\subset & Y^+\ar[d]^{\pi_\Ubf}\\ Y^+/\Ubf(\Cbb)\ar[r]^= & Y^+/\Ubf(\Cbb)},$$ in which the vertical maps are smooth submersions of manifolds. The right one is a $\Ubf(\Cbb)$-torsor,  the left one is a $\Ubf(\Rbb)$-torsor, and $Y^+=\Yscr^+\times^{\Ubf(\Rbb)}\Ubf(\Cbb)$ namely the quotient of $\Yscr^+\times\Ubf(\Cbb)$ by the diagonal action of $\Ubf(\Rbb)$. 
 $\Ubf(\Rbb)\subset\Ubf(\Cbb)$ is Zariski dense when we view $\Ubf(\Cbb)$ as a complex algebraic variety, hence the density of $\Yscr$ in $Y$. The proof for $\Mscr\subset M$ is clear when we restrict to connected components and take quotient by $\Gamma$.
 \end{proof}

 The advantage of S-spaces is that they carry canonical measures of finite volumes. Parallel to the Kuga case studied in \cite{chen kuga} 2.17 and 2.18, we have the following:
 
 \begin{definition-proposition}[canonical measures on S-spaces]\label{canonical measures on s-spaces}
 
 Let $M=\Gamma\bsh Y^+$ be a connected mixed Shimura variety associated to $(\Pbf,Y;Y^+)$, with $\Omega=\Gamma^\dagger\bsh\Pbf^\der(\Rbb)^+$ the corresponding lattice space, and $\Mscr=\Gamma\bsh\Yscr^+$ the S-space. Fix a base point $y\in \Yscr^+\subset Y^+$.
 
 (1) The orbit map $\kappa_y:\Pbf^\der(\Rbb)^+\ra \Yscr^+,\ g\mapsto gy$ is a submersion with compact fibers. The isotropy subgroup $K_y$ of $y$ in $\Pbf^\der(\Rbb)^+$ is a maximal compact subgroup of $\Pbf^\der(\Rbb)^+$.
 
 The left invariant Haar measure $\nu_\Pbf$ on $\Pbf^\der(\Rbb)^+$ passes to a left invariant measure $\mu_\Yscr=\kappa_{y*}\nu_\Pbf$ on $\Yscr^+$, which is independent of the choice of $y$. 
 
 (2) Taking quotients by congruence subgroups, the orbit map $\kappa_y:\Gamma^\dagger\bsh\Pbf^\der(\Rbb)^+\ra \Gamma\bsh \Yscr^+,\ \ \Gamma^\dagger g\mapsto\Gamma gy$ is a submersion with compact fibers. The push-forward $\kappa_{y*}$ sends $\nu_\Omega$ to a probability measure $\mu_{\Mscr}$ on $\Mscr=\Gamma\bsh \Yscr^+$, independent of the choice of $y$. We call it the \emph{canonical measure} on $\Mscr$.
 
 (3) Let $ M'\subset M$ be a special subvariety defined by $(\Pbf',Y';Y'^+)$, and take $y\in Y'^+\subset Y^+$. Then we have the commutative diagram $$\xymatrix{ \Omega'\ar[r]^\subset\ar[d]^{\kappa_y} & \Omega\ar[d]^{\kappa_y}\\ \Mscr'\ar[r]^\subset & \Mscr}$$ with $\Omega'$ the special lattice space associated to $M'$, $\Mscr'$ the corresponding special S-subspace. In particular we have $\kappa_{y*}\nu_{\Omega'}=\mu_{\Mscr'}$, with $\nu_{\Omega'}$ the canonical measure of the lattice subspace $\Omega'$ associated to $\Pbf'^\der(\Rbb)^+$. Note that we identify $\nu_{\Mscr'}$ as a probability measure on $\Mscr$ with support equal to $\Mscr'$.
 
 Similarly, for the fibration over the pure base $\pi:M\ra S=\Gamma_\Gbf\bsh X^+$ with $\Gamma=\Gamma_\Wbf\rtimes\Gamma_\Gbf$, we have the submersions $\pi:\Omega\ra\Omega_\Gbf:=\Gamma_\Gbf^\dagger\bsh\Gbf^\der(\Rbb)^+$, $\pi:\Mscr' \ra S$, together with $\pi_*\nu_\Omega=\nu_{\Omega_\Gbf}$ and $\pi_*\mu_\Mscr=\mu_S$.

 \end{definition-proposition}

 \begin{proof}
   It suffices to replace the $\Vbf$'s etc. in \cite{chen kuga} 2.17 and 2.18 by $\Wbf$'s etc. as the proof there already works for general unipotent $\Vbf$'s.
 \end{proof}

 In the pure case, we have the notion of $\Tbf$-special sub-object, where $\Tbf$ is  the connected center of the $\Qbb$-group defining the subdatum, the special subvarieties, etc. In the mixed case, the connected center is of the form $w\Tbf w^\inv$ following the notations in \ref{structure of subdata}, and we prefer to separate $\Tbf$ and $w$, because $\Tbf$ provides information on the image in the pure base, and $w$ describes how the pure section has been translated from the given one defined by $(\Gbf,X)\subset(\Pbf,Y)$. In Introduction we have seen motivations for this notion on Kuga varieties via the description of special subvarieties as torsion subschemes in some subfamily of abelian varieties.
 
 \begin{definition}[$(\Tbf,w)$-special subdata]\label{tw-special subdata}
 
 Let $(\Pbf,Y)=(\Ubf,\Vbf)\rtimes(\Gbf,X)$ be a mixed Shimura datum, with $\Wbf$ the central extension of $\Vbf$ by $\Ubf$ as the unipotent radical of $\Pbf$. Take $\Tbf$ a $\Qbb$-torus in $\Gbf$ and $w$ an element in $\Wbf(\Qbb)$.
 
 (1) A subdatum $(\Pbf',Y')$ of $(\Pbf,Y)$ is said to be \emph{$(\Tbf,w)$-special} if it is of the form $(\Ubf',\Vbf')\rtimes(w\Gbf'w^\inv,wX')$ presented as in \ref{structure of subdata}, with $\Tbf$ equal to the connected center of $\Gbf'$. In the language of \cite{ullmo yafaev}, $(\Gbf',X')$ is a $\Tbf$-special subdatum of $(\Gbf,X)$, and the element $w\in\Wbf(\Qbb)$ conjugates it to a pure section of $(\Pbf',Y')$, cf. \ref{pure section}.
 
 (2) Similarly, if $M=\Gamma\bsh Y^+$ is a connected mixed Shimura variety  defined by $(\Pbf,Y;Y^+)$, then a special subvariety of $M$ is \emph{$(\Tbf,w)$-special} if it is defined by some (connected) irreducible $(\Tbf,w)$-special subdatum.
 
 We also define notions such as \emph{$(\Tbf,w)$-special lattice subspaces, $(\Tbf,w)$-special S-subspaces}, etc. in the evident way.
 
 (3) When we remove $w\in\Wbf(\Qbb)$ we get the notion of \emph{$\Tbf$-special} sub-objects, similar to the Kuga case studied in \cite{chen kuga} 3.1: a $\Tbf$-special subdatum is an irreducible subdatum $(\Pbf',Y')\subset(\Pbf,Y)$ such that the image of $(\Pbf',Y')$ under the natural projection is $(\Gbf',X')$ a pure irreducible subdatum of $(\Gbf,X)$ with $\Tbf$ equal to the connected center of $\Gbf'$; the notion of $\Tbf$-special subvarieties, etc. is understood in the evident way. If $(\Pbf,Y)$ is pure, i.e. $\Wbf=1$, then being $(\Tbf,1)$-special is the same as $\Tbf$-special pure subdata.

 \end{definition}

 We will also use the following variant to state our main results on the equidistribution of special subvarieties, inspired by the pure case studied in \cite{ullmo yafaev}. Subsets of closed subvarieties of a $\Qac$-variety of finite type are always countable, hence we talk about sequences of special subvarieties indexed by $\Nbb$ instead of ``families'', ``collections'', etc.

 \begin{definition}[bounded sequence]\label{bounded sequence} Let $M=\Gamma\bsh Y^+$ be a connected mixed Shimura variety defined by $(\Pbf,Y;Y^+)=\Wbf\rtimes(\Gbf,X;X^+)$. Fix a finite set $B=\{(\Tbf_1,w_1),\cdots,(\Tbf_r,w_r)\}$ with $\Tbf_i$ a $\Qbb$-torus in $\Gbf$ and $w_i\in\Wbf(\Qbb)$, $i=1,\cdots,r$. We call $B$ a (finite) \emph{bounding set}.

 (1) A special subvariety of $M$ is said to be \emph{bounded by} $B$ (or $B$-\emph{bounded}) if it is $\TW$-special for some $\TW\in B$.  A sequence $(M_n)$ of special subvarieties in $M$ is said to be \emph{bounded by} $B$ if  each $M_n$ is $B$-bounded.

 (2) Similarly,  a sequence of special lattice subspaces resp.   of special S-subspaces is bounded by $B$  if its members are defined by  $\TW$-special subdatum for $\TW\in B$.
 
 (3) For $\Omega$ resp. $\Mscr$ the lattice space resp. the S-space associated to $M$ we write $\Pscr_B(\Omega)$ resp. $\Pscr_B(\Mscr)$ for the set of canonical   measures on $\Omega$ resp. on $\Mscr$ associated to $B$-bounded special subvarieties.

 \end{definition}
 
 Note that when $(\Pbf,Y)=(\Gbf,X)$ is pure, $B$ is simply a finite set of $\Qbb$-tori in $\Gbf$.  
  
 \begin{remark}[conjugation by $\Gamma$]\label{conjugation by gamma}
 
 We consider a connected mixed Shimura variety of the form $M=\Gamma\bsh Y^+$ defined by $(\Pbf,Y;Y^+)=\Wbf\rtimes(\Gbf,X;X^+)$ with $\Gamma=\Gamma_\Wbf\rtimes\Gamma_\Gbf$ and fibred over $S=\Gamma_\Gbf\bsh X^+$. Let $Z$ be a $\TW$-special subvariety, defined by $(\Pbf',Y')=\Wbf'\rtimes(w\Gbf'w^\inv, wX;wX^+)$, with $\Tbf$ the connected center of $\Gbf'$.
 
 The pre-image of $Z$ under the uniformization $\wp_\Gamma: Y^+\ra M$ is the union $\bigcup_{\gamma\in \Gamma}\gamma Y'^+$, hence a subdatum $(\Pbf'',Y'';Y''^+)$ defines the same special subvariety $Z$ as $(\Pbf',Y';Y'^+)$ does \ifof $(\Pbf'',Y'';Y''^+)=(\gamma\Pbf'\gamma^\inv, \gamma Y'; \gamma Y'^+)$ for some $\gamma\in\Gamma$.
 
 In our definition, $\Tbf$ only depends on the image of $Z$ in $S$, and we can conjugate $\Tbf$ by $\Gamma_\Gbf$; in the unipotent radical, since $\Gamma_\Wbf$ acts on $\Ubf(\Cbb)\Wbf(\Rbb)$ by translation, $w$ and $w'$ give rise to the same special subvariety \ifof $w=\gamma w'$ for some $\gamma\in\Gamma_\Wbf$. We thus conclude that the notion of $\TW$-special subvarieties actually only depends on the class $[\Tbf,w]$, by which we mean the $\Gamma_\Gbf$-conjugacy class of $\Tbf$ in $\Gbf$ and the coset $\Gamma_\Wbf w$ in $\Gamma_\Wbf\bsh\Wbf(\Qbb)$.
 
 \end{remark}
 
 \section{Bounded equidistribution of special subvarieties}\label{Bounded equidistribution of special subvarieties}

 We first consider the case when the bound $B$ consists of one single element $\TW$, and we write $\Pscr_\TW(\Sscr)$ in place of $\Pscr_B(\Sscr)$ for $\Sscr\in\{\Omega,\Mscr\}$. This is exactly the analogue of the $\Tbf$-special case for pure Shimura varieties, and the main theorem of this section will rely on the following theorem of S. Mozes and N. Shah:
 
 \begin{theorem}[Mozes-Shah, cf.\cite{mozes-shah}]\label{mozes shah} Let $\Omega=\Gamma\bsh\Hbf(\Rbb)^+$ be the lattice space associated to a $\Qbb$-group of type $\Hscr$ and a congruence subgroup $\Gamma\subset\Hbf(\Rbb)^+$. Write $\Pscr_h(\Omega)$ for the set of canonical measures on $\Omega$ associated to lattice subspaces defined by $\Qbb$-subgroups of type $\Hscr$. Then $\Pscr_h(\Omega)$ is compact for the weak topology as a subset of the set of Radon measures on $\Omega$, and the property of ``support convergence'' holds on it: if $\nu_n$ is a convergent sequence in $\Pscr_h(\Omega)$ of limit $\nu$, then we have $\supp\nu_n\subset \supp\nu$ for $n\geq N$, $N$ being some positive integer, and the union $\bigcup_{n\geq N}\supp\nu_n$ is dense in $\supp\nu$ for the analytic topology.
 
 \end{theorem}
  
 We begin with the strategy of the proof of the equidistribution of $\TW$-special subspaces and S-subspaces, in comparison with the pure case treated in \cite{clozel ullmo} \cite{ullmo yafaev} and the rigid Kuga case in \cite{chen kuga}:\begin{enumerate}
  
 \item[(i)] For lattice subspaces, it suffices to show that $\Pscr_\TW(\Omega)$ is a closed subset of $\Pscr_h(\Omega)$, namely if $\nu_n$ is a sequence of canonical measures of limit $\nu$, such that each $\nu_n$ is associated to $\Pbf_n^\der$ from some $\TW$-special subdata $(\Pbf_n,Y_n)$, then
 $\nu$ is also associated to $\Pbf_\nu^\der$ from some $\TW$-special subdatum $(\Pbf_\nu,Y_\nu)$. In fact:\begin{enumerate}
 \item[(i-1)] in the pure case,   the $\Qbb$-group  $\Qbf$ generated by the union of semi-simple $\Qbb$-groups $\bigcup_{n>>0}\Pbf_n^\der$ is a   semi-simple $\Qbb$-group of type $\Hscr$, and is centralized by $\Tbf$ the common connected center of the $\Pbf_n$, and $\Tbf\Qbf$ is the $\Qbb$-group of some $\Tbf$-special subdatum;
 
 \item[(i-2)] in the Kuga case treated in \cite{chen kuga}, the $\Qbb$-group  $\Qbf$ generated by the union  $\bigcup_{n>>0}\Pbf^\der_n$ is of the form $\Vbf'\rtimes\Hbf'$ with $\Vbf'$ unipotent and $\Hbf'$ semi-simple of type $\Hscr$, however there might be infinitely many subdata $(\Pbf',Y')$ such that $\Pbf'^\der=\Qbf$, because for any  $v\in\Vbf(\Qbb)$   fixed by $\Gbf'^\der$ and $\Pbf'=\Vbf'\rtimes(v\Gbf'v^\inv)$ we have $\Pbf'^\der=\Vbf'\rtimes\Gbf'^\der$, cf. \cite{chen kuga} 2.19 and 3.6; to exclude this situation we have restricted to the $\rho$-rigid case in \cite{chen kuga}, and results like \cite{chen kuga} 3.5 and 4.5 show that the canonical  measures
 associated to $\rho$-rigid subdata form a closed subset of $\Pscr_h(\Omega)$;
 \item[(i-3)] in the general case under the $\TW$-special assumption, the situation is similar to the pure case (i-1), and we will construct a $\TW$-special subdatum out of the $\Qbb$-group generated by the union $\bigcup_{n>>0}\Pbf_n^\der$; thanks to the specification of the unipotent element $w$, we do not need any $\rho$-rigidity to ensure that the new datum is well defined.
 
 \end{enumerate}
 \item[(ii)] For S-spaces, we follow a similar reduction following the pure case in \cite{clozel ullmo}, namely there exists a compact subset $C$ of $\Yscr^+$ such that any $\TW$-special S-subspace in $\Mscr$ can be defined by some subdatum $(\Pbf',Y';Y'^+)$ satisfying $\Yscr'^+\cap C\neq \emptyset$, and then we may repeat the remaining arguments in \cite{clozel ullmo}.
 \end{enumerate}
 The following lemma will be useful both in the pure case and in the mixed case.
 
 \begin{lemma}[maximal $\TW$-special subdata]\label{maximal TW-special subdata} For $(\Pbf,Y)=\Wbf\rtimes(\Gbf,X)$ a mixed Shimura datum. If $\TW$ a pair as in \ref{tw-special subdata}, then the set of maximal $\TW$-special subdata of $(\Pbf,Y)$ is   finite.
 \end{lemma}
 
 \begin{proof}

 If $(\Gbf',X')$ is a maximal $\Tbf$-special subdatum of $(\Gbf,X)$, then $\Wbf\rtimes(\Gbf',X')=\Wbf\rtimes(w\Gbf'w^\inv,wX')$ is a maximal $\TW$-special subdatum of $(\Pbf,Y)$ using $w\in\Wbf(\Qbb)$. Hence we are reduced to the pure case. 
 
 If $(\Gbf_1,X_1)\subset(\Gbf,X)$ is a $\Tbf$-special subdatum, then $\Gbf_1$ is contained in the neutral component $\Zbf^\circ$ of the centralizer $\Zbf_\Gbf\Tbf$, and $\Zbf^\circ$ admits a decomposition into an almost direct product: $\Zbf^\circ=\Cbf\Hbf'\Hbf''$ with $\Cbf$ a $\Qbb$-torus, $\Hbf'$ a semi-simple $\Qbb$-group without compact $\Qbb$-factors, and $\Hbf''$ a semi-simple $\Qbb$-group without non-compact $\Qbb$-factors. We put $\Gbf'=\Cbf\Hbf'$, then the constructions in \cite{ullmo yafaev} 3.6 gives us a maximal $\Tbf$-special subdatum $(\Gbf',\Gbf'(\Rbb)x)$ with $x\in X_1$ arbitrary, cf. \ref{generating subdata}. Note that the $\Qbb$-group $\Gbf'$ is determined by $\Tbf$ and is independent of $(\Gbf_1,X_1)$, thus maximal $\Tbf$-special subdata are associated to $\Gbf'$. Hence by \cite{ullmo yafaev} 3.7 there are only finitely many maximal $\Tbf$-special subdata. \end{proof}
 \begin{theorem}[equidistribution of $\TW$-special subspaces]\label{equidistribution of TW-special subspaces} Let $\Sscr$ be the lattice space (resp. the S-space) associated to a connected mixed Shimura variety $M=\Gamma\bsh Y^+$ defined by $(\Pbf,Y;Y^+)=\Wbf\rtimes(\Gbf,X;X^+)$. Fix a pair $(\Tbf,w)$ as in \ref{tw-special subdata} and put $B=\{\TW\}$. Then the set $\Pscr_B(\Sscr)$ is compact for the weak topology, and the property of support convergence holds in it in the sense of \ref{mozes shah}.
 
 \end{theorem}
 
 We start with the case of lattice subspaces with $(\Pbf,Y)=(\Gbf,X)$ pure. In \cite{clozel ullmo} $\Gbf$ was assumed to be of adjoint type, and in \cite{ullmo yafaev} the case of $\Tbf$-special subdata of $(\Gbf,X)$ with $\Gbf$ of adjoint type was considered. Here we adapt some of their arguments for general reductive $\Gbf$.

 \begin{proof}[Proof for lattice subspaces in the pure case] We have $\Omega=\Omega_\Gbf=\Gamma_\Gbf^\dagger\bsh\Gbf^\der(\Rbb)^+$, hence   $\Pscr_\Tbf(\Omega)$ is a subset of $\Pscr_h(\Omega)$ (as there is no unipotent vector $w$ in this case). We proceed to show that $\Pscr_\Tbf(\Omega)$ is closed  in $\Pscr_h(\Omega)$ for the weak topology. By the proof of \ref{maximal TW-special subdata} all the maximal $\Tbf$-subdata of $(\Gbf,X)$ are associated to a common reductive $\Qbb$-subgroup $\Gbf^\Tbf$ of $\Gbf$, and the lattice subspace of $\Tbf$-special subdata are actually lattice subspaces of $$\Omega_{\Gbf^\Tbf}=\wp_{\Gamma_\Gbf}((\Gbf^\Tbf)^\der(\Rbb)^+)\isom\Gamma_{\Gbf^\Tbf}^\dagger\bsh(\Gbf^\Tbf)^\der(\Rbb)^+$$ with $\Gamma_{\Gbf^\Tbf}^\dagger=(\Gbf^\Tbf)^\der(\Rbb)^+\cap\Gamma_\Gbf^\dagger$.  Hence we may identify $\Pscr_\Tbf(\Omega)$ as a subset of $\Pscr_h(\Omega_{\Gbf^\Tbf})$, the latter being a closed subset of $\Pscr_h(\Omega)$. We assume for simplicity that $\Gbf=\Gbf^\Tbf$, namely $\Tbf$ equals the connected center of $\Gbf$.

 Take a sequence $\nu_n$ of canonical measures, which converges in $\Pscr_h(\Omega)$ to some $\nu$. Each $\nu_n$ is associated to a $\Qbb$-group $\Gbf_n$ of $\Gbf$ coming from some $\Tbf$-special subdatum $(\Gbf_n,X_n)$, with $\Omega_n=\supp\nu_n=\wp_{\Gamma_\Gbf}(\Gbf_n^\der(\Rbb)^+)$. The limit $\nu$ is associated to some connected $\Qbb$-group $\Gbf'$ of type $\Hscr$, of support $\Omega_\nu=\wp_{\Gamma_\Gbf}(\Gbf'(\Rbb)^+)$. We may assume for simplicity that $\bigcup_{n\geq0}\Omega_n$ is contained in $\Omega_\nu$ as a dense subset by restricting to a subsequence after dropping finitely many terms in the original sequence. In particular, $\Omega_n$ and $\Omega_\nu$ are smooth submanifolds of $\Omega$ the quotient of $\Gbf^\der(\Rbb)^+$ by translation of the discrete subgroup $\Gamma_\Gbf^\dagger$, and the inclusion $\Omega_n\subset\Omega_\nu$ implies the inclusion $\Lie\Gbf_n^\der\subset\Lie\Gbf'$ by computing the tangent spaces of the common point $\wp_\Gamma(e)$, $e$ being the neutral point of $\Gbf^\der(\Rbb)^+$. This gives $\Gbf_n^\der\subset\Gbf'\subset\Gbf^\der$.
 
 Since $\Tbf$ is  the common connected center of $\Gbf$ and $\Gbf_n$ , we have the equality of centralizers $\Zbf_\Gbf\Gbf_n=\Tbf\Zbf_{\Gbf^\der}\Gbf_n^\der$ for all $n$. Take any $x_n\in X_n$, $\Zbf_{\Gbf}\Gbf_n(\Rbb)$ is compact modulo $\Tbf$ as it centralizes $x_n(\ibf)$. Hence $\Zbf_{\Gbf^\der}\Gbf_n^\der$ is compact, and $\Gbf'$ is reductive by \cite{eskin mozes shah} Lemma 5.1. We thus obtain an inclusion chain of connected semi-simple $\Qbb$-groups $\Gbf_n^\der\subset\Gbf'\subset\Gbf^\der$, which extends to $\Gbf_n\subset\Tbf\Gbf'\subset\Gbf$. Put $\Gbf_\nu=\Tbf\Gbf'$, then \ref{generating subdata} gives further an inclusion chain of $\Tbf$-special subdatum $(\Gbf_n,X_n)\subset(\Gbf_\nu,X_\nu)\subset(\Gbf,X)$ with $X_\nu=\Gbf_\nu(\Rbb)x_n$ for any $x_n\in X_n$. 
 
 Therefore $\Gbf'=\Gbf_\nu^\der$ does come from some $\Tbf$-special subdatum, and $\nu$ is $\Tbf$-special.\end{proof}
 Note that we have constructed subdata of the form $(\Gbf_\nu,\Gbf_\nu(\Rbb)x_n)$ using an arbitrary $x_n\in X_n$ for all $n\in\Nbb$. Only finitely many $\Tbf$-subdata are obtained in this way due to \ref{common Mumford-Tate group}.
 
  From the proof we also obtain:
 \begin{lemma}\label{generating subgroups}
 Using the notations in the proof,   $\Gbf_\nu$ is generated by $\bigcup_n\Gbf_n$. \end{lemma}
 
 \begin{proof} In the proof above we already have $\Gbf_n^\der\subset\Gbf'=\Gbf_\nu^\der$ for all $n$. Thus $\Gbf'$ contains $\Hbf$ the $\Qbb$-subgroup generated by $\bigcup_n\Gbf_n^\der$ in $\Gbf^\der$. If $\Hbf\subsetneq\Gbf'$, then all the lattice subspaces $\Omega_n=\wp_{\Gamma_\Gbf}(\Gbf_n^\der(\Rbb)^+)$ are contained in the subspace $\wp_{\Gamma_\Gbf}(\Hbf(\Rbb)^+)$ which is a proper submanifold of $\Omega_\nu$, contradicting the density of $\bigcup_n\Omega_n$ in $\Omega_\nu$. Hence $\Hbf=\Gbf'$.\end{proof}
  
 Now we pass to the general mixed case:  
 
 \begin{proof}[proof of \ref{equidistribution of TW-special subspaces} for lattice subspaces]
 
 For any mixed Shimura subdatum $(\Pbf',Y')=\Wbf'\rtimes(\Gbf',X')$ of $(\Pbf,Y)$, the equality $\Pbf'^\der=\Wbf'\rtimes\Gbf'^\der$ by \ref{unipotent radical and Levi decomposition}(4) shows that $\Pbf'^\der$ is of type $\Hscr$. Hence for $\Omega=\Gamma^\dagger\bsh \Pbf^\der(\Rbb)^+$, $\Pscr_\TW(\Omega)$ is a subset of $\Pscr_h(\Omega)$, and we need to show that $\Pscr_\TW(\Omega)$ is closed in $\Pscr_h(\Omega)$ for the weak topology, just as the pure case above.
 
 We thus take a convergent sequence $\nu_n$ in $\Pscr_h(\Omega)$ of limit $\nu$, such that $\nu_n\in\Pscr_\TW(\Omega)$, is associated to some $\TW$-special subdatum $(\Pbf_n,Y_n)=\Wbf_n\rtimes(w\Gbf_nw^\inv,wX_n)$, and the support of $\nu_n$ is the $\TW$-special lattice subspace $\Omega_n=\wp_\Gamma(\Pbf_n^\der(\Rbb)^+)$, with $\Pbf_n^\der=\Wbf_n\rtimes w\Gbf_n^\der w^\inv$. The limit $\nu$ is associated to some $\Qbb$-subgroup $\Pbf'$ of type $\Hscr$, with $\supp\nu=\Omega_\nu=\wp_\Gamma(\Pbf'(\Rbb)^+)$. We assume for simplicity that $\Omega_n\subset\Omega_\nu$ for all $n$, hence $\bigcup_n\Omega_n$ is dense in $\Omega_\nu$ for  the analytic topology.
 
 Let $\Gamma_\Gbf$ be the image of $\Gamma$ in $\Gbf(\Rbb)^+$, which is a congruence subgroup of $\Gbf(\Rbb)^+$, and we write $\Gamma_\Gbf^\dagger=\Gamma_\Gbf\cap\Gbf^\der(\Rbb)^+$, together with $\pi:\Omega\ra\Omega_\Gbf=\Gamma_\Gbf^\dagger\bsh\Gbf^\der(\Rbb)^+$ for the projection deduced from $\pi=\pi_\Wbf:\Pbf\ra\Gbf$ the quotient modulo $\Wbf$. Then by \ref{canonical measures on s-spaces},   $\pi_*\nu_n$  is the canonical measure associated to $\pi(\Pbf_n^\der)=\Gbf_n^\der$. We clearly have the convergence $\lim_{n\ra\infty}\pi_*\nu_n=\pi_*\nu$, hence  by the result in the pure case, $\pi_*\nu$ is $\Tbf$-special, i.e. it is associated to some $\Tbf$-special subdatum $(\Gbf_\nu,X_\nu)$.  $\Gbf_\nu$ is generated by $\bigcup_n\Gbf_n$, and the connected semi-simple $\Qbb$-subgroup $\Gbf'=\Gbf_\nu^\der$ is generated by $\bigcup_n\Gbf_n^\der$. On the other hand, it is direct from the construction that the image $\pi(\Pbf')$ of $\Pbf'$ modulo $\Wbf$ in $\Gbf^\der$ is a $\Qbb$-subgroup of type $\Hscr$, and $\wp_{\Gamma_\Gbf}(\pi(\Pbf')(\Rbb)^+)$ equals the support of $\pi_*\nu$. Hence $\pi(\Pbf')=\Gbf'=\Gbf_\nu^\der$. 

 It is also clear that the unipotent $\Qbb$-subgroup $\Wbf':=\Wbf\cap\Pbf'$ is the unipotent radical of $\Pbf'$ because $\Pbf'/\Wbf'=\Gbf'$, and we write $\Vbf'$ for the image of $\Wbf'$ in $\Vbf=\Wbf/\Ubf$, which makes $\Wbf'$ the central extension of $\Vbf'$ by $\Ubf':=\Ubf\cap\Wbf'$ under the restriction of $\psi:\Vbf\times\Vbf\ra\Ubf$. Just as the pure case, we have inclusions of smooth submanifolds $\Omega_n\subset\Omega_\nu$, which gives $\Pbf_n^\der\subset\Pbf'$ for all $n$. We want to show that $\Pbf'$ is generated by $\bigcup_n\Pbf_n^\der$.
 
 Let $\Wbf''$ be the unipotent $\Qbb$-subgroup of $\Wbf$ generated by $\bigcup_n\Wbf_n$. Then reduction modulo $\Ubf$ shows that $\Vbf'':=\Wbf''/\Ubf''$ is generated by $\bigcup_n\Vbf_n$ with $\Vbf_n=\Wbf_n/\Ubf_n$, and similarly  $\Ubf'':=\Ubf\cap\Wbf''$ is generated by $\bigcup_n\Ubf_n=\bigcup_n\Wbf_n\cap\Ubf$.
 The $\Qbb$-groups $\Vbf_n$, $\Ubf_n$, and $\Wbf_n$ are stable under $w\Gbf_n^\der w^\inv$ and $w\Tbf w^\inv$, hence $\Wbf''$ is stabilized by $w\Gbf'w^\inv$, by $w\Tbf w^\inv$, and by $w\Gbf_\nu w^\inv$. Thus $\Wbf''\rtimes w\Gbf'w^\inv$ is already a $\Qbb$-subgroup of type $\Hscr$ containing $\bigcup_n\Pbf_n^\der$. This forces the equality $\Pbf'=\Wbf''\rtimes w\Gbf'w^\inv$ due to the density of $\bigcup_n\Omega_n$ in $\Omega_\nu$. In particular $\Wbf'=\Wbf''$ is a central extension of $\Vbf'=\Vbf''$ by $\Ubf'=\Ubf''$, stable under the actions of $w\Gbf'w^\inv$, of $w\Tbf w^\inv$, and thus of $w\Gbf_\nu w^\inv$. 
 
 We thus put $\Pbf_\nu:=\Pbf'w\Tbf w^\inv=\Wbf'\rtimes w\Gbf_\nu w^\inv$. It contains $\Pbf_n$ for all $n$, and it is clear that $(\Pbf_\nu,\Ubf',\Pbf_\nu(\Rbb)\Ubf'(\Cbb)y_n)$ is a $\TW$-special subdatum by \ref{generating subdata}. The equality $\Pbf_\nu^\der=\Wbf'\rtimes w\Gbf'w^\inv$ shows that $\nu$ is $\TW$-special. \end{proof}

 \begin{proof}[proof of \ref{equidistribution of TW-special subspaces} for S-subspaces] 
 The idea is similar to the pure case in \cite{clozel ullmo} and the Kuga case treated in \cite{chen kuga}(Section 5), and we merely sketch the main arguments.
 
 \begin{itemize}
 
   \item There exists a compact subset $K\TW$ of $\Yscr^+$ such that if $\Mscr'\subset \Mscr$ is a $\TW$-special S-subspace, then $\Mscr'=\wp_\Gamma(\Yscr'^+)$ is given by some connected $\TW$-special subdatum $(\Pbf',Y';Y'^+)$, with real part $\Yscr'^+$ meeting $K\TW$ non-trivially. In the pure case such a compact subset $K(\Tbf)\subset X^+$ is given in \cite{clozel ullmo} 4.5; in the mixed case it suffices to take a compact subset $C$ of $\Wbf(\Rbb)$ containing $w$ and a fundamental domain for the action of $\Gamma_\Wbf$ on $\Wbf(\Rbb)$, and then $K(\TW):=(C\cdot X^+)\cap\pi^\inv(K(\Tbf))$, the proof for which is the same as \cite{chen kuga} 5.4.
 
   \item The set $\Pscr_\TW(\Mscr)$ is compact for the weak topology: if $\mu_n$ is a sequence of $\TW$-special canonical measures on $\Mscr$ defined by $(\Pbf_n,Y_n;Y_n^+)$, given as $\mu_n=\kappa_{y_n*}\nu_n$ for $y_n\in K\TW$ and $\nu_n$ the canonical measure associated to $\Omega_n=\Gamma_n^\dagger\bsh \Pbf_n^\der(\Rbb)^+$, then up to restriction to subsequences, we may assume that $y_n$ converges to some $y\in K\TW$ and $(\nu_n)$ converges to some $\nu$ associated to a $\TW$-special subdatum $(\Pbf',Y;Y'^+)$ with $y\in\Yscr'^+\cap K\TW$. Thus $\mu_n$ converges to $\mu=\kappa_{y*}\nu$.
 
 \end{itemize}
 
 The property of support convergence holds similarly.
 \end{proof}
 
 \begin{corollary}[bounded equidistribution]\label{bounded equidistribution} (1) For $B$ a finite bounding set, $\Sscr\in\{\Omega,\Mscr\}$, we have $\Pscr_B(\Sscr)=\bigcup_{\TW\in B}\Pscr_\TW(\Sscr)$. In particular, $\Pscr_B(\Sscr)$ is compact for the weak topology, in which holds the support convergence.
 
 (2) For $\Sscr=\Omega$ (resp. $\Sscr=\Mscr$), the closure of a sequence of special lattice subspaces (resp. of special S-subspaces) bounded by $B$ for the analytic topology is a finite union of special lattice subspaces (resp. of special S-subspaces) bounded by $B$.
 
 \end{corollary}
 
 \begin{proof}
 
   (1) This is clear because $\Pscr_B(\Sscr)$ is a finite union of compact subsets of the set of Radon measures on  $\Sscr$. The property of support convergence holds because if a sequence $(\mu_n)$ converges to $\mu$, then it contains a subsequence that converges into $\Pscr_\TW(\Sscr)$  for some $\TW\in B$. Hence the $\nu\in\Pscr_\TW(\Sscr)$. All the convergent subsequence of $(\mu_n)$ are of the same limit, so it is not possible to have an infinite subsequence lying outside $\Pscr_\TW(\Sscr)$. Hence the sequence itself is in $\Pscr_\TW(\Sscr)$.
 
   (2) This is clear using the convergence of measures and the property of support convergence.
 \end{proof}
 
 \begin{corollary}[bounded Andr\'e-Oort]\label{bounded Andr\'e-Oort} Let $M$ be a connected Shimura variety defined by $(\Pbf,Y;Y^+)=(\Ubf,\Vbf)\rtimes(\Gbf,X;X^+)$, with $B$ a finite bounding set. Let $(M_n)$ be a sequence of special subvarieties bounded by $B$. Then the Zariski closure of $\bigcup_nM_n$ is a finite union of special subvarieties bounded by $B$.
 
 \end{corollary}
 
 \begin{proof}
   This is clear because analytic closure is finer than Zariski closure, and   S-subspaces are Zariski dense in the corresponding special subvarieties.
 \end{proof}

 \begin{remark}[compact tori vs. algebraic tori]\label{compact tori vs. algebraic tori}
 When $\Vbf=0$ and $\Gamma=\Gamma_\Ubf\rtimes\Gamma_\Gbf$ for some lattice $\Gamma_\Ubf$ in $\Ubf(\Qbb)$ stabilized by $\Gamma_\Gbf$ a congruence subgroup of $\Gbf(\Qbb)_+$, the fibration $M=\Gamma\bsh Y^+\ra S=\Gamma_\Gbf\bsh X^+$ is a torus group scheme over $S$, whose fibers are complex tori isomorphic to $\Gamma_\Ubf\bsh\Ubf(\Cbb)\isom (\Cbb/\Zbb)^d$, $d$ being the dimension of $\Ubf$. Thus the S-space $\Mscr=\Gamma\bsh \Yscr^+$ is a real analytic subgroup of $M$ relative to the base $S$, whose fibers are compact tori isomorphic to $\Gamma_\Ubf\bsh \Ubf(\Rbb)\isom(\Rbb/\Zbb)^d$ in the split complex tori $\Gamma_\Ubf\bsh \Ubf(\Cbb)$, hence Zariski dense.
 
 Using harmonic analysis on $\Gamma_\Ubf\bsh \Ubf(\Rbb)$ one can prove that the analytic closure of a sequence of connected closed Lie subtori in it is still a connected closed Lie subtorus, which implies that the Zariski closure of a sequence of connected algebraic subtori in $\Gamma_\Ubf\bsh\Ubf(\Cbb)$ is an algebraic subtorus, cf. \cite{ratazzi ullmo} Section 4.1. This can be viewed as a motivation for our notion of S-spaces.
 
 \end{remark}

 \section{Lower bound of the Galois orbit of a pure special subvariety}\label{Lower bound of the Galois orbit of a pure special subvariety}
 The results from this section on rely heavily on \cite{ullmo yafaev}, especially the estimation on Galois orbits of $\Tbf$-special subvarieties in pure Shimura varieties. Hence we assume that all the mixed Shimura (sub)data we encounter are irreducible in the sense of \ref{mixed Shimura data}(3). This actually forces  the pure part to be irreducible, due to the following lemma, 
 
 In \cite{ullmo yafaev} Lemma 2.1, a special subvariety $S'$ of a pure Shimura variety $M_K(\Gbf,X)$ is realized as the image of a connected component of some morphism $M_{K'}(\Gbf',X')\ra M_K(\Gbf,X)$, with $\Gbf'$ the generic Mumford-Tate $\Qbb$-group of $S'$, and $K'=K\cap\Gbf'(\adele)$. All the estimations concerning the $\Tbf$-special subvarieties requires $\Tbf$ to be the connected center of the generic Mumford-Tate $\Qbb$-group $\Gbf'$. The mixed case is similar: by \ref{insensitivity of levels} and the discussion in \ref{reduction to subdata of a good product}, it suffices to treat special subvarieties in $M^+=\Gamma\bsh Y^+$, where $Y^+$ comes from some irreducible mixed Shimura datum $(\Pbf,Y)$ and $\Gamma=\Pbf(\Qbb)_+\cap K$ for any fixed \cosg $K\subset\Pbf(\adele)$; in $M^+$ special subvarieties are obtained, using the discussion in \ref{special subvarieties}(1), as the image of $Y'^+\times K$ with $Y'^+\subset Y^+$ coming from some subdatum $(\Pbf',Y')$, and we may assume that $\Pbf'=\MT(Y'^+)$ because this does not change $Y'^+$.




 We start with some preliminaries on the reciprocity map describing the Galois action on the set of connected components of pure Shimura varieties. 
 
 \begin{definition}[connected components]\label{connected components}
 
 Let $\Gbf$ be a connected reductive $\Qbb$-group. We write $\pibar(\Gbf)$ for the set $\pi_0(\Gbf(\Abb)/\Gbf(\Qbb)^-\rho(\Hbf(\adele)))$ where $\Gbf(\Qbb)^-$ stands for the closure of $\Gbf(\Qbb)$ in $\Gbf(\Abb)$, and $\rho:\Hbf\ra\Gbf^\der$ is the simply-connected covering of $\Gbf^\der$. Clearly $\pibar(\Gbf)=\pi_0(\Gbf(\Abb)/\Gbf(\Qbb)^-\Gbf(\Rbb)^+\rho(\Hbf(\adele))$. 
 
 Since $\Gbf(\adele)$ is totally disconnected,   the natural action of $\Gbf(\adele)$ on $\Gbf(\Abb)$ by left translation gives an action on $\pibar(\Gbf)$. For $K\subset\Gbf(\adele)$, we denote the quotient by $K$ of $\pibar(\Gbf)$ as $\pibar(\Gbf)/K$. We also have the natural action of $\pi_0(\Gbf(\Rbb))$ on $\pibar(\Gbf)$.
 
 From the finiteness of class numbers of linear algebraic group over global fields (\cite{platonov rapinchuk} 8.1), we know that for each $K\subset\Gbf(\adele)$, the quotient $\pibar(\Gbf)/K$ is a finite abelian group. Hence $\pibar(\Gbf)=\limproj_K\pibar(\Gbf)/K$ is a pro-finite abelian group.
 
 In particular, if $F$ is a number field, we have the $\Qbb$-torus $\Gbb_\mrm^F=\Res_{F/\Qbb}\mult_F$, and  class field theory gives us the reciprocity isomorphism $\rec^F:\Gal(F^\ab/F)\isom\pibar(\Gbb_\mrm^F)$.
 
 For a pure Shimura variety  $S=M_K(\Gbf,X)$, the set of its geometrically connected components is $\pi_0(S)=\pibar(\Gbf)/\Gbf(\Rbb)_+K$, with $\Gbf(\Rbb)_+$ acts through $\pi_0(\Gbf(\Rbb)_+)\subset\pi_0(\Gbf(\Rbb))$. 
 
 \end{definition}
 
 \begin{definition-proposition}[reflex fields and reciprocity maps, cf. \cite{deligne pspm}, \cite{pink thesis}, \cite{ullmo yafaev}]\label{reflex fields and reciprocity maps}
 
 (1) Let $(\Pbf,Y)$ be a mixed Shimura datum. The \emph{reflex field} $E(\Pbf,Y)$ is the smallest subfield $E$ of $\Cbb$ such that $\Aut(\Cbb/E)$ fixes the $\Pbf(\Cbb)$-conjugacy class of $\mu_y:\mult_\Cbb\ra\Pbf_\Cbb$, with $\mu_y$ the restriction of $y:\Sbb_\Cbb\ra\Pbf_\Cbb$ to $\mult\times\{1\}$ via $\Sbb_\Cbb\isom\mult\times\mult$.
 
 $E(\Pbf,Y)$ is a number field embedded in $\Cbb$. Whenever there is a morphism of mixed Shimura data $(\Pbf,Y)\ra(\Pbf',Y')$ we have $E(\Pbf,Y)\supset E(\Pbf',Y')$. In particular, using the natural projection and the pure section of $(\Pbf,Y)=\Wbf\rtimes(\Gbf,X)$, we have $E(\Pbf,Y)=E(\Gbf,X)$.
 
 (2) If $(\Tbf,x)$ is a pure Shimura datum with $\Tbf$ a $\Qbb$-torus, then $E=E(\Tbf,x)$ is the field of definition of $\mu_x:\mult_\Cbb\ra\Tbf_\Cbb$, and the reciprocity map of $(\Tbf,x)$ is the composition $$\rec_x:\Gal(E^\ab/E)\isom\pibar(\Gbb_\mrm^E)\ot{\mu_x}\lra\pibar(\Tbf^E)\ot{\Nm_{E/\Qbb}}\lra\pibar(\Tbf)$$ with $\Tbf^E=\Res_{E/\Qbb}T_E$ and $\Nm_{E/\Qbb}$ induced by the norm $E^\times\ra\Qbb^\times$.
 
 For a general pure Shimura datum $(\Gbf,X)$ of reflex field $E=E(\Gbf,X)$, we still have a continuous homomorphism, referred to as the \emph{reciprocity map}: $$\rec_X:\Gal(E^\ab/E)\isom\pibar(\Gbb_\mrm^E)\ot{\mu_X}\lra\pibar(\Gbf)$$ and the action of $\Gal(\Qac/E)$ on $\pi_0(M_K(\Gbf,X))$ is through   translation by the homomorphism $$\Gal(\Qac/E)\ra\Gal(E^\ab/E)\ra\pibar(\Gbf)\ra\pibar(\Gbf)/\Gbf(\Rbb)_+K.$$ Each connected component of $M_K(\Gbf,X)$ is defined over a finite abelian extension of $E$. 
 
 The reciprocity maps are functorial \wrt morphisms between Shimura data and between Shimura varieties.
 \end{definition-proposition}

 \begin{assumption}\label{assumption}
 In our study of special subvarieties, we will be mainly concerned with the following situation: $M=M_K(\Pbf,Y)$ is a mixed Shimura variety defined by $(\Pbf,Y)=\Wbf\rtimes(\Gbf,X)$ at some level $K=K_\Wbf\rtimes K_\Gbf$ of fine product type. We fix $Y^+=\Ubf(\Cbb)\Wbf(\Rbb)X^+$ a connected component of $Y$ lying over a connected component $X^+$ in $X$, and we have a fixed connected mixed Shimura variety $M^+=\Gamma\bsh Y^+$, with $\Gamma=\Pbf(\Qbb)_+\cap K$. 
 
 We study special subvarieties in $M^+$ that are of the form $\wp_\Gamma(Y'^+)$, coming from connected subdatum $(\Pbf',Y';Y'^+)$ of $(\Pbf,Y;Y^+)$. Similar to the case of Kuga varieties, cf. \cite{chen kuga} 2.12 and 2.13,  all special subvarieties are obtained this way, as long as one passes to different connected components of $M$ using Hecke translates, cf. \ref{morphisms of mixed Shimura varieties and Hecke translates}.  In particular, the special subvariety $\wp_\Gamma(Y'^+)$ is a connected component of the image $M_{K'}(\Pbf',Y')\ra M_K(\Pbf,Y)$ where $K'$ is some \cosg of $\Pbf'(\adele)$ contained in $K$.
 
 The fine product condition on $K$ shows that the natural projection $M_K(\Pbf,Y)\ra S:= M_{K_\Gbf}(\Gbf,X)$ has a section given by  $(\Gbf,X)\mono(\Pbf,Y)$, so $S$ is identified as a closed subscheme of $M$. Similarly, $M^+$ is fibred over $S^+=\Gamma_\Gbf\bsh X^+$ with $\Gamma_\Gbf=\Gbf(\Qbb)_+\cap K_\Gbf$ and $S^+$ is a closed subscheme of $M^+$ by the pure section.
 
 We write $E$ for the reflex field of $(\Pbf,Y)$, and we study the Galois orbits of the form $\Gal(\Qac/E)\cdot S'$ for $S'$ a special subvariety in $S^+$, as well as its mixed analogue. Note that the orbit may exceed $S^+$. We can nevertheless restrict to orbits under $\Gal(\Qac/E^+)$ with $E^+$ the field of definition of $S^+$ corresponding to the kernel of $\Gal(E^\ab/E)\ra\pibar(\Gbf)/\Gbf(\Rbb)_+K_\Gbf$. The difference is bounded by $\#\pi_0(M_{K_\Gbf}(\Gbf,X))$ which is constant when we fix $K=K_\Wbf\rtimes K_\Gbf$. 
 
 Finally, when we mention irreducible mixed Shimura (sub)data $(\Pbf,Y)$, we require that for some connected component $Y^+$ of $Y$ we have $\Pbf=\MT(Y^+)$. 
 \end{assumption}
  
 The original estimation in \cite{ullmo yafaev} 2.19 requires $\Gbf$ to be of adjoint type. We prefer the following version for general reductive $\Gbf$:

 \begin{theorem}[lower bound involving splitting fields]\label{lower bound involving splitting fields} 
 Let $S=M_K(\Gbf,X)$ be a pure Shimura variety with reflex field $E$ at some level $K\subset\Gbf(\adele)$ of fine product type. Write $\Lscr=\Lscr_K$ for the automorphic line bundle on $S$, namely the ample line bundle of top degree automorphic forms on $S$, such that a fixed positive power of $\Lscr$ defines the Baily-Borel compactification of $S$. We also fix an integer $N>0$.
 
 Let $\Tbf$ be a non-trivial $\Qbb$-torus in $\Gbf$, with splitting field $F_\Tbf$, arising as the connected center of $\Gbf'$ for some irreducible pure subdatum $(\Gbf',X')$ of $(\Gbf,X)$. Assume that the GRH holds for $F_\Tbf$. Then the following inequality holds for any $\Tbf$-special subvariety $S'\subset S$ defined by $(\Gbf',X')$: $$\deg_\Lscr(\Gal(\Qac/E)\cdot S')\geq c_ND_N(\Tbf)\cdot\prod_{p\in\Delta(\Tbf,K)}\maxx\{1,I(\Tbf,K_p)\}$$ where \begin{itemize}
 \item $D_N(\Tbf)=(\log{D(\Tbf)})^N$ with $D(\Tbf)$  the absolute discriminant of the splitting field of $F_\Tbf$;
 
 \item $\Delta(\Tbf,K)$ is the set of rational primes $p$ such that $K_{\Tbf,p}\subsetneq K_{\Tbf,p}^\maxx$, with \begin{itemize}
 \item $K_{\Tbf,p}=\Tbf(\Qbbp)\cap K_p$ 
 \item $K_{\Tbf,p}^\maxx$   the maximal \cosg of $\Tbf(\Qbbp)$
 \end{itemize}
 and $\Delta(\Tbf,K)$ is finite, i.e. $K_{\Tbf,p}=K_{\Tbf,p}^\maxx$ for all but finitely many $p$;
 
 \item $I(\Tbf,K_p)=b[K_{\Tbf,p}^\maxx:K_{\Tbf,p}]$;
 \item $c_N,b\in\Rbb_{>0}$ are constants independent of $K,\Tbf$; moreover $b$ is independent of $N$.
 \end{itemize}
 
 \end{theorem}
 
 The proof makes use of a few useful uniform bounds , among which we single out the following  (cf. \cite{ullmo yafaev} 2.4, 2.5):
 
 \begin{lemma}[uniform bounds]\label{uniform bounds} Let $(\Gbf,X)$ be a pure Shimura datum, and write $d_\Gbf$ for the dimension of $\Gbf$. Then:

 (1) If $\Tbf$ is a $\Qbb$-torus in $\Gbf$, then the degree of the splitting field $F_\Tbf$ of $\Tbf$ is uniformly bounded in terms of $d_\Gbf$.
 
 (2) If $(\Gbf',X')$ is a pure subdatum of $(\Gbf,X)$, then the degree of the reflex field $E'=E(\Gbf',X')$ is uniformly bounded in terms of $d_\Gbf$. In fact we can find a $\Qbb$-torus $\Hbf$ in $\Gbf$ whose splitting field contains $E(\Gbf',X')$.
 
 \end{lemma}
 
 
 Of course by being uniformly bounded in terms of $d_\Gbf$ we mean being less than some positive constant that only depends on $d_\Gbf$.

 \begin{proof}[Sketch of the proof of \ref{lower bound involving splitting fields}] We adapt the strategy in \cite{ullmo yafaev} Subsection 2.2 (from Definition 2.10 to Theorem 2.19) into two steps: \begin{enumerate}
 \item[Step 1.] For $(\Gbf',X')$ a general pure Shimura datum and $S'$ a connected component of $M_{K'}(\Gbf',X')$, with reflex field $E'=E(\Gbf',X')$, $\Lscr_{K'}$ the automorphic line bundle, $\Tbf$ the connected center of $\Gbf'$, and $F_\Tbf$ the splitting field of $\Tbf$, we have $$\deg_{\Lscr_{K'}}(\Gal(\Qac/E')\cdot S')\geq c'_ND_N(\Tbf)\cdot\prod_{p\in\Delta(\Tbf,K')}\maxx\{1, I(\Tbf,K'_p) \}$$ under the GRH for $F_\Tbf$,  where  $D_N(\Tbf)$ is the $N$-th power of the logarithm of the absolute discriminant of $F_\Tbf$, $\Delta(\Tbf,K')$ and $I(\Tbf,K'_p)$ are defined in the same way as in the statement of \ref{lower bound involving splitting fields} with $K$ replaced by $K'$. The definition of $I(\Tbf,K'_p)=b'[K'^\maxx_{\Tbf,p}:K'_{\Tbf,p}]$ involves some absolute constant $b'$ that only depends on the dimension $d_{\Gbf'}$ of $\Gbf'$ and some fixed faithful representation $\rho':\Gbf'\mono\GL_{\Lambda,\Qbb}$ with $\Lambda$ some free $\Zbb$-module of finite rank, and the constant $c'_N$ only depends on $d_{\Gbf'}$, $\rho'$ and $N$.
 
 Note that when $\Gbf$ is of adjoint type this is Theorem 2.19 of \cite{ullmo yafaev}. The proof for general $\Gbf$ is almost the same, and one needs to modify some estimation of absolute constants: when $\Gbf$ is of adjoint type, the Hodge structures defined by $x\in X$ on algebraic representations of $\Gbf$ (such as $\Lambda$) are of weight zero; for general $\Gbf$ the weight is not necessarily zero, but only one such representation $\Lambda$ is involved, and only finitely many Hodge types arise when $x$ runs through $X$ (and is actually independent of the choice of $x$), hence one obtains estimations similar to \cite{ullmo yafaev} 2.13 and \cite{yafaev duke} 2.13.

 \item[Step 2.] When we consider $\Tbf$-special subvarieties $f(S')$ realized as  the image of some connected component $S'$ of $M_{K'}(\Gbf',X')$ along $f:M_{K'}(\Gbf',X')\ra M_K(\Gbf,X)$ with $K'=K\cap\Gbf'(\adele)$, we have $$\deg_{\Lscr}f(S')\geq\deg_{\Lscr_{K'}}S'$$ which is adapted from \cite{klingler yafaev} 5.3.10, cf. \ref{generic injectivity} below. We use a fixed faithful representation $\rho:\Gbf\mono\GL_{\Lambda,\Qbb}$ whose restriction to $\Gbf'$ gives the representation $\rho'$ in Step 1. Note that $K'=K\cap\Gbf'(\adele)$ hence $\Delta(\Tbf,K')=\Delta(\Tbf,K)$ and $[K'^\maxx_{\Tbf,p}:K'_{\Tbf,p}]=[K^\maxx_{\Tbf,p}:K_{\Tbf,p}]$ for $p\in\Delta(\Tbf,K)$. From this the desired estimation on the degree of $\Gal(\Qac/E)\cdot f(S')$ is obtained, after replacing $c'_N$ and $b'$ by some new constants that only depend on $d_\Gbf$, $\rho$, and $N$. In particular, the constants $c_N$ and $b$ do not depend on $K$ and $\Tbf$.
 
 \end{enumerate}
 
 It should be mentioned that the estimation in Step 1 involves the morphism $\pi':M_{K'}(\Gbf',X')\ra M_{K'^m}(\Gbf',X')$ where $K'^m=K'^m_3\times\prod_{p\neq 3}K'^\maxx_p$, where \begin{itemize}
 \item $K'^\maxx_p:=K'_pK^\maxx_{\Tbf,p}$ only enlarges $K'_p$ using the maximal \cosg $K^\maxx_{\Tbf,p}$ of $\Tbf(\Qbbp)$;
 
 \item for $p=3$, $K'^m_3=K'^\maxx_3\cap K^\Lambda_3$ where $K^\Lambda_3$ is a fixed neat \cosg of $\GL_\Lambda(\Zbb_3)$, and $[K'^\maxx_3:K'^m_3]$ is bounded by some constant that only depends on the rank of $\rho'$.
 \end{itemize}
 $\pi'$ is finite \'etale of degree $[K'^m:K']$ by \cite{ullmo yafaev} 2.11, and by \cite{ullmo yafaev} 2.12 the degree of $\Gal(\Qac/F_\Tbf)\cdot S'$ is at least the degree of $\Gal(\Qac/F_\Tbf)\cdot S'\cap \pi'^\inv\pi'(S')$ times the cardinality of the $\Gal(\Qac/F_\Tbf)$-orbit of $\pi'(S')$. Under the GRH for $F_\Tbf$, this latter cardinality is at least $D_N(\Tbf)$ up to some absolute constant, following the proof of \cite{ullmo yafaev} 2.13, and the former degree is at least $I(\Tbf,K')$ up to some absolute constant, following \cite{ullmo yafaev} 2.16, 2.17, and 2.18.\end{proof}

 \begin{lemma}[generic injectivity]\label{generic injectivity} Let $(\Gbf',X')\subset(\Gbf,X)$ be an inclusion of  pure Shimura data. Assume that $(\Gbf',X')$ and $(\Gbf,X)$ are irreducible, and let $K\subset\Gbf(\adele)$ be a neat \cosg. Put $K'=K\cap\Gbf'(\adele)$, then 
 
 (1) we have $\Tbf_{\Gbf'}\supset\Tbf_\Gbf$, where $\Tbf_{\Gbf'}$ resp. $\Tbf_\Gbf$ is the connected center of $\Gbf'$ resp. of $\Gbf$;
 
 (2) the morphism $f:M_{K'}(\Gbf',X')\ra M_K(\Gbf,X)$ is generically injective;
 
 (3) $\deg_{\Lscr}f(S')\geq \deg_{\Lscr'}S'$ for $\Lscr$ resp. $\Lscr'$ the automorphic line bundle on $M_K(\Gbf,X)$ resp. on $M_{K'}(\Gbf',X')$, and $S'$ any connected component of $M_{K'}(\Gbf',X')$. 
 \end{lemma}
 
 \begin{proof}
 
 (1) Note that $\Gbf'$ is a $\Qbb$-subgroup of $\Gbf$, and it normalizes $\Gbf^\der$. Hence $\Gbf'\Gbf^\der$ is already a $\Qbb$-subgroup of $\Gbf$. Take an arbitrary $x\in X'$, we have $x(\Sbb)\subset\Gbf'_\Rbb$. Conjugate $x$ by an arbitrary $g=th\in\Gbf(\Rbb)$ with $t\in\Tbf_\Gbf(\Rbb)$ and $h\in\Gbf^\der(\Rbb)$, we have $t(x)=\Int(t)\circ x=x$ because $t$ is central in $\Gbf(\Rbb)$ and thus $g(x)=h(x)=\Int(h)\circ x$ has image in $\Gbf'\Gbf^\der$. Since $X=\Gbf(\Rbb)x$, by \ref{generating subdata} we get a subdatum $(\Gbf'\Gbf^\der,X=\Gbf'(\Rbb)\Gbf^\der(\Rbb)x)$, and the irreducibility of $(\Gbf,X)$ gives $\Gbf=\Gbf'\Gbf^\der$. Clearly we have $\Gbf'^\der\subset\Gbf^\der$, and $\Gbf'=\Tbf_{\Gbf'}\Gbf'^\der$ and $\Gbf=\Tbf_\Gbf\Gbf^\der$, hence the inclusion $\Tbf_{\Gbf}\subset\Tbf_{\Gbf'}$.

 (2) Let $\Nbf:=\Nbf_\Gbf\Gbf'$ be the normalizer of $\Gbf'$ in $\Gbf$. By the arguments in \cite{ullmo yafaev} Lemma 2.2, $\Nbf$ is reductive. 
 
 We claim that the quotient $\Nbf/\Gbf'$ is compact. Since $\Nbf\supset\Gbf'\supset\Tbf_\Gbf$, we have $\Nbf/\Gbf'=\Nbf_\ad/\Gbf'_\ad$, where $\Gbf'_\ad$ is the image of $\Gbf'$ in $\Gbf_\ad:=\Gbf/\Tbf_\Gbf$, and $\Nbf_\ad$ is the normalizer of $\Gbf'_\ad$ in $\Gbf_\ad$, equal to the reduction modulo $\Tbf_\Gbf$ of $\Nbf$, and thus the equality $\Nbf/\Gbf'=\Nbf_\ad/\Gbf'_\ad$. Since $(\Gbf'_\ad,X'_\ad)=(\Gbf',X')/\Tbf_\Gbf$ is a subdatum of $(\Gbf_\ad,X_\ad):=(\Gbf,X)/\Tbf_\Gbf$ with $\Gbf_\ad$ semi-simple, the centralizer of $\Gbf'_\ad$ in $\Gbf_\ad$ is compact because it commutes with $x(\ibf)$ for any $x\in X'_\ad$ in $\Gbf_\ad(\Rbb)$. The neutral component $\Nbf_\ad^\circ$ admits an almost direct product of the form $\Gbf'_\ad\Lbf'$, with $\Lbf$ commuting with $\Gbf'_\ad$, hence $\Lbf$ is compact, from which we get the compactness of $\Nbf/\Gbf'=\Nbf_\ad/\Gbf'_\ad$.
 
 Since $K$ is neat, we have $\Nbf(\Qbb)\cap K=\Gbf'(\Qbb)\cap K$, and repeat the arguments in \cite{ullmo yafaev} gives the generic injectivity of $f:M_{K'}(\Gbf',X')\ra M_{K}(\Gbf,X)$.

 (3) The sheaf $\Lambda_{K,K'}=f^*\Lscr\otimes\Lscr'^\inv$ is nef by \cite{klingler yafaev} 5.3.5. It remains to argue as in \cite{klingler yafaev} 5.3.10, using the generic injectivity proved in (2), because the original proof of 5.3.10 requires $\Gbf$ to be of adjoint type only for the generic injectivity of $f$ using \cite{ullmo yafaev} Lemma 2.2. \end{proof}

 The reader might be left with the impression that in order to adapt the strategy in \cite{klingler yafaev} and \cite{ullmo yafaev}  for general mixed Shimura varieties one might need to assume the GRH for all the splitting fields $F_\Tbf$ for the $\Tbf$-special subvarieties involved. Actually we have:
 
 \begin{lemma}[reduction to CM fields]\label{reduction to the CM case}  
 In order to prove the Andr\'e-Oort conjecture, it suffices work with a connected mixed Shimura variety $M=\Gamma\bsh Y^+$ defined by a connected mixed Shimura datum $(\Pbf,Y;Y^+)=\Wbf\rtimes(\Gbf,X;X^+)$ such that if $(\Pbf_1,Y_1;Y_1^+)$ is an irreducible $\TW$-special subdatum of $(\Pbf,Y;Y^+)$, then the splitting field $F_\Tbf$ of $\Tbf$ is a CM field.
 
 \end{lemma}
 
 \begin{proof}
 
 By \ref{insensitivity of isogeny}, \ref{reduction lemma}, and the discussion in \ref{reduction to subdata of a good product}, the Andr\'e-Oort conjecture is already reduced to mixed Shimura varieties defined by subdata of $(\Gbf_0,X_0)\times(\Lbf,Y_\Lbf)$, where $\Gbf_0$ is semi-simple of adjoint type, and $(\Lbf,Y_\Lbf)=\Nbf\rtimes(\Hbf,X_\Hbf)$ is a product of finitely many mixed Shimura data of Siegel type. If $(\Pbf',Y')$ is an irreducible $\TW$-special subdatum of $(\Gbf_0,X_0)\times(\Lbf,Y_\Lbf)$, then $\Tbf$ is the connected center of $\Gbf'$ coming from some irreducible subdatum $(\Gbf',X')\subset(\Gbf_0,X_0)\times(\Hbf,X_\Hbf)$, with $\Hbf$ a product of finitely many $\GSp_{\Vbf_j}$. In particular, the connected center of $\Hbf$ is a split $\Qbb$-torus. 
 
 Consider the image $(\Gbf'',X'')$ of $(\Gbf',X')$ along $(\Gbf_0,X_0)\times(\Hbf^\ad,X^\ad_\Hbf)$, where $\Hbf^\ad$ is the adjoint quotient of $\Hbf$. Since the kernel of $\Hbf\ra\Hbf^\ad$ is a split $\Qbb$-torus, and the image of $\Tbf$ in $\Gbf''$ is the connected center $\Tbf''$ of $\Gbf''$, we see that the kernel $\Tbf'$ of $\Tbf\ra\Tbf''$ is a $\Qbb$-group of multiplicative type isogeneous to a split $\Qbb$-torus. In particular, take the character group $\Xbf(\Tbf)=\Hom_{\Qac-\Gr}(\Tbf_{\Qac},\mult_{\Qac})$ we get an commutative diagram with exact rows $$\xymatrix{0\ar[r] & \Xbf(\Tbf'')\ar[r]^\subset \ar[d]^\cap & \Xbf(\Tbf)\ar[r] \ar[d]^\cap & \Xbf(\Tbf')\ar[d] &   \\ 0\ar[r] & \Xbf(\Tbf'')\otimes_\Zbb\Qbb\ar[r]^\subset &\Xbf(\Tbf)\otimes_\Zbb\Qbb\ar[r] & \Xbf(\Tbf')\otimes_\Zbb\Qbb \ar[r] & 0}$$ where the morphisms are equivariant \wrt the evident action of $\Gal(\Qac/\Qbb)$. The first two vertical maps are inclusion because $\Tbf$ and $\Tbf''$ are $\Qbb$-torus, and $\Gal(\Qac/\Qbb)$ acts on $\Xbf(\Tbf')\otimes_\Zbb\Qbb$ because $\Tbf'$ is a $\Qbb$-subgroup of the split center of $\Hbf$. Diagram chasing shows that the actions of $\Gal(\Qac/\Qbb)$ on $\Xbf(\Tbf'')$ and on $\Xbf(\Tbf)$ have the same kernel, which means $\Tbf$ and $\Tbf''$ have the same splitting field, which is a CM field by \cite{yafaev duke} 2.3. \end{proof}

 \begin{remark}[irreducible subdata]\label{irreducible subdata}
 Although we do not repeat all the proofs in \cite{ullmo yafaev},  we remark that the condition of $\Gbf$ being of adjoint type could   be relaxed into requiring $(\Gbf,X)$ to be irreducible for the first half of Section 1 of \cite{ullmo yafaev}, using our lemma \ref{generic injectivity}. In \cite{ullmo yafaev} 2.13, one needs $\Gbf$ to be of adjoint type so that the splitting fields of the connected centers of irreducible $\Tbf$-special subdata are CM fields (for $\Tbf$ non-trivial). We cannot achieve this for general $\Gbf$ unless we make use of the reduction \ref{reduction to the CM case} to modify the defining data. 
 \end{remark}

In the rest of this section, only the next theorem requires the GRH for the splitting field $F_\Tbf$ for the estimation of Galois orbits of pure special subvarieties in the $\TW$-special case. The CM version will be needed in the last section, where we work under \ref{CM splitting fields} and all the splitting fields are CM fields.

 \begin{theorem}[orbit of a pure special subvariety]\label{orbit of a pure special subvariety} Let $(\Pbf,Y)=\Wbf\rtimes(\Gbf,X)$ be a mixed Shimura datum, with $E$ its reflex field, and $M=M_K(\Pbf,Y)$ the mixed Shimura variety it defines at some level $K$ of fine product type. Write $\pi:M\ra S=M_{K_\Gbf}(\Gbf,X)$ for the natural projection and $\iota(0):S\mono M$ the pure section. Denote by $\Lscr$ the pull-back $\pi^*\Lscr_S$, with $\Lscr_S$ the canonical line bundle on $S$. We also fix an integer $N>0$.
 
 Let $M'$ be a pure special subvariety of $M$ defined by a subdatum of the form $(w\Gbf'w^\inv, wX';wX'^+)$ for some $\Tbf$-special pure subdatum $(\Gbf',X')\subset(\Gbf,X)$ with $\Tbf$ non-trivial $\Qbb$-torus and $w\in\Wbf(\Qbb)$. Then we have the following lower bound assuming the GRH for the splitting field $F_\Tbf$ of  $\Tbf$, using the same constants $c_N,b$ and the notations as in  \ref{lower bound involving splitting fields}:$$\deg_\Lscr(\Gal(\Qac/E)\cdot M')\geq c_ND_N(\Tbf)\prod_{p\in\Delta(\Tbf,K_\Gbf(w))}\maxx\{1,I(\Tbf,K_\Gbf(w))\}$$ where  
 $D_N(\Tbf)=(\log{D(\Tbf)})^N$, $K_\Gbf(w)=\{g\in K_\Gbf:wgw^\inv g^\inv\in K_\Wbf\}$ and $I(\Tbf,K_\Gbf(w)_p)=b[K_{\Tbf,p}^\maxx:K_\Tbf(w)_p]$ with $K_\Tbf(w)=K_\Gbf(w)\cap\Tbf(\adele)$.

 \end{theorem}
 Before entering the proof, we first justify some notations in the statement:
 
 \begin{lemma}
 In \ref{orbit of a pure special subvariety}, $K_\Gbf(w)_p=\Gbf(\Qbbp)\cap K_\Gbf(w)$ is a \cosg contained in $K_{\Gbf,p}$, and it is equal to $K_{\Gbf,p}$ for all but finitely many $p$'s. In particular, $K_\Gbf(w)=\prod_pK_\Gbf(w)_p$ is of fine product type.
 \end{lemma}

 \begin{proof}
 For all but finitely many $p$'s, we have $w\in\Wbf(\Qbb)\cap K_{\Wbf,p}$ and $wgw^\inv g^\inv\in K_{\Wbf,p}$ for all $g\in K_{\Gbf,p}$ as $K_\Gbf$ stabilizes $K_\Wbf$.
 
 When $w\notin K_{\Wbf,p}$, write $w=(u,v)$ for some $u\in\Ubf(\Qbb)$ and $v\in\Vbf(\Qbb)$. We have seen in \ref{group law} that $w^n=(nu,nv)$, hence $w^n\in K_\Wbf$ for $n$ a multiple of some constant integer $N>0$. In particular, the subgroup $K_\Vbf[v]$ generated by $v$ and $K_\Vbf$ in $\Vbf(\adele)$ is compact, the subgroup generated by $\psi(K_\Vbf[v],K_\Vbf[v])$, $u$ and $K_\Ubf$ is compact, and thus they generate a \cosg $K_\Wbf[w]$ of $\Wbf(\adele)$ in which $K_\Wbf$ is cofinite, and its stabilizer in $K_\Gbf$ is cofinite, hence the claim on $K_\Gbf(w)$.
 \end{proof}
 
 The proof of \ref{orbit of a pure special subvariety} is easily reduced to the following:
 
 \begin{lemma}[finite index]\label{finite index}
 The action of $\Gal(\Qac/E)$ on $M_{K^w_\Gbf}(w\Gbf w^\inv,wX)$ is identified with its action on $M_{K_\Gbf(w)}(\Gbf,X)$ where $K_\Gbf^w=w\Gbf(\adele)w^\inv\cap K_\Wbf\rtimes K_\Gbf$.
 \end{lemma}
 
 \begin{proof}
 
 We notice that $K_\Gbf^w=w\Gbf(\adele)w^\inv\cap K_\Wbf\rtimes K_\Gbf$ is equal to $wK_\Gbf(w)w^\inv$, because $(w,1)(1,g)(w^\inv,1)=(wg(w^\inv),g)=(wgw^\inv g^\inv,g)$ for $w\in\Wbf$ and $g\in\Gbf$. From the isomorphism $w\Gbf w^\inv\isom \Gbf$ we get an isomorphism of pure Shimura data $(w\Gbf w^\inv,wX)\isom(\Gbf,X)$, and thus an isomorphism of pure Shimura varieties $\lambda:M(w):=M_{K_\Gbf^w}(w\Gbf w^\inv,wX)\isom S(w):=M_{K_\Gbf(w)}(\Gbf,X)$. From the equality $w(K_\Wbf\rtimes K_{\Gbf}(w))w^\inv=wK_\Wbf w^\inv\rtimes K_\Gbf^w$ as \cosgs in $\Pbf(\adele)$, we deduce that the Hecke translation $\tau_w$ (defined in  \ref{morphisms of mixed Shimura varieties and Hecke translates}(2)), namely the isomorphism $$\tau_w:M_{wK_\Wbf w^\inv\rtimes K_\Gbf^w}(\Pbf,Y)=M_{w(K_\Wbf\rtimes K_\Gbf(w))w^\inv}(\Pbf,Y)\ra M_{K_\Wbf\rtimes K_\Gbf(w)}(\Pbf,Y)$$ restricts to the isomorphism $\lambda:M(w)\isom S(w)$ above, where $M(w)$ resp. $S(w)$ is regarded as a pure section of $M_{wK_\Wbf w^\inv\rtimes K_\Gbf^w}(\Pbf,Y)$ resp. of $M_{K_\Wbf\rtimes K_\Gbf(w)}(\Pbf,Y)$ using the equality $(\Pbf,Y)=\Wbf\rtimes(w\Gbf w^\inv,wX)$ resp. $(\Pbf,Y)=\Wbf\rtimes(\Gbf,X)$. Since the Hecke translation $\tau_w$ is an isomorphism defined over the common reflex field $E=E(\Gbf,X)$,  it transports the action of $\Gal(\Qac/E)$ on $S(w)$ to that on $M(w)$, hence the claim. \end{proof}
  
 \begin{proof}[Proof of \ref{orbit of a pure special subvariety}]
 
 We have the commutative diagram:$$\xymatrix{M(w)\ar[r]^{\iota(w)} \ar[d]^\lambda & M\ar[d]^\pi\\ S(w)\ar[r]^f & S=M(0)}$$ with $\iota(w)$ the inclusion of the pure special subvariety $M(w)\mono M$. Since $\lambda$ is an isomorphism, the degree of a closed subvariety $Z$ in $M(w)$ against $\iota(w)^*\pi^*\Lscr_S$ is the same as the degree of $\lambda(Z)\subset S(w)$ against $f^*\Lscr_S=\Lscr_{S(w)}$, with the last equality by \cite{klingler yafaev} 5.3.2(1). Taking $Z$ to be the pure special subvariety defined by $(w\Gbf'w^\inv,wX';WX'^+)$, then its image under $\lambda$ is the special subvariety in $S(w)$ defined by $(\Gbf',X';X'^+)$. Using \ref{lower bound involving splitting fields} with $K_\Gbf$ replaced by $K_\Gbf(w)$ we get $$\deg_{\Lscr}(\Gal(\Qac/E)\cdot M')=\deg_{\Lscr_{S(w)}}(\Gal(\Qac/E)\cdot\lambda(M'))\geq c_ND_N(\Tbf)\prod_{p\in\Delta(\Tbf,K_\Gbf(w))}\maxx\{1,I(\Tbf,K_\Gbf(w)) \}.$$ \end{proof}

 \begin{remark}[unipotent Hecke translation]\label{unipotent Hecke translation}
 

  We have seen in the proof of \ref{finite index} that the Hecke translation $\tau_w$ (namely conjugation by $w^\inv$ cf. \ref{morphisms of mixed Shimura varieties and Hecke translates}(2)) gives $w(K_\Wbf\rtimes K_\Gbf)w^\inv\isom (wK_\Wbf w^\inv)\rtimes K_\Gbf^w$ under the assumption $K_\Gbf=K_\Gbf(w)$, hence the isomorphism $M_{wKw^\inv}(\Pbf,Y)\ra M_K(\Pbf,Y)$ sending the pure section given by $(w\Gbf w^\inv, wX)$ to the one given by $(\Gbf,X)$. 
 
 In particular, if $\psi=0$, then $\Wbf$ is commutative, which gives us equalities $wK_\Wbf w^\inv=K_\Wbf$ and $w(K_\Wbf\rtimes K_\Gbf)w^\inv\isom wK_\Wbf w^\inv\rtimes K_\Gbf^w$. In this case, if we have $K_\Gbf=K_\Gbf(w)$, then   conjugation by $w^\inv$ defines an automorphism of $M_K(\Pbf,Y)$ translating the pure section $M_{wK_\Gbf w^\inv}(w\Gbf w^\inv, wX)$ to $M_{K_\Gbf}(\Gbf,X)$. 

 If we are in the Kuga case $\Ubf=0$, then we are again led to the picture of torsion sections of abelian schemes. The natural projection $\pi:M=M_{K_\Vbf\rtimes K_\Gbf}(\Vbf\rtimes\Gbf,\Vbf(\Rbb)X)\ra S=M_{K_\Gbf}(\Gbf,X)$ is an abelian $S$-scheme. If $v\in\Vbf(\Qbb)$ is of order $n$ in $\Vbf(\adele)/K_\Vbf$, then the maximal pure Shimura subvariety defined by $(v\Gbf v^\inv, vX)$ is contained in $M[n]$ the $n$-torsion part of $\pi:M\ra S$. It is a section to $\pi$ \ifof $K_\Gbf=K_\Gbf(v)$. Similarly, if $\Vbf=0\neq \Ubf$, then $\pi:M=M_{K_\Ubf\rtimes K_\Gbf}(\Ubf\rtimes\Gbf,\Ubf(\Cbb)X)\ra S=M_{K_\Gbf}(\Gbf,X)$ is an $S$-torus, and in this case an element $u\in\Ubf(\Qbb)$ satisfying $K_\Gbf=K_\Gbf(u)$ gives a torsion section $M(u)$ using formulas similar to the Kuga case. This also allows us to talk about the case $\psi=0$ where $\Wbf\isom\Ubf\oplus\Vbf$ is a commutative unipotent $\Qbb$-group, and the results are parallel.
 
 \end{remark}

 We want to take a closer look at the term $I(\Tbf,K_\Gbf(w))$. In \cite{ullmo yafaev} 3.15, the faithful representation $\rho:\Gbf\ra\GL_{\Lambda,\Qbb}$ leads to an inequality $I(\Tbf,K_p)\geq c\cdot p$ for all prime $p$ such that $p$ is unramified in the splitting field of $\Tbf$ and that $K_p=\GL_{\Lambda}(\Zbbp)\cap\Gbf(\Qbbp)$. In our case, besides the representation $\rho$ we have further $\rho_\Ubf$ and $\rho_\Vbf$ in the definition of mixed Shimura data, and we want to show that for $p\in\Delta(\Tbf,K_\Gbf(w))$ we have further $I(\Tbf,K_\Gbf(w)_p)\geq c\cdot {\ord_p(w,K_\Wbf)}$, using the following definition:

 \begin{definition}[torsion order] 
 
 Let $\Wbf$ be the unipotent radical of $\Pbf$ from some mixed Shimura datum $(\Pbf,Y)=\Wbf\rtimes(\Gbf,X)$, which is a central extension of $\Vbf$ by $\Ubf$ via $\psi$. The \emph{order} of $w$ \wrt a \cosg $K_\Wbf\subset\Wbf(\adele)$ is the smallest integer $n>0$ such that $w^N\in K_\Wbf$ for all $N\in n\Zbb$.
 
 If $K_\Wbf$ is the subgroup of $\Wbf(\adele)$ generated by \cosgs $K_\Ubf\subset\Ubf(\adele)$ and $K_\Vbf\subset\Vbf(\adele)$ via $\psi$ (satisfying $\psi(K_\Vbf\times K_\Vbf)\subset K_\Ubf$), then by writing $w=(u,v)$ for $u\in\Ubf(\Qbb)$ and $v\in\Vbf(\Qbb)$ we see that $\ord(w,K_\Wbf)$ is the least common multiple of $\ord(u,K_\Ubf)$ and $\ord(v,K_\Vbf)$, where $\ord(u,K_\Ubf)$ is the order of the class $[u]$ in the torsion abelian group $\Ubf(\adele)/K_\Ubf$, and similarly for $\ord(v,K_\Vbf)$. The order for $\Wbf(\adele)/K_\Wbf$ is thus well-defined, although the quotient is only a pointed set with an action of $\Zbb$, rather than a group. 
 
 If the $K_\Ubf$ and $K_\Vbf$ above are of fine product type, then we have $\ord(w,K_\Wbf)=\prod_p\ord_p(w,K_\Wbf)$, where $\ord_p(w,K_\Wbf)$ is the order of $w$ in $\Wbf(\Qbbp)/K_{\Wbf,p}$. The product makes sense because $w\in K_{\Wbf,p}$ for all but finitely many $p$'s, i.e. $\ord_p(w,K_\Wbf)=1$ for almost every $p$. We also have $\ord_p(w,K_\Wbf)=\maxx\{\ord_p(u,K_\Ubf),\ord_p(v,K_\Vbf)\}$ when we write $w=(u,v)$. Note that  the abelian groups $\Ubf(\Qbbp)/K_{\Ubf,p}$ and $\Vbf(\Qbbp)/K_{\Vbf,p}$ are $p$-torsion groups, the torsion orders $\ord_p(u,K_\Ubf)$ and $\ord_p(v,K_\Vbf)$ are $p$-powers, hence so it is with $\ord_p(w,K_\Wbf)$.
 
 \end{definition}

 \begin{example}[torsion order in the Siegel case]\label{torsion order in the siegel case}
 
 We first consider the torsion order in the Siegel case, using the mixed Shimura datum defined in \ref{Siegel data}. We have a symplectic form $\psi:\Vbf\rtimes\Vbf\ra\Ubf$ with $\Ubf=\Gaa$, the extension $\Wbf$ of $\Vbf$ by $\Ubf$, and we have the pure Shimura datum $(\Gbf,X)=(\GSp_\Vbf,\Hscr_\Vbf)$, as well as the mixed Shimura datum $(\Pbf,Y)=\Wbf\rtimes(\Gbf,X)$, the reflex field of which is $\Qbb$. Taking a \cosg $K=K_\Wbf\rtimes K_\Gbf$ of fine product type, we have the mixed Shimura datum $M=M_K(\Pbf,Y)$ fibred over $S=M_{K_\Gbf}(\Gbf,X)$.
 
 Note that $(\GSp_\Vbf,\Hscr_\Vbf)$ is irreducible. In fact $\Sp_\Vbf$ is the minimal $\Qbb$-subgroup of $\GSp_\Vbf$ whose base change to $\Rbb$ contains all $x(\Sbb^1)$ for $x\in\Hscr_\Vbf$ and $\Sbb^1=\Ker(\Nm_{\Cbb/\Rbb}:\Sbb\ra\Gbb_{\mrm,\Rbb})$, otherwise the Hermitian symmetric domain $\Hscr_\Vbf^+$ could be produced from $\Hbf(\Rbb)^+$ by some smaller reductive $\Qbb$-subgroup $\Hbf\subsetneq\Sp_\Vbf$ using \ref{generating subdata}, contradicting the classification of simple Hermitian symmetric spaces; and for any $x\in\Hscr_\Vbf$, the image of $\Gbb_{\mrm,\Rbb}$ in $\Sbb$ (corresponding to $\Rbb^\times\subset\Cbb^\times$) along $x$ is a central $\Rbb$-torus of $\GSp_{\Vbf,\Rbb}$, and coincides with the center of $\GSp_{\Vbf,\Rbb}$, hence $\GSp_\Vbf$ is already the generic Mumford-Tate group of $\Hscr_\Vbf$.  Let $(\Gbf',X')$ be an irreducible pure subdatum of $(\Gbf,X)$, with $\Tbf$ the connected center of $\Gbf'$. Since $(\GSp_\Vbf,\Hscr_\Vbf)$ is already irreducible, by \ref{generic injectivity}(1) we see that  $\Tbf$ contains $\Gbb_\mrm$ the center of $\GSp_\Vbf$.

 The center   $\mult$ of $\Gbf=\GSp_\Vbf$   acts on $\Vbf$ by central scaling, and it acts on $\Ubf$ by the square of central scaling. Taking a \cosg $K_\Gbf$ of $\Gbf(\adele)$ of the form $K_\mult K_{\Gbf^\der}$ and $K_\Ubf$, $K_\Vbf$, $K_\Wbf$ as in \ref{levels of product type}, we see that $K_\Tbf=\Tbf(\adele)\cap K_\Gbf\supset K_\mult$. 
 
 In particular, for any $w\in\Wbf(\Qbb)$ and $p\in\Delta(\Tbf,K_\Gbf(w))$ we have $$[K^\maxx_{\Tbf,p}:K_\Tbf(w)_p]\geq [K_{\mult,p}^\maxx:K_\mult(w)_p].$$ Write $w=(u,v)$ for $u\in\Ubf(\Qbb)$ and $v\in\Vbf(\Qbb)$. $K^\maxx_{\mult,p}\isom\Zbbptimes$ acts on $\Ubf(\Qbbp)/K_{\Ubf,p}$ by automorphism $t(\ubar)=\overline{t^2u}$ where $t\in\Zbbptimes$ and the bar stands for the class modulo $K_{\Ubf,p}$. The action preserves the order $\ord_p(u,K_\Ubf)$, and it stabilizes the image of $\Zbbptimes u$ modulo $K_\Ubf$, which is isomorphic to $\Zbbp/p^m$ with $p^m=\ord_p(u,K_\Ubf)$. Hence the cardinality of the quotient $[K^\maxx_{\mult,p}:K_{\mult,p}]$ is equal to $$\#\{t^2:t\in(\Zbbp/p^m)^\times\}=\frac{1}{2}\#(\Zbbp/p^m)=\frac{p-1}{2p}\ord_p(u,K_\Ubf).$$

 The case of $\Vbf$ is similar: $\mult$ is the common center of $\GL_\Vbf$ and  $\GSp_\Vbf$, hence $$[K_{\mult,p}^\maxx:K_\mult(w)_p]\geq\frac{p-1}{p}\ord_p(v,K_\Vbf)$$ and combining it with the case of $\Ubf$ we get $$[K^\maxx_{\Tbf,p}:K_\Tbf(w)_p]\geq\frac{1}{4}\ord_p(w,K_\Wbf).$$
 \end{example}
 
 \begin{proposition}[torsion order in the product case]\label{torsion order in the product case} Let $(\Pbf,Y)=\Wbf\rtimes(\Gbf,X)$ be an irreducible mixed Shimura subdatum of a product datum of finitely many mixed Shimura data of Siegel type (including Kuga data) $(\Lbf,Y_\Lbf)=\Nbf\rtimes(\Hbf,X_\Hbf)=\prod_j(\Pbf_j,\Uscr_j)\times(\Qbf,\Vscr)$, with $(\Gbf,X)\subset(\Hbf,X_\Hbf)$. Let $K=K_\Nbf\rtimes K_\Hbf$ be a \cosg of fine product type restricting to $K_\Pbf=K_\Wbf\rtimes K_\Gbf$.
 
 Let $(\Pbf',Y'):=\Wbf'\rtimes(w\Gbf'w^\inv,wX')$ be an irreducible $\TW$-special subdatum of $(\Pbf,Y)$, with $w\in\Wbf(\Qbb)$ non-trivial. Keeping the notations as in \ref{orbit of a pure special subvariety}, we have:
 
 (1) for $p\in\Delta(\Tbf,K_\Gbf(w))$ we have $I(\Tbf,K_\Gbf(w)_p)\geq c\cdot\ord_p(w,K_\Wbf)$, $c$ being some constant independent of $K$ and $\Tbf$;
 
 (2) there is a constant integer $N>0$ such that $w^N\in K_{\Wbf,p}$ for $p\notin\Delta(\Tbf,K_\Gbf(w))$, and $N$ is independent of $(\Gbf',X')$, $\Tbf$, and $K$.
 
 \end{proposition}

 \begin{proof}
 
 Since $(\Pbf',Y')=\Wbf'\rtimes(w\Gbf'w^\inv,wX')$ is irreducible, $(\Gbf',X')$ is irreducible by \ref{pure irreducibility}. Take any connected component $X'^+$ of $X'$, we have $\Gbf'':=\MT(X'^+)$ with $\Gbf''^\der=\Gbf'^\der$, and strictly irreducible data $(\Gbf'',X'')$ with $X''=\Gbf''(\Rbb)X'^+$ and  $(\Pbf'',Y'')=\Wbf'\rtimes(w\Gbf''w^\inv,wX'')$. $(\Gbf'',X'')$ being a strictly irreducible pure Shimura subdatum of $(\Hbf,X_\Hbf)=(\GSp_{\Vbf_1},\Hscr_{\Vbf_1})\times\cdots\times(\GSp_{\Vbf_n},\Hscr_{\Vbf_n})$, we have $X''^+=X'^+\subset X_\Hbf^+$ for some connected component $X_\Hbf^+$ of $X_\Hbf$, and this gives an inclusion of strictly irreducible pure Shimura data $(\Gbf'',X'')\subset(\Hbf',X'_\Hbf=\Hbf'(\Rbb)X^+_\Hbf)$, where $\Hbf'$ is the $\Qbb$-subgroup generated by $\Hbf^\der=\prod\Sp_{\Vbf_j}$ and a central split $\Qbb$-torus $\mult_m$ which acts on each $\Vbf_j$ by the central scaling. From \ref{generic injectivity}(1) we get $\mult\subset\Tbf'$ with $\Tbf'$ the connected center of $\Gbf'$. Since $\Gbf'\subset\Gbf$ and $\Gbf'^\der=\Gbf^\der$, we get $\mult\subset\Tbf'\subset\Tbf$.
 
 We write $\Cbf$ for this split $\Qbb$-torus so as to avoid ambiguities with other $\Gbb_\mrm$, like those arising as the connected centers of the $\GSp_{\Vbf_j}$. 
 
 (1) The inclusion $\Cbf\subset\Tbf$ gives $K_\Cbf(w)\subset K_\Tbf(w)\subset K_\Gbf(w)$, and for any prime $p$ we have $K_\Cbf(w)_p\subset K_\Tbf(w)_p\subset K_\Gbf(w)_p$ and $K^\maxx_{\Cbf,p}\subset K^\maxx_{\Tbf,p}$. The cardinality $K^\maxx_{\Cbf,p}/K_\Cbf(w)_p$ resp. $K^\maxx_{\Tbf,p}/K_\Tbf(w)_p$ equals the cardinality of the $K^\maxx_{\Cbf,p}$-orbit resp. the $K^\maxx_{\Tbf,p}$-orbit of the class of $w$ in the quotient set $K_{\Wbf,p}\bsh \Wbf(\Qbbp)$, hence for $p\in\Delta(\Tbf,K_\Gbf(w))$ i.e. $K^\maxx_{\Tbf,p}\supsetneq K_\Tbf(w)_p$ we have $$[K^\maxx_{\Tbf,p}:K_\Tbf(w)_p]\geq [K^\maxx_{\Cbf,p}:K_{\Cbf}(w)_p]\geq \frac{1}{4}\ord_p(w,K_\Wbf)$$ where the last inequality follows from \ref{torsion order in the siegel case}: although $w$ comes from $\Wbf(\Qbb)$ with $\Wbf$ an extension of $\Vbf$ by $\Ubf$, what matters is that $K^\maxx_{\Cbf,p}=\Zbbptimes$ acts on $\Vbf(\Qbbp)$ by the central scaling and on $\Ubf(\Qbbp)$ by the square of central scaling, and the estimation in \ref{torsion order in the siegel case} applies. It remains to put $c=b/4$ using $I(\Tbf,K_\Gbf(w)_p)=b[K^\maxx_{\Tbf,p}:K_\Tbf(w)_p]$.
 
 (2) For $p\notin\Delta(\Tbf,K_\Gbf(w))$ we have $K^\maxx_{\Tbf,p}=K_{\Tbf,p}=K_{\Tbf}(w)_p$, and in this case $K^\maxx_{\Cbf,p}=K_{\Cbf,p}\isom\Zbbptimes$: if $K_{\Cbf,p}=K_{\Tbf,p}\cap\Cbf(\Qbbp)$ is not maximal, then we get $K^\maxx_{\Cbf,p}K_{\Tbf,p}\supsetneq K^\maxx_{\Tbf,p}$ which is absurd. 
 For any $g\in K_{\Tbf,p}$, we have $wgw^\inv g^\inv\in K_{\Wbf,p}$. Take $g=t\in K_{\Cbf,p}=\Zbbptimes$ we have $$wgw^\inv g^\inv=(u-t^2u-\psi(v,tv),v-tv,1)=(u-t^2u,v-tv,1)\in K_{\Wbf,p}$$ where we use $\psi(v,tv)=t\psi(v,v)=0$.

 Concerning the term $v-tv=(1-t)v$: for $p\geq 3$, we may take $t=2\in\Zbb_p^\times$, and $v-tv\in K_\Vbf$ implies $v\in K_\Vbf$; for $p=2$, we still have $3\in\Zbb_p^\times$, which gives $2v\in K_\Vbf$.
 
 Concerning the term $u-t^2u=(1-t^2)u$: for $p\geq 5$, the subgroup $\{t^2:t\in\Zbb_p^\times\}$ is a subgroup of index 2 in $\Zbb_p$, which contains $1+p\Zbb_p$ as a proper subset. In particular we can find $t\in\Zbb_p^\times$ such that $1-t^2$ is a unit in $\Zbb_p$, hence $u\in K_\Ubf$ in this case. For $p\leq 3$,  we can still find $t\in\Zbb_p^\times$ such that $1-t^2$ divides 12. 
 
 Hence it suffices to take $N=12$.
  \end{proof}
  
 \begin{corollary}[torsion order in the embedded case]\label{torsion order in the embedded case} Let $(\Pbf,Y)=\Wbf\rtimes(\Gbf,X)=(\Ubf,\Vbf)\rtimes(\Gbf,X)$ be a subdatum of a product of the form $(\Gbf_0,X_0)\times(\Lbf,Y_\Lbf)$ with 
 $\Gbf_0$ of adjoint type, $(\Lbf,Y_\Lbf)=\Nbf\rtimes(\Hbf,X_\Hbf)=(\Ubf_\Nbf,\Vbf_\Nbf)\rtimes(\Hbf,X_\Hbf)$ a finite product of mixed Shimura data of Siegel type and $(\Gbf,X)\subset(\Gbf_0,X_0)\times(\Hbf,X_\Hbf)$. We write $(\Hbf,X_\Hbf)=(\GSp_{\Vbf_1},\Hscr_{\Vbf_1})\times\cdots\times(\GSp_{\Vbf_n},\Hscr_{\Vbf_n})$.  Let $K=K_\Nbf\rtimes(K_{\Gbf_0}\times K_\Hbf)$ be a \cosg of fine product type, which restricts to $K_\Pbf=K_\Wbf\rtimes K_\Gbf$.
 
 If $(\Pbf',Y')=\Wbf'\rtimes(w\Gbf'w^\inv,wX')$ is a strictly irreducible $\TW$-subdatum of $(\Pbf,Y)$ for some $w\in\Wbf(\Qbb)$, then using the notations in \ref{orbit of a pure special subvariety} we have:
 
 (1)$I(\Tbf,K_\Gbf(w)_p)\geq c\cdot \ord_p(w,K_\Wbf)$, $c$ being some constant independent of $(\Gbf',X')$, $K$, and $\Tbf$;
 
 (2) there is a constant integer $N>0$ such that $w^N\in K_{\Wbf,p}$ for $p\notin\Delta(\Tbf,K_\Gbf(w))$, and $N$ is independent of $(\Gbf',X')$, $K$, and $\Tbf$.
 
 \end{corollary} 
 
 \begin{proof}
 
 Let $x=(x_0,x_1)$ be a point in $X_0\times X_\Hbf$ from the datum $(\Gbf_0,X_0)\times(\Hbf,X_\Hbf)$, and let $\mult_\Rbb\subset\Sbb$ be the split $\Rbb$-torus corresponding to $\Rbb^\times\subset\Cbb^\times$. Then $x_1(\mult_\Rbb)$ acts on $\Vbf_{\Nbf,\Rbb}$ via the central scaling and on $\Ubf_{\Nbf,\Rbb}$ via the square of the central scaling, while $x_0(\mult_\Rbb)$ is trivial, because by \ref{mixed Shimura data}(i) $x_0$ sends $\mult_\Rbb$ into the center of $\Gbf_0$ which is trivial. Let $\Cbf$ be the split $\Qbb$-torus isomorphic to $\mult$ in $\Hbf=\GSp_{\Vbf_1}\times\cdots\times\GSp_{\Vbf_n}$ which is the diagonal of $\mult\times\cdots\times\mult$, as we have seen in \ref{Siegel data}(3) (we write $\Cbf$ instead of $\mult$ so as to avoid ambiguities with the connected centers of $\GSp_{\Vbf_j}$). $\Cbf$ is one-dimensional, and we have shown that for any given $x\in X_0\times X_\Hbf$,  $\Cbf$ is the minimal $\Qbb$-subtorus in $\Gbf_0\times\Hbf$ whose base change to $\Rbb$ contains $x(\mult_\Rbb)$, and in fact $\Cbf_\Rbb=x(\mult_\Rbb)$.
 
 When $x$ runs through $X'$, the generic Mumford-Tate group $\Gbf'=\MT(X'^+)$ for some connected component $X'^+$ of $X'$ necessarily contains $\Cbf$ because $\Gbf'_\Rbb\supset x(\mult_\Rbb)=\Cbf_\Rbb$. Clearly $\Cbf$ is central in $\Gbf_0\times\Hbf$, hence it is central in $\Gbf'$, and we get $\Cbf\subset\Tbf$. It remains to argue as in \ref{torsion order in the product case}.
 
 Note that the constants $c$ and $N$ are the same as in \ref{torsion order in the product case}.\end{proof}

 \begin{remark}[torsion order in general]\label{torsion order in general}
 
 (1) The corollary \ref{torsion order in the embedded case} above is sufficient for the study of Andr\'e-Oort conjecture for mixed Shimura varieties using the reductions \ref{insensitivity of isogeny} and \ref{reduction lemma}. For a general mixed Shimura datum $(\Pbf,Y)=\Wbf\rtimes(\Gbf,X)$ mapped into $(\Gbf_0,X_0)\times(\Lbf,Y_\Lbf)$ with finite kernel, where $(\Lbf,Y_\Lbf)=\Nbf\rtimes(\Hbf,X_\Hbf)$ is a product of mixed Shimura data of Siegel type with $(\Gbf,X)$ mapped into $(\Hbf,X_\Hbf)$, the image of a strictly irreducible $\TW$-special subdatum $(\Pbf',Y')=\Wbf'\rtimes(w\Gbf'w^\inv,wX')$ in $(\Lbf,Y_\Lbf)$ is a strictly irreducible subdatum $(\Pbf'',Y'')=\Wbf''\rtimes(w''\Gbf''w''^\inv,w''X'')$, and the connected center $\Tbf''$ of $\Gbf''$, which is the image of $\Tbf$, contains $\Cbf$ the split central $\Qbb$-torus constructed in the product case. The pre-image of $\Cbf$ in $\Tbf$ must contain a split $\Qbb$-torus $\Cbf'$ which is mapped onto $\Cbf$. We do have $[K^\maxx_{\Tbf,p}:K_\Tbf(w)_p]\geq [K^\maxx_{\Cbf',p}:K_{\Cbf'}(w)_p]$, but $[K^\maxx_{\Cbf',p}:K_{\Cbf'}(w)_p]\geq c'[K^\maxx_{\Cbf,p}:K_\Cbf(w)_p]$ is not evident.   We  have an exact sequence $$1\ra\Cbf''\ra\Cbf'\ra\Cbf\ra 1$$ with $\Cbf''$ a finite $\Qbb$-group of degree $d$ (i.e. associated to a commutative  Hopf $\Qbb$-algebra of dimension $d$) and we have the exactness of $$1\ra\Cbf''(\Qbbp)\ra\Cbf'(\Qbbp)\ra\Cbf(\Qbbp)\ra H^1(\Gal(\Qac_p/\Qbbp),\Cbf''(\Qac_p).$$ Here the Galois cohomology $H^1(\Qbbp,\Cbf'')$ is of $d$-torsion, which implies that $d$-powers of elements in $K^\maxx_{\Cbf,p}=\Zbbptimes$ fall in the image of $K^\maxx_{\Cbf',p}$. Using the exponential map \cite{neukirch number theory} II.5.5, we can find an open subgroup $U'_p$ in $U_p=1+2p\Zbbp\subset\Zbbptimes$ with index $[U_p:U'_p]$ uniformly bounded independent of $p$. However $[\Zbbptimes:U_p]=p-1$ for $p\geq 3$, and using these arguments we can only arrive at an estimation of the form $[K^\maxx_{\Tbf,p}:K_\Tbf(w)_p]\geq \frac{c}{p}\ord_p(w,K_\Wbf)$ for some absolute constant $c$, and this is not sufficient for some results in the last section concerning bounded equidistribution. It is for this reason that we do not proceed further for an estimation of torsion order in general.

 (2) Some of the arguments in   \ref{finite index} and  in \ref{torsion order in the siegel case} have appeared in \cite{chen thesis} 4.2. GAO Ziyang has informed us of his work \cite{gao mixed}, where he obtained similar results independently. His work concentrates on Galois orbits of special points, and directly leads to the Andr\'e-Oort conjecture without repeating the arguments for special subvarieties, due to the powerful machinery of o-minimality. He has restricted to the case that the pure part of the mixed Shimura varieties in question are given by subdata of Siegel type. Using our arguments it is natural to expect similar results  valid for general mixed Shimura data that are subdata in some $(\Gbf_0,X_0)\times(\Lbf,Y_\Lbf)$ as in the corollary above.
 \end{remark}

 \section{Special subvarieties with bounded test invariants}
 
 In this section we prove the equivalence between bounded sequences of special subvarieties and sequences with bounded test invariants, where the notion of test invariants is introduced as a substitute to the estimation of degrees of lower bounds for Galois orbits. We will also draw some conclusions under the GRH for CM fields, when the following assumption is satisfied:
 
 \begin{assumption}[CM splitting fields]\label{CM splitting fields}
 The ambient mixed Shimura variety $M_K(\Pbf,Y)$ is defined by a mixed Shimura datum $(\Pbf,Y)=\Wbf\rtimes(\Gbf,X)$ of some datum of the form $(\Gbf_0,X_0)\times(\Lbf,Y_\Lbf)$ where $\Gbf_0$ is semi-simple of adjoint type and $(\Lbf,Y_\Lbf)=\Nbf\rtimes(\Hbf,\Nbf)$ is a finite product of mixed Shimura data of Siegel type (including Kuga ones). It is also required that $(\Gbf,X)\mono(\Gbf_0,X_0)\times(\Hbf,X_\Hbf)$.
 
 \end{assumption}
 
 When this assumption holds, for any irreducible $\TW$-special subdatum $(\Pbf',Y')$ of $(\Pbf,Y)$, the splitting field $F_\Tbf$ of $\Tbf$ is a CM field as we have seen in \ref{reduction to the CM case}, and the Andr\'e-Oort conjecture can be reduced to this case.
   
 We first recall the following criterion for the ergodic-Galois alternative in the pure case:
 \begin{theorem}[bounded Galois orbits, cf. \cite{ullmo yafaev} 3.10]\label{bounded galois orbits} Let $S=M_K(\Gbf,X)$ be a pure Shimura variety, with $E$ its reflex field, and $(S_n)$ a sequence of special subvarieties, defined by $(\Gbf_n,X_n)$. Write $\Tbf_n$ for the connected center of $\Gbf_n$, and $D(\Tbf)_n$ the absolute discriminant of its splitting field over $\Qbb$. If there exists some constant $C>0$ such that $$\log D(\Tbf_n)\cdot\prod_{p\in\Delta(\Tbf_n,K)}\maxx\{1,I(\Tbf_n,K_p)\}\leq C, \forall n\in\Nbb$$ then there exists finitely many $\Qbb$-tori $\{\Cbf_1,\cdots,\Cbf_N\}$ in $\Gbf$ such that each $S_n$ is $\Cbf_i$-special for some $i$.

 In particular, assume that \ref{CM splitting fields} is satisfied, which means, in the pure case, that $(\Gbf,X)$ is a subdatum of a product $(\Gbf_0,X_0)\times(\Hbf,X_\Hbf)$ with $(\Hbf,X_\Hbf)$ a product of finitely many pure Shimura data of Siegel type, and assume the GRH for CM fields, if $\deg_{\pi^*\Lscr_S}\Gal(\Qac/E)\cdot S_n\leq C, \forall n$ for some constant $C>0$, $\Lscr_S$ being the automorphic line bundle on $S$, then the sequence $(S_n)$ is bounded in the sense of \ref{bounded sequence}.
 \end{theorem} 
 
 The original statement was made assuming $\Gbf$ being of adjoint type, and from \ref{lower bound involving splitting fields} and \ref{reduction to the CM case} the modified form above is immediate. Note that the arguments in \cite{ullmo yafaev} 3.13 - 3.21 do not rely on the GRH for CM fields, and the GRH is not involved for sequences of pure special subvarieties with bounded test invariants. It is used when we pass to  Galois orbits of bounded degrees.
 
 For a sequence of pure special subvarieties in a mixed Shimura variety we immediately get:
 \begin{corollary}[pure special subvarieties of bounded Galois orbits]\label{pure special subvarieties of bounded Galois orbits} Let $M=M_K(\Pbf,Y)$ be a mixed Shimura variety satisfying \ref{CM splitting fields}, which is defined over the reflex field $E=E(\Pbf,Y)$ as in \ref{assumption} with $\Lscr_S$ the automorphic line bundle on $S$, and let $(M_n)$ be a sequence of pure special subvarieties, defined by irreducible $(\Tbf_n,w_n)$-special pure subdata $(\Pbf_n,Y_n)=\Wbf_n\rtimes(w_n\Gbf_nw_n^\inv,w_nX_n)$. If there exists some constant $C>0$ such that  $$\log{}D(\Tbf_n)\prod_{p\in\Delta(\Tbf_n,K_\Gbf(w_n))}\maxx\{1,I(\Tbf_n,K_\Gbf(w_n)_p)\}\leq C,\ \forall n\in\Nbb$$ then the sequence is $B$-bounded for some $B$ in the sense of \ref{bounded sequence}. The analytic closure of $\bigcup_nM_n$ is a finite union of pure special subvarieties bounded by $B$.
 
 In particular, under   the GRH for CM fields, a sequence of pure special subvarieties $(M_n)$ such that   $$\deg_{\pi*\Lscr_S}\Gal(\Qac/E)\cdot M_n\leq C,\ \forall n$$  for some constant $C>0$ is bounded by some $B$, and the analytic closure of $\bigcup_nM_n$ is a finite union of $B$-bounded pure special subvarieties.  
 \end{corollary}
 
 \begin{proof}
 Write $\pi:M\ra S=M_{K_\Gbf}(\Gbf,X)$ for the natural projection. Then the images $S_n=\pi(M_n)\subset S$ is a sequence of pure special subvarieties, with $S_n$ being $\Tbf_n$-special. It is clear that $\Delta(\Tbf_n,K_\Gbf)\subset\Delta(\Tbf_n,K_\Gbf(w_n))$, and thus $$\log D(\Tbf_n)\prod_{p\in\Delta(\Tbf_n,K_\Gbf)}\maxx\{1,I(\Tbf_n,K_\Gbf)\}\leq C$$ which implies that the sequence $(S_n)$ is bounded. In particular, we can choose the defining subdatum for $(M_n)$ to be $(w_n\Gbf_nw_n^\inv,w_nX_n)$ such that the connected centers of the $\Gbf_n$ come from a fixed finite set $\{\Tbf_\alpha:\alpha\in A\}$ ($A$ finite), and   the absolute discriminants $D(\Tbf_\alpha)$ assume only finitely many values.
 
 Therefore the sequence $I(\Tbf_n,K_\Gbf(w_n))$ is also bounded by some constant $C'$, and $\ord_p(w_n,K_\Wbf)\leq C'/c$ for $p\in\Delta(\Tbf_n,K_\Gbf(w_n))$ where $c$ is the constant in \ref{torsion order in the embedded case}. We see that  
 
 \begin{itemize}
 \item for $p\in\Delta(\Tbf_n,K_\Gbf(w_n))$, $\ord_p(w,K_\Wbf)\leq C_1$, with $C_1$ some constant independent of $w_n,\Tbf_n,K$; in particular, the union $\bigcup_n\Delta(\Tbf_n,K_\Gbf(w_n))$ is finite;
 
 \item for $p\notin\Delta(\Tbf_n,K_\Gbf(w_n))$ and $p$ not dividing $N$ (the constant in \ref{torsion order in the embedded case}(2)), we have $w\in K_{\Wbf,p}$;
 
 \item for $p\notin\Delta(\Tbf,K_\Gbf(w_n))$ and $p$ dividing $N$, we have $w_n^N\in K_{\Wbf,p}$ and $\ord_p(w_n,K_\Wbf)\leq p^{v_p(N)}$ ($v_p$ is the $p$-adic valuation such that $v_p(p)=1$); more precisely we have $N=12$ and for $p=2$ or $p=3$ we have $\ord_p(w_n,K_\Wbf)\leq 4$.
 \end{itemize}
 
  Hence there are only finitely many choices for the classes of $w_n=(u_n,v_n)$ modulo $K_\Wbf$ and  we may take $w_n$'s from a fixed finite subset of $\Wbf(\Qbb)$, which means that the sequence $(M_n)$ is bounded.
 
 The claim on sequences whose Galois orbits are of bounded degrees is clear.\end{proof}
 
 Note that the conclusion of the corollary above is very restrictive: we have started with  a sequence of pure special subvarieties and ended up with finitely many pure special subvarieties. The finitely many choices of the $w_n$'s modulo $K_\Wbf$ have forced the sequence to lie in the union of finitely many maximal pure special subvarieties of $M$.\bigskip
 
 In the mixed case, we propose the following substitute for the estimation of Galois orbits.

 \begin{definition}[test invariants]\label{test invariants} Let $M=M_K(\Pbf,Y)$ be a mixed Shimura variety satisfying \ref{CM splitting fields}, defined at some level $K$ of fine product type with $M^+$ a connected component of $M$ given by $Y^+\subset Y$. For $M'$ a special subvariety defined by a strictly irreducible $\TW$-special subdatum $(\Pbf',Y';Y'^+)=\Wbf'\rtimes(w\Gbf'w^\inv,wX';wX'^+)$, we define the \emph{test invariant} of $M'$ to be $$\tau_M(M'):=\log D(\Tbf)\min_{w'\in\Wbf'(\Qbb)w}\prod_{p\in\Delta(\Tbf,K_\Gbf(w'))}\maxx\{1,I(\Tbf,K_\Gbf(w'))\}$$ where $\Tbf$ is the connected center of $\Gbf'$, $D(\Tbf)$ is the  absolute discriminant of its splitting field. The subdata $\Wbf'\rtimes(w'\Gbf'w'^\inv,w'X';w'X'^+)$ define the same special subvariety $M'$ by \ref{structure of subdata}, and the minimum over $\Wbf'(\Qbb)w$ is justified by \ref{torsion order in the embedded case}.
 
 The definition is independent of the choice of the subdata defining $M'$, as one may verify using \ref{conjugation by gamma}
 \end{definition}

 \begin{theorem}[special subvarieties of bounded test invariants]\label{special subvarieties of bounded test invariants}Let $M=M_K(\Pbf,Y)$ be a mixed Shimura variety satisfying \ref{CM splitting fields}, with $K=K_\Wbf\rtimes K_\Gbf$ and reflex field $E$ be as in \ref{assumption}, and fix $M^+$ the component of $M$ given by $Y^+$ and $\Gamma=\Pbf(\Qbb)_+\cap K$. Let $(M_n)$ be a sequence of special subvarieties in $M^+$ defined by strictly irreducible $(\Tbf_n,w_n)$-special connected subdata $(\Pbf_n,Y_n;Y_n^+)=\Wbf_n\rtimes(w_n\Gbf_nw_n^\inv,w_nX_n;w_nX_n^+) $ such that for some constant $C>0$ we have $$\tau_M(M_n)\leq C,\forall n.$$ Then 
 
 (1) the sequence is $B$-bounded for some finite $B$. 
 
 (2) Assume further the GRH for CM fields.  Then we can find a \cosg $K'\subset K$ of fine product type $K'=K'_\Wbf\rtimes K'_\Gbf$ such that for $M'^+$ the connected component corresponding to $Y^+$ and $\Pbf(\Qbb)_+\cap K'$ in  $M'=M_{K'}(\Pbf,Y)$ we have the estimation of Galois orbits 
  $$\deg_{\pi^*\Lscr_{S'}}\Gal(\Qac/E)\cdot S_n'\geq c_N D_N(\Tbf_n)\min_{w'\in\Wbf_n(\Qbb)w_n}\prod_{p\in\Delta(\Tbf_n,K'_\Gbf(w'))}\maxx\{1,I(\Tbf_n,K'_\Gbf(w')_p)\}$$ where 
 \begin{itemize}
 \item $K'_\Gbf(w)=\{g\in K'_\Gbf:wgw^\inv g^\inv\in K'_\Wbf\}$;
 \item $\Lscr_{S'}$ is the automorphic line bundle on $S'=M_{K_\Gbf'}(\Gbf,X)$;
 \item $S_n'$ is any maximal pure special subvariety in $M_n'$ the special subvariety defined by $(\Pbf_n,Y_n;Y_n^+)$ in  $M'^+$.
 \end{itemize}  
 \end{theorem}
 
 In other words, the test invariants are potentially the ``correct'' numerical bounds for Galois orbits in the mixed case, as long as we restrict our attention to bounded sequences.
 
 \begin{proof}
 
 (1) Write $\pi:M\ra S=M_{K_\Gbf}(\Gbf,X)$ for the natural projection, and $S_n=\pi(M_n)$. Then similar to \ref{pure special subvarieties of bounded Galois orbits}, $S_n$ is $\Tbf_n$-special, with bounded test invariants $$\log D(\Tbf_n)\prod_{p\in\Delta(\Tbf_nK_\Gbf)}\maxx\{1,I(\Tbf_n,K_{\Gbf,p})\}\leq C.$$ Therefore $(S_n)$ is bounded, and the sequence $\log D(\Tbf_n)$ only takes finitely many values. We may replace  $w_n$ by $w_n'\in\Wbf_n(\Qbb)w_n$ which minimizes the following set of values $$\{\prod_{p\in\Delta(\Tbf_n,K_\Gbf(w))}\maxx\{1,I(\Tbf_n,K_\Gbf(w)_p)\}:w\in\Wbf_n(\Qbb)w_n\}$$ without changing the special subvarieties. Then $\prod_{p\in \Delta(\Tbf_n,K_\Gbf(w'_n))}\maxx\{1,I(\Tbf_n,K_\Gbf(w'_n)_p)\}\leq C',\ \forall n$ for some constant $C'>0$. We deduce that the classes of   $w'_n$ modulo $K_\Wbf$ is finite, using \ref{torsion order in the embedded case}. Since we may always translate $w'_n$ by elements in $\Gamma_\Wbf=\Wbf(\Qbb)\cap K_\Wbf$, we may take the $w'_n$'s from a fixed finite subset of $\Wbf(\Qbb)$, hence the claim.
 
 (2) Let $B$  be a finite bounding set for $(M_n)$. 
 
 Write $w=(u,v)$, then $wK_\Wbf w^\inv=\{(u'+2\psi(v,v'),v'):u'\in K_\Ubf, v'\in K_\Vbf\}$. We may shrink $K_\Vbf$ to a  \cosg $K_\Vbf'$ such that $2\psi(v,v')\in K_\Ubf$ for all $v'\in K_\Vbf'$. One may simply take $K_\Vbf'=NK_\Vbf$ for some integer $N>0$, and thus $K'_\Vbf$ is itself stabilized by $K_\Gbf$. Take $K'_\Gbf:=\bigcap_{\TW\in B} K_\Gbf(w)$, and take $K_\Wbf'$ the subgroup of $K_\Wbf$ generated by $K_\Ubf$ and $K'_\Vbf$. Then $K_\Gbf'$ stabilizes $K_\Wbf'$, and we take $K'=K'_\Wbf\rtimes K'_\Gbf$.
 
 Thus $wK' w^\inv=K'$ when $\TW\in B$, and the Hecke translation by $w^\inv$ gives an automorphism of $M'=M_{K'}(\Pbf,Y)$, sending the pure Shimura subvariety $S'(w)=M_{wK'_\Gbf w}(w\Gbf w^\inv, wX)$ to $S'(0)=M_{K'_\Gbf}(\Gbf,X)$, both of which are sections to the natural projection $\pi:M'\ra S'=M_{K'_\Gbf}(\Gbf,X)$.  Let $M_1$ be $\TW$-special subvariety in $M'^+$ is defined by a strictly irreducible subdatum $(\Pbf',Y';Y'^+)=\Wbf'\rtimes(w\Gbf'w^\inv, wX';wX'^+)$. Then $S_1(w):=M_1\cap S'(w)$ is isomorphic to $S_1:=\pi(M_1)$ using the pure section, and we get the bijection between $\Gal(\Qac/E)\cdot S_1$, $\Gal(\Qac/E)\cdot S_1(w)$, and $\Gal(\Qac/E)\cdot M_1$, from which the estimation follows.
 \end{proof}
 
 \begin{remark}[lower bound of Tsimerman]\label{lower bound of Tsimerman}
 
 The factor $c_ND_N(\Tbf)$ in \ref{lower bound involving splitting fields} is  as in \cite{ullmo yafaev}, which comes from a lower bound of the image of the reciprocity map of the form $\rec_\Tbf:\Gal(\Qac/F)\ra\pibar(\Gbb_\mrm^F)\ra\pibar(\Tbf)/K_\Tbf^\maxx$ with $F$ the splitting field of the $\Qbb$-torus $\Tbf$, proved in \cite{yafaev duke}. In \cite{tsimerman bound} Tsimerman proved a  lower bound for a special point $x$ in a Seigel modular variety  in the form $C_g (D(Z_x))^{\delta_g}$, where $Z_x$ is the center of the endomorphism algebra of the CM abelian variety parameterized by $x$, and $D(Z_x)$ is the absolute discriminant of $Z_x$, and it is polynomially equivalent to $D(\Tbf_x)[K^\maxx_{\Tbf_x}:K_{\Tbf_x}]$ with $\Tbf_x$ the Mumford-Tate group of $x$ which is a $\Qbb$-torus. The estimation is established unconditionally for CM abelian varieties of dimension at most 6, and it holds in arbitrary dimension under the GRH for CM fields. In our discussion we had preferred a uniform treatment for all Shimura varieties, including pure Shimura varieties that do not admit finite coverings embedded in Siegel modular varieties. Hence we have followed the formulation and the strategy in \cite{ullmo yafaev}. 
 \end{remark}

   \end{document}